\documentclass[12pt]{article}
\usepackage{mathrsfs}
\usepackage{}
\usepackage{amsfonts}
\usepackage{amssymb} 
\usepackage{amsmath,amscd,graphicx}
\usepackage{amsthm}  
\usepackage{paralist}

\usepackage{color}

 \def\nnm{\notag}

       \newtheorem{lemma}{\bf Lemma}[section]
       \newtheorem{theorem}[lemma]{Theorem}

       \newtheorem{remark}[lemma]{\bf Remark}
       
       \numberwithin{equation}{section}
\newcommand{\bm}{\mathbb}
\begin{document}

\title{{\LARGE \textbf{Non-relativistic limit analysis of  the Chandrasekhar-Thorne relativistic Euler equations with physical vacuum}}
 \footnotetext{\emph{Corresponding Authors*:} hailiang.li.math@gmail.com(Hailiang Li), lasu$_{-}$mai@163.com(La-Su Mai*), pierangelo.marcati@gssi.infn.it(Pierangelo Marcati)}
}
\author{{Hai-Liang Li$^{a}$ ,~~La-Su Mai$^{*~d,c}$,~~ Pierangelo Marcati$~^{b,c},$}\\[2mm]
{ $^{a}$ \it\small  Department of Mathematics and BCMIIS,  Capital Normal University}\\
 {\it\small Beijing 100048, P.R. China}\\
 {$^{b}$ \it\small Dipartimento di Matematica Pura ed Applicata, Universit$\grave{a}$ degli Studi dell' Aquila}\\
{\it \small 67100 L'Aquila, Italy}\\
 {$^{c}$
\it \small Gran Sasso Science Institute- GSSI School for Advanced Studies}\\
 {\it \small 67100 L'Aquila, Italy}\\
 {$^{d}$
\it \small School of Mathematical Science,
Inner Mongolia University}\\
 {\it \small Hohhot 010021, P.R. China}\\}\date{}

\maketitle
\begin{abstract}
Our results provide a first step to make rigorous the formal analysis in terms of $\frac{1}{c^2}$ proposed by Chandrasekhar \cite{Chandra65b}, \cite{Chandra65a}, motivated by the methods of Einstein, Infeld and Hoffmann, see  Thorne \cite{Thorne1}.  We consider the non-relativistic limit for the local smooth solutions to the free boundary value problem of the cylindrically symmetric relativistic Euler equations, when the mass energy density includes the vacuum states at the free boundary. For large enough (rescaled) speed of light $c$ and suitably small time $T,$ we obtain uniform, with respect to $c,$ \lq\lq a priori\rq\rq estimates for the  local smooth solutions.  Moreover, the smooth solutions of the cylindrically symmetric relativistic Euler equations converge to the solutions of the classical compressible Euler equation, at the rate of order $\frac{1}{c^2}$.

\end{abstract}
\tableofcontents
\section{Introduction}\label{sec-1}

Our analysis aims to investigate the post-Newtonian approximation of the Euler fluid system by taking into account the effects of order $\frac{1}{c^2}$ from the General Relativity Einstein{'}s equations.
\par  We consider the isentropic relativistic Euler equations with conservation laws of the baryon numbers and the momentum \cite{Liang,Taub}. Let
\begin{equation}\label{TE}
\begin{aligned}
T^{\mu \lambda}=(e+p)u^{\mu}u^{\lambda}+pg^{\mu \lambda}
\end{aligned}
\end{equation}
be the relativistic energy momentum tensor, where   $\mu, \lambda=0...3$, while $e$ is the relativistic rest energy density, $g^{\mu \lambda}$  is the Minkowski metric tensor, $\text{sign}\ g^{\mu \lambda}=(-,+,+,+)$ and
$\mathbf{u}$ denotes the 4 - vector flow velocity. Then, in an arbitrary Lorentz frame, where
$\mathbf{u}= [\frac{1}{ {\sqrt{1-|{\mathbf{v}}|^{2}/c^{2}}}},\frac{\mathbf{v}}{ {\sqrt{1-|{\mathbf{v}}|^{2}/c^{2}}}} ]$ and where  $\mathbf{v}$ is the spatial velocity (particle speed),
the hydrodynamic equations reduces to the baryon number conservation law  and the continuity equation of the stress - energy tensor
$$(\widetilde{n}\mathbf{u^{\mu}})_{,\mu}=0, \qquad T^{\lambda\mu}_{\quad , \mu}=0,$$ namely, taking into account that classically $e=c^2\widetilde{\rho}$, it follows

\begin{equation}\label{RE}
\left\{\begin{aligned}
&\partial_{t}\left(\frac{\widetilde{n}}{\widetilde{\Theta}}\right)+\nabla\cdot
\left(\frac{\widetilde{n}}{\widetilde{\Theta}}{\mathbf{v}}\right)=0,\\
&\partial_{t}\left(\frac{\widetilde{\rho} c^{2}+p(\widetilde{\rho})}{c^{2}\widetilde{\Theta}^{2}}{\mathbf{v}}\right)
+\nabla\cdot\left(\frac{\widetilde{\rho} c^{2}+p(\widetilde{\rho})}{c^{2}\widetilde{\Theta}^{2}}{\mathbf{v}}\otimes{\mathbf{v}}\right)+\nabla p(\widetilde{\rho})
=0,
\end{aligned}
\right.
\end{equation}
where $ \widetilde{n}~$~and~ $c$ represent  the proper number density of baryons and the speed of light, respectively. The Lorentz factor $\widetilde{\Theta}$ satisfies  $\widetilde{\Theta}=\sqrt{1-|{\mathbf{v}}|^{2}/c^{2}}.$  The pressure $p(\widetilde{\rho})$ is given by
\begin{equation}\label{gamma-law}
p(\widetilde{\rho})=(\widetilde{\rho})^{\gamma} \quad\text{ for } \gamma>1,
\end{equation}  and the mass energy density $\widetilde{\rho}(\widetilde{n})$ is a function of $\widetilde{n}$ satisfying
\begin{equation}\frac{d\widetilde{\rho}}{d\widetilde{n}}=\frac{\widetilde{\rho}+p(\widetilde{\rho})/c^{2}}{\widetilde{n}},\label{first}\end{equation}
which is obtained by the first law of thermo-dynamics in the isentropic case. From \eqref{first} we  can derive the relation for $\widetilde{\rho}$ and $\widetilde{n}$ (see\cite{8}) as
\begin{equation}\widetilde{\rho}=\widetilde{n}(1-\frac{\widetilde{n}^{\gamma-1}}{c^{2}})^{\frac{1}{1-\gamma}},\label{1.1}\end{equation}



\noindent where for simplicity we assume

\begin{equation}\frac{\widetilde{n}^{\gamma-1}}{c^{2}}<1.\label{A1}\end{equation}

These physical models for relativistic Euler equations, have been in the literature since many years.
At the beginning of the 60{'}s various investigations  on the dynamical stability of gaseous masses, in the framework of the general theory of relativity, showed that the theory predicts, already in the post- Newtonian approximation, the phenomena which are qualitatively different from those to be expected on the Newtonian theory, namely gaseous masses are predicted to become dynamically unstable much before the Schwarzschild limit is reached.
Because of these results, in 1965 Chandrasekhar \cite{Chandra65b,Chandra65a}  was motivated to start a systematic investigation of the post-Newtonian effects of general relativity on the behaviors of hydrodynamic systems. It was then necessary to deduce the generalization of the standard Eulerian equations of Newtonian hydrodynamic which could consistently allow for all effects of order $\frac{1}{c^2},$ originating in the exact field equations of Einstein.

By following \cite{Chandra65b}, with the choice of the form of  $T^{\mu \lambda},$ the entire behavior of the system is then determined, in terms of initial conditions, by the Einstein field equations.
\begin{equation}\label{EE}
\begin{aligned}
R^{\mu \lambda}=-\frac{8\pi G}{c^4}(T^{\mu \lambda}-Tg^{\mu \lambda}),
\end{aligned}
\end{equation}
where $R^{\mu \lambda}$ is the Ricci tensor and $T$ is trace of the relativistic energy momentum tensor.
The asymptotic expansions in  $\frac{1}{c^2}$  following Einstein, Infeld, and Hoffmann allow to deduce  the correct hydrodynamic post--Newtonian formulation.
We refer to Thorne \cite{Thorne1} and Novikov and Thorne \cite{Novik} for a more accurate presentation.

\par In the theory of special relativity, the mass of matter is not conserved, whereas the particle numbers are conserved. In system \eqref{RE}, the first equation describes the conservation of baryon number, the second equation is the conservation of momentum equation. By $\eqref{RE}_{2}$ and \eqref{first} we can obtain energy equation similarly as the one dimensional argument made by Pant \cite{Pant-Thesis}:
 \begin{equation}\partial_{t}\left(\frac{\widetilde{\rho} c^{2}+p(\widetilde{\rho})}{c^{2}\widetilde{\Theta}^{2}}-\frac{p'(\widetilde{\rho})}{c^{2}}\right)
+\nabla\cdot\left(\frac{\widetilde{\rho} c^{2}+p(\widetilde{\rho})}{c^{2}\widetilde{\Theta}^{2}}\mathbf{v}\right)
=0.\label{moment}\end{equation}
Formally, in the non-relativistic limit as $c\rightarrow\infty$, the system \eqref{RE} reduces to the classical compressible Euler equations:
\begin{equation}
\left\{\begin{aligned}
&\partial_{t}\widetilde{\rho}+\nabla\cdot(\widetilde{\rho}{\mathbf{v}})=0,\\
&\partial_{t}(\widetilde{\rho}{\mathbf{v}})+\nabla\cdot(\widetilde{\rho}{\mathbf{v}}\otimes{\mathbf{v}})+\nabla p(\widetilde{\rho})=0.
\end{aligned}
\right.\label{Euler}
\end{equation}

It is an interesting and challenging problem ``per se" to analyze the well-posedness and  behaviors of strong/week solutions to the relativistic Euler equation and then to use these results to provide a rigorous justification to the relativistic Euler model. Recently, there have been made important progress on the mathematical theory on these topics. For instance, the global existence of Riemann solutions, BV solutions and related non-relativistic limits have been obtained in \cite{CCF,1116,8,1112,Geng-Li1,Li-Feng-Wang,Min-Ukai,PantCPDE,Pant-Thesis,sm} respectively for either the relativistic system Eq.~\eqref{RE} or the relativistic Euler equations consisting of the momentum equation $\eqref{RE}_{2}$ and energy equation \eqref{moment}. The construction of global weak solutions in  $L^{\infty}$ norm has been made in \cite{1113,14} and the well-posdness and blow-up of smooth solutions have been proved in \cite{M-T1,M-T2,pan-smoller} for Eq.~$\eqref{RE}_{2}$ and \eqref{moment}. There are also several results about the full relativistic Euler equations where the pressure $p$ depends on $\widetilde{\rho}$ and the internal energy $e,$  the interested readers can refer to \cite{Chen97,Geng-Li}.

The free boundary problem for the relativistic Euler equation \eqref{RE} is taken into consideration recently. If the mass energy density $\widetilde{\rho}$ is strictly positive up to the free boundary, i.e., the mass density connects with vacuum through the jump discontinuity,  Trakhinin \cite{TrakhininY} employed the Nash-Moser type iteration scheme to show the well-poseness of local classical solutions to the free boundary value problem of full relativistic Euler equations as mentioned above where the mass energy density was assumed to equal the particle number density. However, as the mass energy density $\widetilde{\rho}$ becomes zero at the moving boundary, i.e., the mass density connects with vacuum continuously,  the relativistic Euler equation \eqref{RE} changes the type to be a degenerate hyperbolic system and the classical theory of Friedrich-Lax-Kato for quasilinear strictly hyperbolic system can not be applied to prove the short time existence of classical solutions. To overcome this difficulties, the first step is to establish the a-priori estimates of classical solutions (supposed to exist) to the free boundary value problem  for the 3D relativistic system \eqref{RE}, which has been made recently by Jang, Lefloch and Masmoudi \cite{JM} in the framework of \cite{Jang1,Jang2} and by Hadzic, Shkoller and Speck \cite{HSJ} in the framework of \cite{D-H-S,Coutand1,Coutand2} respectively. Yet, due to the strong nonlinearity caused by the Lorentz effect $\widetilde{\Theta},$ the existence of short time classical solution to the free boundary problem for the relativistic system \eqref{RE} does not seem to be carried out straightforward as those made in \cite{Coutand1,HSJ,JM,Jang2} for compressible Euler equations \eqref{Euler}. In addition, the asymptotical behaviors of the classical solution such as the non-relativistic limits and the corresponding rates are not justified in the presence of free boundary and vacuum. Indeed, as one can see below in section~\ref{reform},  the appearance of Lorentz effect causes the the relativistic Euler equation \eqref{RE} to take the form of quasilinear system of Euler's type with the additional source terms (refer to \eqref{C-8}--\eqref{C-4} and \eqref{cylind-Euler} for instance). These source terms, which are nonlinear functions of the solutions and vanish in the non-relativistic limits, make it difficult to establish the a-priori estimates and prove the existence of classical solution.

In this paper, we study the well-posedness and non-relativistic limit of local smooth cylindrically symmetric solution to the free boundary value problem for the relativistic Euler equation \eqref{RE} as the mass energy density $\widetilde{\rho}$ connects with the vacuum continuously at the free boundary. We first derive the corresponding equations in cylindrical symmetric coordinates, establish the uniformly a-priori estimates of classical solution to the free boundary problem, and then construct the approximate solutions to show the well-posedness of the classical solution to original problem. Based on the uniform estimates independent of the speed of light $c$, we can obtain the non-relativistic limits as $c\rightarrow\infty,$  in particular, we show that this classical solution of the cylindrical symmetric relativistic Euler equations \eqref{RE} converges to the solution of the classical compressible Euler equation in  the  $C^{0}-$ norm at the rate  $\frac{1}{c^{2}}$ (refer to Theorem~\ref{THM-1} for details).
\par  We briefly state the main difficulties . As mentioned above, the Lorentz factor $\widetilde{\Theta}$ makes the essential difference for the structure of Euler equations. We translate the cylindrical relativistic Euler equation of \eqref{RE} in Euler coordinates into a quasilinear Euler equation with source terms in Lagrangian coordinates.  The part of quasiliear Euler equation converges to the classical cylindrical Euler equation in Lagrangian coordinates and the additional source terms vanish as $c\rightarrow\infty.$ Especially, the source terms involving the main equations of the angular component  and the axial component of velocity are equivalent to the first derivative of the pressure $p(\widetilde{\rho})$ with respect to special variables, which is hard to control for constructing the higher order energy estimates because the degeneracy of the mass energy density $\widetilde{\rho}$ at the boundary. We have to use the special structure of system \eqref{RE} to handle these terms. On the other hand, the strong nonlinear structure of  coefficients (see \eqref{C-24}-\eqref{C-26}) of quasilinear Euler system  makes the a-priori estimates more complicated and tedious, which is also caused by the Lorentz factor $\widetilde{\Theta}$.  For convenience, we summarize the some estimates of coefficients in a lemma (see Lemma \ref{lemma2.1}) in order to simplify our energy estimates.
In the limit system \eqref{Euler}, at the original point $0$ the degeneracy rate $x^{1/2}$ (fractional order) of  the cylindrical symmetric system makes the estimates for the higher order derivatives more complicated than the spherically symmetric system (analyzed in \cite{LuoT}), where the rate is $x.$

\par This paper is arranged as follows. In section \ref{reform}, we describe the our problem and state main results in Lagrangian coordinates. In section \ref{sec-2}, we make some a priori assumptions and computations for the case $\gamma=2,$ which are very important to construct the a priori estimates of solutions. In section \ref{sec-3} and section \ref{sec-4}, we mainly construct the uniformly a-priori estimates of local smooth solutions independent of the speed of light $c$ for large enough $c$ and suitably small $T.$ The energy estimates for the higher order time derivatives are obtained in section \ref{sec-3} and the elliptic type estimates(including the estimates near the original point $x=0$ and the boundary point $x=1$) are established in section \ref{sec-4}.
In section \ref{sec-5}, we prove the existence results by a particular degenerate parabolic regularization to the relativistic Euler system \eqref{C-24}.
 In section \ref{sec-6} and section \ref{sec-7},  we consider the uniqueness and the non-relativistic limits of solution obtained in section \ref{sec-5}, respectively. Finally, we generalize our results to the case $\gamma\neq2$ in section \ref{sec-8}.
 \par \noindent\textbf{~~~Notation and Weighted Sobolev Spaces.}  Let $H^{k}(0,1)$ denote the usual Sobolev spaces with the norm $\|\cdot\|_{k},$ especially ,
$\|\cdot\|_{0}=\|\cdot\|_{L^{2}(0,1)}.$
For real number $l,$ the Sobolev spaces $H^{l}(0,1)$ and the norm $\|\cdot\|_{l}$
 are defined by interpolation. The function space $L^{\infty}(0,1)$ is simplified by $L^{\infty}.$ The notation $C$  denotes the generic positive constants  depending on the (renormalized) light speed $c$  and the notation $M_{0}$ denotes the generic constants independent of $c$, respectively.
 \par Let $d(x)$ be distance function to boundary $\Gamma=\{0,1\}$~as
$d(x)=dist(x,\Gamma)=\min\{x,1-x\} ~\text{for}~x\in \Gamma.$
  For any $a>0$ and nonnegative $b,$ the weighted Sobolev space $H^{a,b}$ is given by
  $H^{a,b}:=\{d^{\frac{a}{2}}:\int^{1}_{0}d^{a}|D^{k}F|^{2}dx<\infty,~0\leq k\leq b\}$
   with the norm
  $\|F\|^{2}_{H^{a,b}}:=\sum^{b}_{k=0}\int^{1}_{0}d^{a}|D^{k}F|dx.$
Then, it holds the following embedding:
$H^{a,b}(0,1)\hookrightarrow H^{b-a/2}(0,1),$
with the estimate
$\|F\|_{b-a/2}\leq C_{0}\|F\|_{H^{a,b}}.\label{WS1}$
In particular, we have
\begin{align}\|F\|^{2}_{0}&\leq C_{0}\int^{1}_{0}d(x)^{2}\left(|F(x)|^{2}+|\partial_{x}F(x)|^{2}\right)dx,\label{WS2}\\
\|F\|^{2}_{1/2}&\leq C_{0}\int^{1}_{0}d(x)\left(|F(x)|^{2}+|\partial_{x}F(x)|^{2}\right)dx.\label{WS3}
\end{align}

\section{Reformulation and main results}\label{reform}
   In this section, we first reformulate the solution from original coordinate \eqref{RE} into cylindrical symmetric form, and derive an expended quasilinear cylindrical Euler equation with source terms parameterized by the $\lambda=\frac{1}{c}.$ Then, we introduce the corresponding Lagrangian coordinate transformation to obtain the cylindrical symmetric Euler Equation in Lagrangian representation and finally we state the main results.

\par  Define the cylindrical symmetric transformation:
\begin{equation}\label{cylind}
\left\{\begin{aligned}
&{\mathbf{v}}=\left(\widetilde{u}\frac{x_{1}}{r}-\widetilde{\upsilon}\frac{x_{2}}{r},\widetilde{u}\frac{x_{2}}{r}+\widetilde{\upsilon}\frac{x_{1}}{r},\widetilde{\omega}\right),
  ~r=\sqrt{x^{2}_{1}+x^{2}_{2}}, \\
&\widetilde{u}=\widetilde{u}(r,t),~\widetilde{\upsilon}=\widetilde{\upsilon}(r,t),~\widetilde{\omega}=\widetilde{\omega}(r,t), \quad  t>0,
\end{aligned}
\right.
\end{equation}
where the scalar functions $\widetilde{u},~\widetilde{\upsilon}$ and $\widetilde{\omega}$ represent the radial component, the angular component and the axial component of the velocity ${\mathbf{v}},$ respectively.
By \eqref{cylind}, we are able to obtain the cylindrical symmetric form for \eqref{RE} after a tedious computation as
\begin{align}
&\partial_{t}\left(\frac{\widetilde{n}}{\widetilde{\Theta}}\right)+\partial_{r}
\left(\frac{\widetilde{n}\widetilde{u}}{\widetilde{\Theta}}\right)+\frac{1}{r}\frac{\widetilde{n}\widetilde{u}}{\widetilde{\Theta}}=0,\label{cylind-R21}\\
&~~~~~~{\mathbb M}{U}={F},\label{cylind-R24}
\end{align}
where the corresponding flow motion variable ${U}$, the quasilinear matrix $\mathbb{M}$ and the source term ${F}$ are defined by
\begin{equation*}
{U}=\left(\begin{array}{cccc}\widetilde{u}_{t}+\widetilde{u}\widetilde{u}_{r}-\frac{\widetilde{\upsilon}^{2}}{r}\\ \widetilde{\upsilon}_{t}+\widetilde{u}\widetilde{\upsilon}_{r}+\frac{\widetilde{u}\widetilde{\upsilon}}{r}\\ \widetilde{\omega}_{t}+\widetilde{u}\widetilde{\omega}_{r}\end{array}\right),
\end{equation*}
\begin{equation*}
{\mathbb M}=\frac{(\widetilde{\rho} c^{2}+p(\widetilde{\rho}))}{c^{2}\widetilde{\Theta}^{2}}
\left(\begin{array}{cccc}
1+\frac{(1-p'(\widetilde{\rho})/c^{2})\widetilde{u}^{2}}{c^{2}\widetilde{\Theta}^{2}}
&-\frac{1-p'(\widetilde{\rho})/c^{2}}{c^{2}\widetilde{\Theta}^{2}}\widetilde{u}\widetilde{\upsilon}&\frac{1-p'(\widetilde{\rho})/c^{2}}{c^{2}\widetilde{\Theta}^{2}}\widetilde{u}\widetilde{\omega}
\\ \frac{1-p'(\widetilde{\rho})/c^{2}}{c^{2}\widetilde{\Theta}^{2}}\widetilde{u}\widetilde{\upsilon}&1-\frac{(1-p'(\widetilde{\rho})/c^{2})\widetilde{\upsilon}^{2}}{c^{2}\widetilde{\Theta}^{2}}
&\frac{1-p'(\widetilde{\rho})/c^{2}}{c^{2}\widetilde{\Theta}^{2}}\widetilde{\upsilon}\widetilde{\omega}\\\frac{1-p'(\widetilde{\rho})/c^{2}}{c^{2}\widetilde{\Theta}^{2}}\widetilde{u}\widetilde{\omega}
&-\frac{1-p'(\widetilde{\rho})/c^{2}}{c^{2}\widetilde{\Theta}^{2}}\widetilde{\upsilon}\widetilde{\omega}&1+\frac{(1-p'(\widetilde{\rho})/c^{2})\widetilde{\omega}^{2}}{c^{2}\widetilde{\Theta}^{2}}
\end{array}\right),
\end{equation*}
\begin{equation*}
{F}=\frac{(\widetilde{\rho} c^{2}+p(\widetilde{\rho}))}{c^{2}\widetilde{\Theta}^{2}}\left(\begin{array}{cccc}\frac{p'(\widetilde{\rho})}{c^{2}}(\widetilde{u}_{r}+\frac{\widetilde{u}}{r})\widetilde{u}-\frac{c^{2}\widetilde{\Theta}^{2}}{(\widetilde{\rho} c^{2}+p(\widetilde{\rho}))}p_{r}(\widetilde{\rho})\\ \frac{p'(\widetilde{\rho})}{c^{2}}(\widetilde{u}_{r}+\frac{\widetilde{u}}{r})\widetilde{\upsilon}\\\frac{p'(\widetilde{\rho})}{c^{2}}(\widetilde{u}_{r}+\frac{\widetilde{u}}{r})\widetilde{\omega}\end{array}\right).
\end{equation*}
The  cylindrical symmetric system \eqref{cylind-R21} and \eqref{cylind-R24} is supplemented with the following free boundary condition and initial data for $(0,R(t))\times[0,T]:$
 \begin{equation}\label{1.2}
\left\{\begin{aligned}
&\widetilde{\rho}>0,~\text{in}~ [0,R(t)),\\
&\rho(R(t),t)=0,\widetilde{u}(0,t)=0,\\
&\frac{dR(t)}{dt}=\widetilde{u}(R(t),t), ~R(0)=1,\\
&(\widetilde{\rho},\widetilde{u},\widetilde{\upsilon},\widetilde{\omega})(x,0)=(\rho_{0},u_{0},\upsilon_{0},\omega_{0}),~\rho_{0}(r)>0 ~\text{in}~ [0,1),\\
&-\infty<\frac{\partial}{\partial r}p'(\rho_{0})<0,~ \text{on}~r=1,
\end{aligned}
\right.
\end{equation}
where  $\eqref{1.2}_{5}$ is called the physical vacuum condition (\cite{Coutand1,Jang1}) , which confirms that $\rho_{0}$
is equivalent to the distance function $d(x)$ of  the  boundary near $x=1,$  and also is very important to obtain the regularities of higher order spatial derivatives of velocity.

\par
In special relativity, the light speed $c>0$ is the maximal speed.  Therefore, we denote ${\mathbf{v}}_{0}=(u_{0},\upsilon_{0},\omega_{0})$ and assume that
\begin{equation}12\|{\mathbf{v}}_{0}\|^{2}_{L^{\infty}(0,1)}<c^{2},\label{A2}\end{equation}
 which implies
\begin{equation}\Theta^{2}_{0}\geq\frac{11}{12}\label{P},\end{equation}
where $\Theta_{0}$ satisfies
\begin{equation}~\Theta_{0}=\sqrt{1-|{\mathbf{v}}_{0}|^{2}/c^{2}}.\label{T}\end{equation}
Similarly, the sound speed
$\sqrt{p'(\rho_{0})}$ should satisfy
\begin{equation}\sqrt{p'(\rho_{0})}<c.\label{L}\end{equation}

\par To simplify the equation \eqref{cylind-R24}, we define
\begin{align}
&\widetilde{\Lambda}_{1}:=1+\frac{(1-p'(\widetilde{\rho})/c^{2})\widetilde{u}^{2}}{c^{2}\widetilde{\Theta}^{2}},\quad
~\widetilde{\Lambda}_{2}:=1-\frac{(1-p'(\widetilde{\rho})/c^{2})\widetilde{\upsilon}^{2}}{c^{2}\widetilde{\Theta}^{2}},\label{C-3}\\
&~\widetilde{\Lambda}_{3}:=1+\frac{(1-p'(\widetilde{\rho})/c^{2})\widetilde{\omega}^{2}}{c^{2}\widetilde{\Theta}^{2}}, \quad
\widetilde{A}_{0}:=\frac{1}{\widetilde{\Lambda}_{2}\widetilde{\Lambda}_{3}+\frac{1}{c^{4}\widetilde{\Theta}^{4}}(1-p'(\widetilde{\rho})/c^{2})\widetilde{\omega}^{2}\widetilde{\upsilon}^{2}}.\label{C-5}
\end{align}
For any smooth solution $(\widetilde{\rho},\widetilde{u},\widetilde{\upsilon},\widetilde{\omega})$ to \eqref{cylind-R21}--\eqref{cylind-R24}  satisfying 
 \begin{equation}
 \widetilde{\Lambda}_{i}>0, i=1,2,3,            \label{positive1}
 \end{equation}
we define the positive matrix ${\bm Q}$ by

\begin{equation}
{\mathbb Q}=\left(\begin{array}{cccc}
1&\widetilde{A}_{0}\frac{1-p'(\widetilde{\rho})/c^{2}}{c^{2}\widetilde{\Theta}^{2}}\widetilde{u}\widetilde{\upsilon}
&-\widetilde{A}_{0}\frac{1-p'(\widetilde{\rho})/c^{2}}{c^{2}\widetilde{\Theta}^{2}}\widetilde{u}\widetilde{\omega}\\
0&\widetilde{A}_{0}\Lambda_{3}&-\widetilde{A}_{0}\frac{1-p'(\widetilde{\rho})/c^{2}}{c^{2}\widetilde{\Theta}^{2}}\widetilde{\upsilon}\widetilde{\omega}\\
0&\widetilde{A}_{0}\frac{1-p'(\widetilde{\rho})/c^{2}}{c^{2}\widetilde{\Theta}^{2}}\widetilde{\upsilon}\widetilde{\omega}&\widetilde{A}_{0}\Lambda_{2}
\end{array}\right),
\end{equation}
and  multiply \eqref{cylind-R24} by the matrix ${\mathbb M}$ on the left to obtain the equations
$
{\mathbb Q}{\mathbb M}U={\mathbb Q} F,
$
which can be written as
\begin{align}
 &\widetilde{a}_{11}\frac{(\widetilde{\rho} c^{2}+p(\widetilde{\rho}))}{c^{2}\widetilde{\Theta}^{2}}
   \left(\widetilde{u}_{t}+\widetilde{u}\widetilde{u_{r}}-\frac{\widetilde{\upsilon}^{2}}{r}\right)
   -\widetilde{a}_{12}\frac{p'(\widetilde{\rho})(\widetilde{\rho} c^{2}+p(\widetilde{\rho}))}{c^{4}\widetilde{\Theta}^{2}}(\widetilde{u}_{r}+\frac{\widetilde{u}}{r})\widetilde{u}+p(\widetilde{\rho})_{r}=0,\label{C-8}
   \\
&\widetilde{\upsilon}_{t}+\widetilde{u}\widetilde{\upsilon}_{r}+\frac{\widetilde{u}\widetilde{\upsilon}}{r}
 -\frac{\widetilde{b}_{11}}{c^{2}}(\widetilde{u}_{r}+\frac{\widetilde{u}}{r})\widetilde{\upsilon}
 +\frac{\widetilde{b}_{12}}{c^{2}}\left(\widetilde{u}_{t}+\widetilde{u}\widetilde{u}_{r}-\frac{\widetilde{\upsilon}^{2}}{r}\right)\widetilde{u}\widetilde{\upsilon}=0,\label{C-6}
 \\
&\widetilde{\omega}_{t}+\widetilde{u}\widetilde{\omega}_{r}-\frac{\widetilde{b}_{11}}{c^{2}}(\widetilde{u}_{r}+\frac{\widetilde{u}}{r})\widetilde{\omega}
 -\frac{\widetilde{b}_{12}}{c^{2}}\left(\widetilde{u}_{t}+\widetilde{u}\widetilde{u}_{r}-\frac{\widetilde{\upsilon}^{2}}{r}\right)\widetilde{u}\widetilde{\omega}=0,\label{C-4}
\end{align}
where
\begin{align}
&\widetilde{a}_{11}(\widetilde{u}^{2},\widetilde{\upsilon}^{2},\widetilde{\omega}^{2},r_{x},\frac{x}{r},\widetilde{\rho})
  :=\widetilde{\Lambda}_{1}-\frac{(1-p'(\widetilde{\rho})/c^{2})^{2}}{c^{4}\widetilde{\Theta}^{4}}\widetilde{A}_{0}\widetilde{\omega}^{2}\widetilde{u}^{2}
     +\frac{(1-p'(\widetilde{\rho})/c^{2})^{2}}{c^{4}\widetilde{\Theta}^{4}}\widetilde{A}_{0}\widetilde{\upsilon}^{2}\widetilde{u}^{2},\label{C-9}\\
&\widetilde{a}_{12}(\widetilde{u}^{2},\widetilde{\upsilon}^{2},\widetilde{\omega}^{2},r_{x},\frac{x}{r},\widetilde{\rho})
  :=1+\frac{(1-p'(\widetilde{\rho})/c^{2})}{c^{2}\widetilde{\Theta}^{2}}\widetilde{A}_{0}\left(\widetilde{\upsilon}^{2}-\widetilde{\omega}^{2}\right)\label{C-10},\\
&\widetilde{b}_{11}(\widetilde{u}^{2},\widetilde{\upsilon}^{2},\widetilde{\omega}^{2},r_{x},\frac{x}{r},\widetilde{\rho}):=\widetilde{A}_{0}p'(\widetilde{\rho}),
 ~\widetilde{b}_{12}(\widetilde{u}^{2},\widetilde{\upsilon}^{2},\widetilde{\omega}^{2},r_{x},\frac{x}{r},\widetilde{\rho})
  :=\widetilde{A}_{0}\frac{1-p'(\widetilde{\rho})/c^{2}}{\widetilde{\Theta}^{2}}, \label{B-101}
\end{align}
where $\Lambda_{i}>0$ for $i=1,2,3,$  are defined by \eqref{C-3}--\eqref{C-5} and satisfy \eqref{positive1}.

\begin{remark}
By the coordinate transform \eqref{cylind} we can also derive the  cylindrical symmetric form for the  Euler equation\eqref{Euler} as
\begin{equation}\label{cylind-Euler}
\left\{\begin{aligned}
&\widetilde{\rho}_{t}+(\widetilde{\rho} \widetilde{u})_{r}+\frac{\widetilde{\rho} \widetilde{u}}{r}=0,\\
&\widetilde{\rho}(\widetilde{u}_{t}+\widetilde{u}\widetilde{u}_{r}-\frac{\widetilde{\upsilon}^{2}}{r})+p_{r}(\widetilde{\rho})=0,\\
&\widetilde{\upsilon}_{t}+\widetilde{u}\widetilde{\upsilon}_{r}+\frac{\widetilde{u}\widetilde{\upsilon}}{r}=0,\\
&\widetilde{\omega}_{t}+\widetilde{u}\widetilde{\omega}=0.
\end{aligned}
\right.
\end{equation}
To compare the different structures of Eq.~\eqref{cylind-Euler} and Eq.~\eqref{cylind-R21} and \eqref{C-8}--\eqref{C-4}, it is obvious that due to the  relativistic effect, the equations \eqref{C-8}--\eqref{C-4} are quasilinear and the  pressure gradient term is also involved in \eqref{C-6}--\eqref{C-4} and affects the flow motion not only in the radial but also in the angular and  axial direction. Although these influences shall vanish in the non-relativistic limit, yet they cause essential difficulties to deal with the existence of smooth solution to the free boundary value problem for Eq.~\eqref{cylind-R21} and \eqref{C-8}--\eqref{C-4}.
\end{remark}

\par We define  the Lagrangian variables $\eta(x,t)$  in order to transform the region $(0,R(t))$ into $(0,1)$
as
\begin{equation}\partial_{t}r(x,t)=u(r(x,t),t),~\text{for}~t>0,~\text{and}~r(x,0)=x,~x\in(0,1).\label{PP}\end{equation}
Denoting
$u(x,t):=\widetilde{u}(r(x,t),t),\upsilon(x,t):=\widetilde{\upsilon}(r(x,t),t),~\omega(x,t):=\widetilde{\omega}(r(x,t),t),$
$\rho(x,t):=\widetilde{\rho}(r(x,t),t),~n(x,t):=\widetilde{n}(r(x,t),t),$
 the baryon numbers conservation equation  \eqref{cylind-R21} is equivalent to
\begin{equation}\frac{n}{\Theta}r_{x}r=\frac{n(\rho_{0})}{\Theta_{0}}x=\frac{\rho_{0}}{(1+\frac{\rho^{\gamma-1}_{0}}{c^{2}})^{\frac{1}{\gamma-1}}\Theta_{0}}x,\label{C-15}\end{equation}
where
$
\Theta
=\sqrt{1-(u^{2}+\upsilon^{2}+\omega^{2})/c^{2}}.
$
By \eqref{C-15}, it holds
\begin{equation}
n=\frac{\rho_{0}}{(1+\frac{\rho^{\gamma-1}_{0}}{c^{2}})^{\frac{1}{\gamma-1}}\Theta_{0}}\frac{1}{r_{x}}\frac{x}{r}\Theta,\label{C-17}
\end{equation}
which together with \eqref{1.1} shows
\begin{equation}
\rho=\frac{\rho_{0}}{(1+\frac{\rho^{\gamma-1}_{0}}{c^{2}})^{\frac{1}{\gamma-1}}\Theta_{0}}\frac{1}{r_{x}}\frac{x}{r}\Theta
\left(1-\frac{1}{c^{2}}\frac{\rho^{\gamma-1}_{0}}{(1+\frac{\rho^{\gamma-1}_{0}}{c^{2}})\Theta^{\gamma-1}_{0}}\frac{1}{r^{\gamma-1}_{x}}(\frac{x}{r})^{\gamma-1}\Theta^{\gamma-1}\right)^{\frac{1}{1-\gamma}}.\label{C-18}\end{equation}
Then, it follows from  \eqref{C-8}-\eqref{C-4} that
\begin{align}
&a^{\gamma}_{11}\frac{x}{r}\rho_{0}(u_{t}-\frac{\upsilon^{2}}{r})\nnm\\
&+\left[\frac{\rho^{\gamma}_{0}}{(1+\frac{\rho^{\gamma-1}_{0}}{c^{2}})^{\frac{\gamma}{\gamma-1}}\Theta^{\gamma}_{0}}
\frac{1}{r^{\gamma}_{x}}(\frac{x}{r})^{\gamma}\Theta^{\gamma}
\left(1-\frac{1}{c^{2}}\frac{\rho^{\gamma}_{0}}{(1+\frac{\rho^{\gamma-1}_{0}}{c^{2}})^{\frac{\gamma}{\gamma-1}}\Theta^{\gamma}_{0}}\frac{1}{r^{\gamma-1}_{x}}(\frac{x}{r})^{\gamma-1}\Theta^{\gamma-1}\right)^{\frac{\gamma}{1-\gamma}}\right]_{x}\nnm\\
&~~~~~~~~~~~~~~~~~~~~~~~~~~~~~~~~~~~~~~~~~~+\frac{a^{\gamma}_{12}}{c^{2}}\frac{\rho^{\gamma}_{0}}{(1+\frac{\rho^{\gamma-1}_{0}}{c^{2}})^{\frac{\gamma}{\gamma-1}}\Theta^{\gamma}_{0}}\frac{1}{r^{\gamma}_{x}}(\frac{x}{r})^{\gamma}(u_{x}+\frac{r_{x}}{r}u)u=0,\label{GB1}\\
&\upsilon_{t}+\frac{u}{r}\upsilon-\frac{b^{\gamma}_{11}}{c^{2}}\frac{\rho^{\gamma-1}_{0}}{(1+\frac{\rho^{\gamma-1}_{0}}{c^{2}})\Theta^{\gamma-1}_{0}}(u_{x}+\frac{u}{r}r_{x})\upsilon+\frac{b^{\gamma}_{12}}{c^{2}}(u_{t}-\frac{\upsilon^{2}}{r})
u~\upsilon=0,\label{C-250}\\
&\omega_{t}-\frac{b^{\gamma}_{11}}{c^{2}}\frac{\rho^{\gamma-1}_{0}}{(1+\frac{\rho^{\gamma-1}_{0}}{c^{2}})\Theta^{\gamma-1}_{0}}(u_{x}+\frac{u}{r}r_{x})\omega+\frac{b^{\gamma}_{12}}{c^{2}}(u_{t}-\frac{\upsilon^{2}}{r})
u~\omega=0,\label{C-260}
\end{align}
where
\begin{align}
&a^{\gamma}_{11}:=a_{11}\frac{\left(1-\frac{1}{c^{2}}\frac{\rho^{\gamma-1}_{0}}{(1+\frac{\rho^{\gamma-1}_{0}}{c^{2}})\Theta^{\gamma-1}_{0}}\frac{1}{r^{\gamma-1}_{x}}(\frac{x}{r})^{\gamma-1}\Theta^{\gamma-1}\right)^{\frac{1}{1-\gamma}}(1+\frac{\rho^{\gamma-1}}{c^{2}})}{\Theta},\label{CCC0}\\
&a^{\gamma}_{12}:=\gamma a_{12}\Theta^{\gamma-2}\left(1-\frac{1}{c^{2}}\frac{\rho^{\gamma-1}_{0}}{(1+\frac{\rho^{\gamma-1}_{0}}{c^{2}})\Theta^{\gamma-1}_{0}}\frac{1}{r^{\gamma-1}_{x}}(\frac{x}{r})^{\gamma-1}\Theta^{\gamma-1}\right)^{\frac{\gamma}{1-\gamma}}
(1+\frac{\rho^{\gamma-1}}{c^{2}}),\nnm\\
&b^{\gamma}_{11}:=\gamma A_{0}(\frac{x}{r}\frac{1}{r_{x}}\Theta)^{\gamma-1} \left(1-\frac{1}{c^{2}}\frac{\rho^{\gamma-1}_{0}}{(1+\frac{\rho^{\gamma-1}_{0}}{c^{2}})\Theta^{\gamma-1}_{0}}\frac{1}{r^{\gamma-1}_{x}}(\frac{x}{r})^{\gamma-1}\Theta^{\gamma-1}\right)^{-1},\nnm\\
&b^{\gamma}_{12}:=A_{0}\frac{1-\gamma\frac{\rho^{\gamma-1}}{c^{2}}}{\Theta}.\label{CCC1}
\end{align}
with $\rho$ given by \eqref{C-18}.
\par Corresponding to the system \eqref{GB1}-\eqref{C-260} in $(0,1)$, the conditions \eqref{1.2} become
\begin{equation}\label{L1.2g}
\left\{\begin{aligned}
&\rho_{0}(x)>0,~x ~\text{in}~ [0,1),~\rho_{0}(1)=0,\\
&-\infty<\frac{\partial}{\partial x}\rho_{0}^{\gamma-1}(1)<0,\\
&u(0,t)=0,~\text{on}~\{x=0\}\times(0,T],\\
&(u,\upsilon,\omega)(x,0)=(u_{0},\upsilon_{0},\omega_{0})~\text{in}~ (0,1),\\
\end{aligned}
\right.
\end{equation}

In this paper  we mainly analyze the case of $\gamma=2.$ For convenience, we denote
$a^{2}_{1j}=a_{1j},~b^{2}_{1j}=b_{1j}(j=1,2),$~$\alpha_{c}(x)=\frac{\rho_{0}}{(1+\frac{\rho^{\gamma-1}_{0}}{c^{2}})^{\frac{1}{\gamma-1}}\Theta_{0}}x
$ and multiply \eqref{GB1} by $r$ to obtain
\begin{align}
&a_{11}\alpha_{c}(x)(u_{t}-\frac{\upsilon^{2}}{r})
+\left(\frac{\alpha^{2}_{c}(x)}{x}\frac{x\Theta^{2}}{rr^{2}_{x}(1-\frac{1}{c^{2}}\frac{\rho_{0}}{r_{x}}\frac{x}{r}\Theta)^{2}}\right)_{x}\nnm\\
&-\frac{\alpha^{2}_{c}(x)}{x^{2}}\frac{x^{2}\Theta^{2}}{r^{2}r_{x}(1-\frac{1}{c^{2}}\frac{\rho_{0}}{r_{x}}\frac{x}{r}\Theta)^{2}}+\frac{xa_{12}}{c^{2}rr^{2}_{x}}\frac{\alpha^{2}_{c}(x)}{x}
(u_{x}+\frac{u}{r}r_{x})u=0,\label{C-24}
\\
&\upsilon_{t}+\frac{u}{r}\upsilon-\frac{b_{11}}{c^{2}}\frac{\alpha_{c}(x)}{x}(u_{x}+\frac{u}{r}r_{x})\upsilon+\frac{b_{12}}{c^{2}}(u_{t}-\frac{\upsilon^{2}}{r})
u~\upsilon=0,\label{C-25}
\\
&\omega_{t}-\frac{b_{11}}{c^{2}}\frac{\alpha_{c}(x)}{x}(u_{x}+\frac{u}{r}r_{x})\omega+\frac{b_{12}}{c^{2}}(u_{t}-\frac{\upsilon^{2}}{r})
u~\omega=0,\label{C-26}
\end{align}
and the conditions \eqref{L1.2g} become
\begin{equation}\label{L1.2}
\left\{\begin{aligned}
&\rho_{0}(x)>0,~x ~\text{in}~ [0,1),~\rho_{0}(1)=0,\\
&-\infty<\frac{\partial}{\partial x}\rho_{0}(1)<0,\\
&u(0,t)=0,~\text{on}~\{x=0\}\times(0,T],\\
&(u,\upsilon,\omega)(x,0)=(u_{0},\upsilon_{0},\omega_{0})~\text{in}~ (0,1).\\
\end{aligned}
\right.
\end{equation}
Formally, as $c\rightarrow \infty,$ we obtain the cylindrical symmetric compressible Euler equations \eqref{cylind-Euler} in Lagrangian coordinates:
\begin{equation}\label{LCE}
\left\{\begin{aligned}
&\alpha_{0}(x)(u_{t}-\frac{\upsilon^{2}}{r})
+\left(\frac{\alpha^{2}_{0}(x)}{x}\frac{x}{rr^{2}_{x}}\right)_{x}-\frac{\alpha^{2}_{0}(x)}{x^{2}}\frac{x^{2}}{r^{2}r_{x}}=0,\\
&\upsilon_{t}+\frac{u}{r}\upsilon=0,\\
&\omega_{t}=0,\\
\end{aligned}
\right.
\end{equation}
where $\alpha_{0}(x)=\rho_{0}(x)x,$ which satisfies  that  $\alpha_{c}(x)\rightarrow \alpha_{0}(x) ~\text{as}~ c\rightarrow\infty,$ This system is the Lagrangian form of the  cylindrical symmetric Euler equation \eqref{cylind-Euler} with $\gamma=2$.

Due to different singularities at the original point $x=0$ and the boundary point $x=1$, we introduce the interior and  the boundary $C^{\infty}$ cut-off functions $\xi(x), \chi(x)$as:
\begin{align} &\xi(x)=1 ~\text{on}~ [0,\delta], ~\xi(x)=0~~ \text{on}~~ [2\delta,1],~|\xi'(x)|\leq\frac{C_{0}}{\delta},\label{cut-1}\\
&\chi(x)=1 ~\text{on} ~[\delta,1], ~\chi(x)=0~ \text{on}~[0,\frac{\delta}{2}],~|\chi'(x)|\leq\frac{C_{0}}{\delta},\label{cut-2}\end{align}
where $C_{0}$ ~and $\delta$ are positive constant and ~$\delta$ ~will be determined later.

Define the energy functional $E(t)$ for the classical solution  $(r,u,\upsilon,\omega)$ as
\begin{equation}E(t):=E(u)+E(\upsilon)+E(\omega)\nnm\end{equation}
with
\begin{align}
E(u):=&\|\frac{\alpha_{0}(x)}{\sqrt{x}}\partial^{4}_{t}\partial_{x}u(t)\|^{2}_{0}+\|\frac{\alpha_{0}(x)}{\sqrt{x}}\frac{\partial^{4}_{t}u}{x}(t)\|^{2}_{0}
+\|\partial^{4}_{x}u(t)\|^{2}_{0}+\|\alpha_{0}(x)u(t)\|^{2}_{3}+\|u(t)\|^{2}_{2}\nnm\\
&+\|\frac{u}{x}(t)\|^{2}_{1}+\|\frac{\alpha_{0}(x)}{\sqrt{x}}\partial^{2}_{t}u(t)\|^{2}_{2}+\|\partial^{2}_{t}u(t)\|^{2}_{1}+\|\frac{\partial^{2}_{t}\partial_{x}u}{\sqrt{x}}(t)\|^{2}_{0}+\|\frac{\partial^{2}_{t}u}{x\sqrt{x}}
(t)\|^{2}_{0}
\nnm\\&+\sum^{1}_{s=0}\left(\|\partial^{2s+1}_{t}u(t)\|^{2}_{\frac{3}{2}-s}
 +\|\frac{\partial^{2s+1}_{t}u}{x}(t)\|^{2}_{1-s}
+\|\sqrt{\alpha_{0}(x)}\partial^{2s+1}_{t}\partial^{2-s}_{x}u(t)\|^{2}_{0}\right)\nnm\\
&+\sum^{1}_{s=0}\left(\|(\frac{\alpha^{3}_{0}(x)}{x})^{\frac{1}{2}}\partial^{2s+1}_{t}\partial^{3-s}_{x}u(t)\|^{2}_{0}
+\|\xi\alpha_{0}(x)\partial^{2s+1}_{t}u(t)\|^{2}_{3-s}
 +\|\xi\partial^{2s+1}_{t}u(t)\|^{2}_{1-s}\right),
\nnm\\ 
E(\upsilon):=&\|\frac{\alpha_{0}(x)}{\sqrt{x}}\partial^{4}_{t}\partial_{x}\upsilon(t)\|^{2}_{0}
+\|\partial^{4}_{t}\upsilon(t)\|^{2}_{0}+\|\partial^{3}_{t}\upsilon(t)\|^{2}_{L^{4}}
+\|\frac{\alpha_{0}(x)}{\sqrt{x}}\partial^{2}_{t}\partial^{2}_{x}\upsilon(\cdot,t)\|^{2}_{0}\nnm\\
&+\|(\partial_{t} \upsilon,\partial^{2}_{t}\upsilon,\frac{\partial_{t}\upsilon}{x},
\frac{\partial^{2}_{t}\upsilon}{x},\frac{\alpha_{0}(x)}{x}\partial_{t}\partial_{x}\upsilon,
\frac{\alpha_{0}(x)}{\sqrt{x}}\partial^{3}_{t}\upsilon)(t)\|^{2}_{L^{\infty}}
\nnm\\&+\|(\upsilon_{x},\partial_{t}\partial_{x}
\upsilon,\sqrt{\alpha_{0}(x)}\partial^{2}_{t}\partial_{x}\upsilon)(t)\|^{2}_{0},\nnm\\
E(\omega):=&\|\frac{\alpha_{0}(x)}{\sqrt{x}}\partial^{4}_{t}\partial_{x}\omega(t)\|^{2}_{0}
+\|\partial^{4}_{t}\omega(t)\|^{2}_{0}+\|\partial^{3}_{t}\omega(t)\|^{2}_{L^{4}}
+\|\frac{\alpha_{0}(x)}{\sqrt{x}}\partial^{2}_{t}\partial^{2}_{x}\omega(t)\|^{2}_{0}\nnm\\
&+\|(\partial_{t} \omega,\partial^{2}_{t}\omega,\frac{\partial_{t}\omega}{x},
 \frac{\partial^{2}_{t}\omega}{x},\frac{\alpha_{0}(x)}{x}\partial_{t}\partial_{x}\omega,
 \frac{\alpha_{0}(x)}{\sqrt{x}}\partial^{3}_{t}\omega)(t)\|^{2}_{L^{\infty}}
 \nnm\\
 &+\|(\omega_{x},\partial_{t}\partial_{x}
\omega,\sqrt{\alpha_{0}(x)}\partial^{2}_{t}\partial_{x}\omega)(t)\|^{2}_{0},\nnm
\end{align}
where  the following compatibility conditions are also assumed to be satisfied for initial data and boundary values
for $1\leq k \leq5$:
\begin{align}
\partial^{k}_{t}u(x,0):=&\partial^{k-1}_{t}\left[\frac{\upsilon^{2}_{0}}{x}
-\frac{1}{\alpha_{c}(x)}\frac{1}{a_{11}(x,0)}\left(\frac{\alpha^{2}_{c}(x)}{x}\frac{\Theta^{2}_{0}}{(1-\frac{1}{c^{2}}\rho_{0}\Theta_{0})^{2}}\right)_{x}\right]\nnm\\
&+\partial^{k-1}_{t}\left[\frac{\alpha^{2}_{c}(x)}{x^{2}}\frac{\Theta^{2}_{0}}{(1-\frac{1}{c^{2}}\rho_{0}\Theta_{0})^{2}}-\frac{a_{12}(x,0)}{c^{2}}\frac{\alpha^{2}_{c}(x)}{x}(\partial_{x}u_{0}+\frac{u_{0}}{x})u_{0}\right],\label{compat-1}\\
\partial^{k}_{t}\upsilon(x,0):=&\partial^{k-1}_{t}\left[-\frac{u_{0}}{x}\upsilon_{0}+\frac{b_{11}(x,0)}{x^{2}}\frac{\alpha_{c}(x)}{x}(\partial_{x}u_{0}+\frac{u_{0}}{x})\upsilon_{0}\right],\nnm\\
&-\partial^{k-1}_{t}\left[\frac{b_{12}(x,0)}{c^{2}}(u_{t}(x,0)-\frac{\upsilon^{2}_{0}}{x})u_{0}\upsilon_{0}\right],\label{compat-2}\end{align}
\begin{align}\partial^{k}_{t}\omega(x,0):=&\partial^{k-1}_{t}\left[\frac{b_{12}(x,0)}{x^{2}}\frac{\alpha_{c}(x)}{x}(\partial_{x}u_{0}+\frac{u_{0}}{x})\omega_{0}-\frac{b_{22}(x,0)}{c^{2}}(u_{t}(x,0)-\frac{\upsilon^{2}_{0}}{x})u_{0}\omega_{0}\right].\label{compat-3}\end{align}
\par

Without the loss of generality, we denote by $\mathscr{P}_m(f)$ the generic polynomial function of $f$ with the order $m>0$. For simplicity, $\mathscr{P}_{0}=\mathscr{P}_{m}(E(0))~$ for any $m>0.$ We also denote  $r=r^{c},u=u^{c},\upsilon=\upsilon^{c},\omega=\omega^{c}$ in order to describe the non-relativistic limit. \par The main result of this paper for the case of $\gamma=2$ is stated as follows.
\begin{theorem}[$\gamma=2$] \label{THM-1}
Assume the initial data $(\rho_{0},u_{0},\upsilon_{0},\omega_{0})\in C^{2}([0,1])$ satisfy \eqref{A2}-\eqref{L},  \eqref{L1.2}, \eqref{compat-1}-\eqref{compat-3} and
\begin{equation}
E(0) <+\infty.\nnm
\end{equation}
Then, there exist two positive constants $c_{0}$  and $ T_{c_{0}}$ such that for any $c\geq c_{0},$ the free boundary problem \eqref{C-24}--\eqref{L1.2} admits a unique classical solution  $(r^{c},u^{c},\upsilon^{c},\omega^{c})$ in $[0,1]\times[0,T_{c_{0}}]$ satisfying \eqref{C-3}--\eqref{positive1} and
\begin{equation}
\underset{t\in[0,T]}{\sup}{E(t)}\leq 2\mathscr{P}_m(E(0))  \label{DRE}
\end{equation}
for some constant integer $m>0.$~Moreover, there exists a  unique classical solution $(r,u,\upsilon,\omega)$ to the free boundary value problem \eqref{LCE} and \eqref{L1.2} so that it holds
\begin{equation}
\|(u^{c}-u,\upsilon^{c}-\upsilon,\omega^{c}-\omega)\|_{C^{0}}
+\|(r^{c}_{x}-r_{x},r^{c}-r)\|_{C^{0}}\leq \mathcal{O}(c^{-2})\label{rate-non}
\end{equation}
as $c\rightarrow\infty$.
\end{theorem}
\begin{remark}\label{LFF}(\textbf{Convergence)}
Our energy functional $E(t)$ satisfies  that
\begin{equation}\label{RE}
\left\{\begin{aligned}
&E(u^{c})~ \text{contains}~\|u^{c}\|_{H^{2}}, ~\|\frac{u^{c}}{x}\|_{H^{1}},~\text{and}~\|u^{c}_{t}\|_{H^{3/2}},\\
&E(\upsilon^{c})~\text{contains}~\|(\upsilon^{c}_{t},\upsilon^{c}_{x},\upsilon^{c}_{xt})\|_{L^{\infty}},\\
&E(\omega^{c})~\text{contains}~\|(\omega^{c}_{t},\omega^{c}_{x},\omega^{c}_{xt})\|_{L^{\infty}}.
\end{aligned}
\right.
\end{equation}
Thus, there exist subsequence $(r^{c},u^{c},\upsilon^{c},\omega^{c})$ converges to $(r,u,\upsilon,\omega)$ which satisfies the problem \eqref{LCE} and \eqref{L1.2} in classical sense, due to \eqref{DRE} and the fundamental theorem of calculous. And, the $C^{0}-$norm is enough to describe the convergence rate.
\end{remark}
\begin{remark}For the original free boundary value problem \eqref{cylind-R21}, \eqref{cylind-R24} and \eqref{1.2} in Euler coordinates, our results can give the cylindrical symmetric solution $(\rho,u,\upsilon,\omega)\in W^{1,\infty}\left(\Omega_{T_{c_{0}}}\right)$ with $\Omega_{T_{c_{0}}}=\{(r,t):0\leq r<R(t), 0\leq t\leq T_{c_{0}}\}.$\end{remark}

\section{Preliminary}\label{sec-2}

In this section, we establish some useful estimates on the coefficients to Eq.~\eqref{C-24}--\eqref{C-26} and other related terms \eqref{C-3}--\eqref{positive1} provided that there exists a classical solution $(r,u,\upsilon,\omega)$ to the boundary value problem ~\eqref{C-24}--\eqref{L1.2}  on $[0,1]\times[0,T]$, which  satisfies the a-priori assumptions~\eqref{1.15} below
\begin{equation}
\underset{t\in[0,T]}{\sup}{
 \|(u_{t},u_{x},\partial^{2}_{t}u,
 \frac{\alpha_{0}(x)}{x}\partial_{t}\partial_{x}u,\frac{u}{x},
 \frac{\partial_{t}u}{x})(t)\|_{L^{\infty}}}\leq K\label{1.15}
\end{equation}
for some constant $K>0$  determined later.
\begin{lemma}
\label{PPP1}
Let $T>0$ and $(r,u,\upsilon,\omega)$ be a classical solution to free boundary problem \eqref{C-24}--\eqref{L1.2} satisfying \eqref{1.15} on $[0,1]\times[0,T]$. Then, there exist a small time $0<\overline{T}\leq T,$ the positive constants $\overline{c}~$ and $~C_{*}$(only depending on $\|(\rho_{0},u_{0},\upsilon_{0},\omega_{0})\|_{L^{\infty}})$  such that for any $t\in(0,\overline{T}]$ and $c\geq \overline{c}$ the following estimates hold
\begin{align}&0<C^{-1}_{*}\leq \Theta^{2},\frac{x}{r}, r_{x},1-\frac{1}{c^{2}}\frac{\rho_{0}}{r_{x}}\frac{x}{r}\Theta,1-\frac{2}{c^{2}}\frac{\rho_{0}}{r_{x}}\frac{x}{r}\frac{\Theta}{(1-\frac{1}{c^{2}}\frac{\rho_{0}}{r_{x}}\frac{x}{r}\Theta)},\Lambda_{i}(i=1,2,3)\leq C_{*}.\label{XXX1}\\
&\|(u,\upsilon,\omega)(t)\|_{L^{\infty}}\leq4\|(u_{0},\upsilon_{0},\omega_{0})\|_{L^{\infty}},~\|(\frac{\upsilon}{x},\frac{\omega}{x})(t)\|_{L^{\infty}}\leq 2\|(\frac{\upsilon_{0}}{x},\frac{\omega_{0}}{x})\|_{L^{\infty}},\label{XXX3}\\
&\|(\partial^{i}_{t}\upsilon,\frac{\partial^{i}_{t}\upsilon}{x},\partial^{i}_{t}\omega,\frac{\partial^{i}_{t}\omega}{x})(t)\|_{L^{\infty}}\leq M_{0}(K^{2}+K+1),~i=1,2,\label{XXX4}\\
&\|(\upsilon_{x},\omega_{x},\alpha_{0}(x)r_{xx},\alpha_{0}(x)u_{xx})(t)\|_{L^{\infty}}\leq M_{0}(K^{2}+K+1),\label{BP-6}
\end{align}
where $M_{0}>0$ is a constant depending only on $\overline{c}~$ and  $\|(\rho_{0},u_{0},\upsilon_{0},\omega_{0})\|_{L^{\infty}}$.
\end{lemma}
\begin{proof}By the fundamental theorem of calculus, we easily obtain that there exist a positive constant $c_{0}$ and small $0<T_{0}\leq T$ such that for any $c\geq c_{0}$ and $t\in(0,T_{0}]$
\begin{equation}\|u(t)\|\leq2\|u_{0}\|,~~\frac{2}{3}\leq\frac{x}{r}\leq2,~\frac{1}{2}\leq r_{x}\leq\frac{3}{2}.\label{A-1.2}\end{equation}
By \eqref{1.15}, integrating \eqref{C-25} over $(0,t)$ with respect to $t$ and taking $L^{\infty}-$norm yield
\begin{equation}\|\upsilon(t)\|_{L^{\infty}}\leq\|\upsilon_{0}\|_{L^{\infty}}+CK\int^{t}_{0}(\|\upsilon(\tau)\|_{L^{\infty}}+\|\upsilon(\tau)\|^{3}_{L^{\infty}})d\tau.\nnm\end{equation}
Using the Gronwall inequality,  it follows
\begin{equation}\|\upsilon(t)\|_{L^{\infty}}\leq 2\|\upsilon_{0}\|_{L^{\infty}}.\label{UUU-1}\end{equation}
Similarly, we can obtain from \eqref{C-26}
\begin{equation}\|\omega(t)\|_{L^{\infty}}\leq 2\|\omega_{0}\|_{L^{\infty}}.\label{UUUX-1}\end{equation}
By \eqref{A-1.2}-\eqref{UUUX-1},we easily obtain that there exist a positive constant $c_{1}\geq c_{0}$ and small $0<T_{1}\leq T_{0}$ such that \eqref{XXX1} holds for any $t\in(0,T_{1}]$. Similarly, we can obtain \eqref{XXX3}-\eqref{XXX4} by dividing \eqref{C-25} and \eqref{C-26} by $x$  and differentiating  with respect to $t,$  respectively.
\par Differentiating \eqref{C-25} with respect to $x,$ we have
\begin{equation}\upsilon_{xt}-\frac{1}{c^{2}}b_{11}\frac{\upsilon}{x} \alpha_{c}(x)u_{xx}-\frac{1}{c^{2}}b_{12}u \upsilon u_{xt}=g_{x},\label{HU-1}\end{equation}
where
\begin{align}g_{x}:=\frac{1}{c^{2}}(b_{11}\frac{\upsilon}{x}\alpha_{c}(x))_{x}u_{x}+\frac{1}{c^{2}}(b_{12}u\upsilon)_{x}u_{t}
+\left[(\frac{b_{11}r_{x}}{c^{2}}\frac{\alpha_{c}(x)}{x}-1+\frac{b_{12}\upsilon^{2}}{c^{2}})\frac{x}{r}\frac{u}{x}\upsilon\right]_{x}.\label{UHC-1}
\end{align}
Integrating \eqref{HU-1} over $(0,t)$ shows
\begin{align}&(\upsilon_{x}-\frac{1}{c^{2}}b_{12}u\upsilon u_{x} )|^{t}_{0}+\int^{t}_{0}(b_{12}u\upsilon)_{t}u_{x}d\tau-\frac{1}{c^{2}}\int^{t}_{0}b_{12}\frac{\upsilon}{x}\alpha_{c}(x)u_{xx}dx-\int^{t}_{0}g_{x}d\tau=0.\label{HU-2}
\end{align}
By differentiating \eqref{C-24} with respect to $t,$ one has
\begin{align}&~a_{11}\alpha_{c}(x)\partial^{2}_{t}u+\partial_{t}a_{11}\alpha_{c}(x)u_{t}-\partial_{t}\left(\alpha_{c}(x)a_{11}\frac{\upsilon^{2}}{r}\right)\nnm\\
&-\left\{J\Theta^{2}\frac{\alpha^{2}_{c}(x)}{x}\left[2\frac{u_{x}}{r^{2}_{x}}+\frac{x}{rr_{x}}\left(1+\frac{1}{c^{2}}\frac{\rho_{0}}{r_{x}}\frac{x}{r}\Theta\right)\frac{u}{x}\right]\right\}_{x}\nnm\\
&-\left[\frac{2J}{c ^{2}r_{x}}\frac{\alpha^{2}_{c}(x)}{x}(uu_{t}+\upsilon\upsilon_{t}
+\omega\omega_{t})\right]_{x}+J\Theta^{2}\frac{x}{r}\frac{\alpha^{2}_{c}(x)}{x^{2}}\left[\left(1+\frac{1}{c^{2}}\frac{\rho_{0}}{r_{x}}\frac{x}{r}\Theta\right)\frac{u_{x}}{r_{x}}
+2\frac{x}{r}\frac{u}{x})\right]\nnm\\
&+\frac{2Jx}{c ^{2}r}\frac{\alpha^{2}_{c}(x)}{x^{2}}(uu_{t}+\upsilon\upsilon_{t}+\omega\omega_{t})
+\frac{1}{c^{2}}\left[a_{12}\frac{x}{rr^{2}_{x}}\frac{\alpha^{2}_{c}(x)}{x}(u_{x}+\frac{xr_{x}}{r}\frac{u}{x})u\right]_{t}=0,\label{first-T}
\end{align}
where
\begin{equation}J:=\frac{x}{rr_{x}}\frac{1}{(1-\frac{1}{c^{2}}\frac{\rho_{0}}{r_{x}}\frac{x}{r}\Theta)^{3}}.\nnm\end{equation}
From\eqref{first-T},
\begin{align}-\frac{1}{c^{2}}\int^{t}_{0}&b_{11}\frac{\upsilon}{x}\alpha_{c}(x)u_{xx}d\tau\geq\frac{1}{c^{4}}b_{11}\frac{\upsilon}{x}\frac{r_{x}}{\Theta^{2}}\alpha_{c}(x)(uu_{x}
+\upsilon\upsilon_{x}+\omega\omega_{x})\nnm\\
&~~~~-M_{0}\left(1+K^{2}+K+K\int^{t}_{0}\|(\upsilon_{x},\omega_{x},\alpha_{c}(x)r_{xx})(\tau)\|_{L^{\infty}}d\tau\right).\label{HU-4}\end{align}
We easily obtain , with the help \eqref{XXX1}-\eqref{XXX3}, that
\begin{align}-\frac{1}{c^{2}}\int^{t}_{0}&\left((b_{12}u\upsilon)_{t}u_{x}+g_{x}\right)d\tau\nnm\\
&\leq M_{0}\left(K^{2}+K+K\int^{t}_{0}\|(\upsilon_{x},\omega_{x},\alpha_{c}(x)r_{xx})(\tau)\|_{L^{\infty}}d\tau\right).\label{HU-5}\end{align}

From \eqref{C-24}, it holds that
\begin{equation}\|\alpha_{0}(x)r_{xx}(t)\|_{L^{\infty}}\leq M_{0}\left(K+1+\|(\upsilon_{x},\omega_{x})(t)\|_{L^{\infty}}\right).\label{TWO-1}\end{equation}
Thus, \eqref{HU-2} together with \eqref{HU-4}-\eqref{TWO-1} gives
\begin{align}&\left[(1+\frac{1}{c^{2}}b_{11}\frac{\alpha_{c}(x)}{x}\frac{\upsilon^{2}}{c^{2}}
\frac{r_{x}}{\Theta^{2}})\upsilon_{x}-\frac{1}{c^{2}}
(b_{12}-\frac{b_{11}}{c^{2}}\frac{\alpha_{c}(x)}{x}\frac{r_{x}}{\Theta^{2}})\upsilon uu_{x}+\frac{1}{c^{4}}b_{11}\frac{\upsilon}{x}\frac{r_{x}}{\Theta^{2}}\alpha_{c}(x)\omega\omega_{x}\right]\nnm\\
&\leq M_{0}\left(1+K^{2}+K+K\int^{t}_{0}\|(\upsilon_{x},\omega_{x},\alpha_{c}(x)r_{xx})(\tau)\|_{L^{\infty}}d\tau\right).\label{HU-6}
\end{align}
By \eqref{C-24}, \eqref{C-26} and  \eqref{first-T}, a similar argument to \eqref{HU-6} yields
\begin{align}&\left[(1+\frac{1}{c^{2}}b_{11}\frac{\alpha_{c}(x)}{x}\frac{\omega^{2}}{c^{2}}
\frac{r_{x}}{\Theta^{2}})\omega_{x}-\frac{1}{c^{2}}
(b_{12}-\frac{b_{11}}{c^{2}}\frac{\alpha_{c}(x)}{x}\frac{r_{x}}{\Theta^{2}})\omega uu_{x}+\frac{1}{c^{4}}b_{11}\frac{\omega}{x}\frac{r_{x}}{\Theta^{2}}\alpha_{c}(x)\upsilon\upsilon_{x}\right]\nnm\\
&\leq M_{0}\left(K^{2}+K+K\int^{t}_{0}\|(\upsilon_{x},\omega_{x},\alpha_{c}(x)r_{xx})(\tau)\|_{L^{\infty}}d\tau\right).\label{HU-7}
\end{align}
which in combination with \eqref{HU-6} gives
\begin{align}&\left(1+\frac{1}{c^{2}}\frac{b_{11}r_{x}}{\Theta^{2}}\frac{\alpha_{c}(x)}{x}
\frac{(\upsilon-\omega)\upsilon}{c^{2}}\right)\upsilon_{x}+
\left(1+\frac{1}{c^{2}}\frac{b_{11}r_{x}}{\Theta^{2}}\frac{\alpha_{c}(x)}{x}
\frac{(\omega-\upsilon)\omega}{c^{2}}\right)\omega_{x}\nnm\\
&\leq M_{0}\left(1+K^{2}+K+K\int^{t}_{0}\|(\upsilon_{x},\omega_{x},\alpha_{c}(x)r_{xx})(\tau)\|_{L^{\infty}}d\tau\right).\label{HU-8}
\end{align}
Then, there exists a positive constant $\overline{c}>c_{1}$ such that for any $c\geq \overline{c}$ and $t\in(0,\overline{T}](0<\overline{T}<T_{1}),$
using the Gronwall inequality to \eqref{HU-8} shows \eqref{XXX4}
with the help of \eqref{first-T} and \eqref{TWO-1}. This completes the proof of Lemma \ref{PPP1}.\end{proof}
\par Due to the complicated structures of coefficients in \eqref{CCC0}-\eqref{CCC1}, we give some important estimates of these coefficients in the following lemma \ref{lemma2.1} in order to simplify our priori estimates. Before the statement, we need the following facts which derive from the higher energy function $E(t).$
Using $H^{1}(0,1)\hookrightarrow L^{\infty}(0,1),~H^{\frac{1}{2}}(0,1)\hookrightarrow L^{p}(0,1)(1<p<\infty)$ and the weighted norm estimate \eqref{WS2}, it holds from $E(t)$ that
\begin{align}&\|\left(\frac{u}{x},u_{x},\alpha_{0}(x)u_{xx},\partial_{t}u,\frac{\partial_{t}u}{x}
,\frac{\alpha_{0}(x)}{x}\partial_{t}\partial_{x}u,\partial^{2}_{t}u,\frac{\alpha_{0}(x)}{\sqrt{x}}\partial^{2}_{t}\partial_{x}u,\frac{\alpha_{0}(x)}{\sqrt{x}}\partial^{3}_{t}u\right)(t)\|_{L^{\infty}}\nnm\\
&+~~\|\left(\frac{\alpha_{0}(x)}{\sqrt{x}}\partial^{3}_{t}\partial_{x}u,\frac{\alpha_{0}(x)}{\sqrt{x}}\partial_{t}\partial^{2}_{x}u\right)(t)\|_{\frac{1}{2}}+\|\left(\partial_{t}\partial_{x}u,\frac{\alpha_{0}(x)}{\sqrt{x}}\partial_{t}\partial^{2}_{x}u,\frac{\alpha_{0}(x)}{\sqrt{x}}\partial^{3}_{t}\partial_{x}u\right)(t)\|_{L^{p}}\nnm\\
&~~~~\leq C\sqrt{E(u)}.\label{E(u)-1}\end{align}
By the fundamental theorem of calculus,
\begin{align}&\|\left(\frac{u}{x},\frac{\alpha_{0}(x)}{x}u_{x},\partial_{t}u,\frac{\alpha_{0}(x)}{\sqrt{x}}\partial_{t}\partial_{x}u,\frac{\alpha_{0}(x)}{\sqrt{x}}\partial^{2}_{t}\partial_{x}u\right)(t)\|_{L^{\infty}}\nnm\\
+&\|\left(u_{x},\frac{\alpha_{0}(x)}{\sqrt{x}}\partial^{2}_{x}u,\partial^{2}_{t}u,\sqrt{\alpha_{0}(x)}\partial_{t}\partial^{2}_{x}u,\partial^{3}_{t}
u\right)(t)\|_{L^{p}}\nnm\\
&\leq\mathscr{P}_{0}+C\int^{t}_{0}\sqrt{E(u)(\tau)}d\tau.\label{E(u)-2}
\end{align}

 \par
 We define the functions $\mathscr{K}_{t,x}^{i,j}(x,t)(i=0,1,...,5,~j=0,1,2)$,where $i$ is the order of time derivatives and $j$ is the order of special derivatives,  as
\begin{align}
\mathscr{K}_{t,x}^{1,0}(x,t)&:=|\partial_{t}u|+|\partial_{t}\upsilon|+|\partial_{t}\omega|+|u_{x}|+|\frac{u}{x}|,\nnm\\
\mathscr{K}_{t,x}^{i,0}(x,t)&:=\sum_{\mbox{\tiny$\begin{array}{c}
\mu+\nu=i\nnm\\
\mu,\nu\geq1\end{array}$}}\mathscr{K}_{t,x}^{\mu,0}(x,t)\mathscr{K}_{t,x}^{\nu,0}(x,t)+|\partial^{i-1}_{t}\partial_{x}u|+|\frac{\partial^{i-1}_{t}u}{x}|\\
&+\sum_{\mbox{\tiny$\begin{array}{c}
\mu+\nu=i\\
\mu,\nu=0\end{array}$}}\left(|\partial^{\mu}_{t}u||\partial^{\nu}_{t}u|+|\partial^{\mu}_{t}\upsilon||\partial^{\nu}_{t}\upsilon|+|\partial^{\mu}_{t}\omega||\partial^{\nu}_{t}\omega|\right)(i\geq2),\nnm\\ \mathscr{K}_{t,x}^{0,1}(x,t)&:=|\partial_{x}u|+|\partial_{x}\upsilon|+|\partial_{x}\omega|+|r_{xx}|+|(\frac{x}{r})_{x}|+|(\rho_{0},u_{0},\upsilon_{0},\omega_{0})_{x}|,\nnm\end{align}
 \begin{align}
 \mathscr{K}_{t,x}^{0,j}(x,t)&=\sum_{\mbox{\tiny$\begin{array}{c}
  \mu+\nu=j\nnm\\
  \mu,\nu\geq1\end{array}$}}\mathscr{K}_{t,x}^{0,\mu}(x,t)\mathscr{K}_{t,x}^{0,\nu}(x,t)+|\partial^{j+1}_{x}r|+|\partial^{j}_{x}(\frac{x}{r})|+|\partial^{j}_{x}(\rho_{0},u_{0},\upsilon_{0},\omega_{0})|\\
  &+\sum_{\mbox{\tiny$\begin{array}{c}
  \mu+\nu=j\\
  \mu,\nu=0\end{array}$}}\left(|\partial^{\mu}_{x}u||\partial^{\nu}_{x}u|+|\partial^{\mu}_{x}\upsilon||\partial^{\nu}_{x}\upsilon|+|\partial^{\mu}_{x}\omega||\partial^{\nu}_{x}\omega|\right)(j\geq2),\nnm\\
 \mathscr{K}_{t,x}^{1,1}(x,t)&:= \mathscr{K}_{t,x}^{1,0}(x,t)\mathscr{K}_{t,x}^{0,1}(x,t)+|u||\partial_{t}\partial_{x}u|+|\upsilon||\partial_{t}\partial_{x}\upsilon|+|\omega||\partial_{t}\partial_{x}\omega|\nnm\\
 &+|u_{xx}|+|(\frac{u}{x})_{x}|,\nnm\\
\mathscr{K}_{t,x}^{i,1}(x,t)&:=  \mathscr{K}_{t,x}^{0,1}(x,t)\mathscr{K}_{t,x}^{i,0}(x,t)+\sum_{\mbox{\tiny$\begin{array}{c}
\mu+\nu=i\nnm\\
\mu,\nu\geq1\end{array}$}}\mathscr{K}_{t,x}^{\mu,1}(x,t)\mathscr{K}_{t,x}^{\nu,0}(x,t)+|\partial^{i-1}_{t}\partial^{2}_{x}u|
\\&+\sum_{\mbox{\tiny$\begin{array}{c}
\mu+\nu=i\\
\mu,\nu=0\end{array}$}}\left(|\partial^{\mu}_{t}\partial_{x}u||\partial^{\nu}_{t}u|+|\partial^{\mu}_{t}\partial_{x}\upsilon||\partial^{\nu}_{t}\upsilon|+|\partial^{\mu}_{t}\partial_{x}\omega||\partial^{\nu}_{t}\omega|\right)\nnm\\
&+|(\frac{\partial^{i-1}_{t}u}{x})_{x}|(i\geq2).\nnm \end{align}

\begin{lemma}\label{lemma2.1}Let
\begin{equation} f(u,\upsilon,\omega,r_{x},\frac{x}{r},\rho_{0},u_{0},\upsilon_{0},\omega_{0})\in C^{\infty}(\underset{i=1}{\prod}{[l_{i},m_{i}]}),(-\infty<l_{i}<m_{i}<\infty, i=1,2,...,6).\nnm\end{equation} Denoting by  $f(x,t)=f(u,\upsilon,\omega,r_{x},\frac{x}{r},\rho_{0},u_{0},\upsilon_{0},\omega_{0}),$ then for any $1<p<\infty,$
\begin{equation}|\partial^{i}_{t}\partial^{j}_{x}f(x,t)|\leq M_{0}\mathscr{K}_{t,x}^{i,j}(x,t)\nnm.\end{equation}
Moreover,
\begin{align}&\|\mathscr{K}_{t,x}^{1,0}(t)\|_{L^{p}}+\|\mathscr{K}_{t,x}^{2,0}(t)\|_{0}+\|\frac{\alpha_{c}(x)}{\sqrt{x}}\mathscr{K}_{t,x}^{2,0}(t)\|_{L^{\infty}}+\|\frac{\alpha_{c}(x)}{\sqrt{x}}\mathscr{K}_{t,x}^{3,0}(t)\|_{L^{p}}\nnm\\
&+\|\alpha_{0}(x)\mathscr{K}_{t,x}^{4,0}(t)\|_{0}\leq M_{0}\left[\mathscr{P}_{0}+\left(\mathscr{P}_{4}(K)+1\right)\int^{t}_{0}\sqrt{E(\tau)}d\tau\right],\label{U-6}\\
&\|\mathscr{K}_{t,x}^{1,0}(t)\|_{L^{\infty}}+\|\mathscr{K}_{t,x}^{2,0}(t)\|_{L^{p}}+\|\frac{\alpha_{0}(x)}{\sqrt{x}}\mathscr{K}_{t,x}^{3,0}(t)\|_{L^{\infty}}+\|\mathscr{K}_{t,x}^{3,0}(t)\|_{0}\nnm\\
&+\|\frac{\alpha_{0}(x)}{\sqrt{x}}\mathscr{K}_{t,x}^{4,0}(t)
\|_{L^{p}}\leq M_{0}\left[\mathscr{P}_{4}(K)+(\mathscr{P}_{4}(K)+1)\sqrt{E(t)}\right],\label{U-7}\\
&\|\sqrt{\alpha_{0}(x)}\mathscr{K}_{t,x}^{5,0}(x,t)(t)\|_{0}\nnm\\
&\leq M_{0}\|(\sqrt{\alpha_{0}(x)}\partial^{5}_{t}u,\frac{\alpha_{0}(x)}{\sqrt{x}}
\partial^{4}_{t}\partial_{x}u,\frac{\alpha_{0}(x)}{\sqrt{x}}\frac{\partial^{4}_{t}u}{x},\sqrt{\alpha_{0}(x)}\partial^{5}_{t}
\upsilon\|_{0}\nnm\\
&+\|\sqrt{\alpha_{0}(x)}\partial^{5}_{t}\omega)(t)\|_{0}+M_{0}\left[\mathscr{P}_{5}(K)+(\mathscr{P}_{5}(K)+1)\sqrt{E(t)}\right],\label{U-8} \\
&\|(\mathscr{K}_{t,x}^{0,1},\alpha_{0}(x)\mathscr{K}_{t,x}^{1,1},\xi\frac{\alpha_{0}}{\sqrt{x}}\mathscr{K}_{t,x}^{2,1})(t)\|_{0}+\|(\xi\mathscr{K}_{t,x}^{0,1},\alpha_{0}(x)\mathscr{K}_{t,x}^{0,1})(t)\|_{L^{\infty}}\nnm\\
&\leq M_{0}\left[\mathscr{P}_{0}+\left(\mathscr{P}_{2}(K)+1\right)\int^{t}_{0}\sqrt{E(\tau)}d\tau\right],\label{U-80}\\
&\|\mathscr{K}_{t,x}^{0,1}(t)\|_{L^{\infty}}+\|\mathscr{K}_{t,x}^{1,1}(t)\|_{0}\leq M_{0}\left[\mathscr{P}_{2}(K)+(\mathscr{P}_{2}(K)+1)\sqrt{E(t)}\right]\label{U-81}. \end{align}

\end{lemma}
\begin{proof}By the chain rules, \eqref{E(u)-1}-\eqref{E(u)-2}, \eqref{WS2}-\eqref{WS3}, the Sobolev embedding and the fundamental theorem of calculus, we can easily obtain \eqref{U-6}-\eqref{U-81}.\end{proof}

\section{Energy Estimates for the case $\gamma=2$}\label{sec-3}
   In this section, we construct the higher order energy estimates of local smooth solutions to  the boundary value problem \eqref{C-24}--\eqref{L1.2}  on $[0,1]\times[0,T]$
    under the assumption \eqref{1.15}. We need to control $\alpha_{0}(x)u_{xxx}$  in order to obtain $\|u\|_{H^{2}}$  according to \eqref{WS2}. Using the equation \eqref{first-T}, the estimate of $\|\partial^{4}_{t}u\|_{0}$ is needed.
    \begin{lemma}\label{PPPVV1}
    Let $T>0$ and $(r,u,\upsilon,\omega)$ be a classical solution to the free boundary problem \eqref{C-24}--\eqref{L1.2} satisfying \eqref{1.15} on $[0,1]\times[0,T]$. Then, there exist a small time $0<\overline{T}_{1}\leq \overline{T}$ and a positive constant $\overline{c}_{1}\geq\overline{c}$ $($only depending on $\|(\rho_{0},u_{0},\upsilon_{0},\omega_{0})\|_{L^{\infty}})$ such that for any $t\in(0,\overline{T}_{1}]$ and $c\geq \overline{c}_{1}$ the following estimates hold
    \begin{align}&\|\left(\sqrt{\alpha(x)}\partial^{5}_{t}u,\frac{\alpha(x)}{\sqrt{x}}\frac{\partial^{4}_{t}
u}{x},\frac{\alpha(x)}{\sqrt{x}}\partial^{4}_{t}\partial_{x}u\right)(\tau)\|^{2}_{0}\nnm\\
&\leq \mathscr{P}_{0}+M_{0}(\mathscr{P}_{10}(K)+1)\int^{t}_{0}(E^{2}(\tau)+E(\tau))d\tau
\nnm\\
&+M_{0}(\mathscr{P}_{10}(K)+1)E(\tau)\int^{t}_{0}E(\tau)d\tau
.\label{mai5} \end{align}\end{lemma}
\begin{proof} Taking $\partial^{k}_{t}$ over \eqref{first-T} gives
 \begin{align}&a_{11}\alpha_{c}(x)\partial^{k+2}_{t}u-\partial^{k+1}_{t}\left(a_{11}\alpha_{c}(x)\frac{\upsilon^{2}}{r}\right)\nnm\\
 &-\left\{J\Theta^{2}\frac{\alpha^{2}_{c}(x)}{x}\left[2\frac{\partial^{k}_{t}\partial_{x}u}{r^{2}_{x}}+\frac{x}{rr_{x}}\left(1+\frac{1}{c^{2}}\frac{\rho_{0}}{r_{x}}\frac{x}{r}\Theta\right)\frac{\partial^{k}_{t}u}{x}\right]\right\}_{x}\nnm\\
 &-\left[\frac{2J}{c ^{2}r_{x}}\frac{\alpha^{2}_{c}(x)}{x}(u\partial^{k+1}_{t}u+\upsilon\partial^{k+1}_{t}\upsilon +\omega\partial^{k+1}_{t}\omega)\right]_{x}\nnm\\
 &+J\Theta^{2}\frac{x}{r}\frac{\alpha^{2}_{c}(x)}{x^{2}}\left[\left(1+\frac{1}{c^{2}}\frac{\rho_{0}}{r_{x}}\frac{x}{r}\Theta\right)\frac{\partial^{k}_{t}\partial_{x}u}{r_{x}}
 +2\frac{x}{r}\frac{\partial^{k}_{t}u}{x}\right]+\sum^{5}_{l=1}\mathcal {J}^{k}_{l}=0,\label{TB-1}\end{align}
where the functions $\mathcal {J}^{k}_{l}$ satisfy
\begin{align} \mathcal {J}^{k}_{1}&:=\partial^{k+1}_{t} a_{11}\alpha_{c}(x)\partial_{t}u+\sum^{k}_{i=1}(C^{i}_{k}
  +C^{i-1}_{k})\partial^{i}_{t}a_{11}\alpha_{c}(x)\partial^{k+2-i}_{t}u,\nnm\\
  \mathcal {J}^{k}_{2}&:=-\sum^{k}_{i=1}C^{i}_{k}\left\{\partial^{i}_{t}(\frac{J\Theta^{2}}{r^{2}_{x}})\frac{\alpha^{2}_{c}(x)}{x}2\partial^{k-i}_{t}\partial_{x}u +\partial^{i}_{t}\left[\frac{J\Theta^{2}x}{rr_{x}}\left(1+\frac{1}{c^{2}}\frac{\rho_{0}}{r_{x}}\frac{x}{r}\Theta\right)\right]\frac{\alpha^{2}_{c}(x)}{x}\frac{\partial^{k-i}_{t}u}{x}\right\}_{x},\nnm\end{align}

  \begin{align}{J}^{k}_{3}&:=-\sum^{k-1}_{i=0}\sum^{i}_{j=0}C^{i}_{k}C^{j}_{i}\left[\partial^{k-i}_{t}(\frac{2J}{c ^{2}r_{x}})\frac{\alpha^{2}_{c}(x)}{x}(\partial^{j}_{t}u\partial^{i-j+1}_{t}u+\partial^{j}_{t}\upsilon\partial^{i-j+1}_{t}\upsilon
  +\partial^{j}_{t}\omega\partial^{i-j+1}_{t}\omega)\right]_{x}\nnm\\ &-\sum^{k}_{j=1}C^{j}_{k}\left[\frac{2J}{c ^{2}r_{x}}\frac{\alpha^{2}_{c}(x)}{x}(\partial^{j}_{t}u\partial^{k-j+1}_{t}u+\partial^{j}_{t}\upsilon\partial^{k-j+1}_{t}\upsilon
  +\partial^{j}_{t}\omega\partial^{k-j+1}_{t}\omega)\right]_{x},\nnm\\
  {J}^{k}_{4}&:=\sum^{k}_{i=1}C^{i}_{k}\left\{\partial^{i}_{t}\left[\frac{J\Theta^{2}x}{rr^{2}_{x}}\left(1+\frac{1}{c^{2}}\frac{\rho_{0}}{r_{x}}\frac{x}{r}\Theta\right)\right]
  \frac{\alpha^{2}_{c}(x)}{x^{2}}\partial^{k-i}_{t}\partial_{x}u+\partial^{i}_{t}(\frac{J\Theta^{2}x^{2}}{r^{2}r_{x}})
  \frac{\alpha^{2}_{c}(x)}{x^{2}}\frac{\partial^{k-i}_{t}u}{x}\right\},\nnm\\
  {J}^{k}_{5}&:=\frac{1}{c^{2}}\sum^{k}_{i=0}\sum ^{i}_{j=0}C^{i}_{k}C^{j}_{i}\partial^{k-i}_{t}(\frac{2Jx}{r})\frac{\alpha^{2}_{c}(x)}{x^{2}}(\partial^{j}_{t}u\partial^{i-j+1}_{t}u+\partial^{j}_{t}\upsilon\partial^{i-j+1}_{t}\upsilon +\partial^{j}_{t}\omega\partial^{i-j+1}_{t}\omega)\nnm\\
  &-\frac{1}{c^{2}}\sum^{k}_{i=0}\sum ^{i}_{j=0}C^{i}_{k+1}C^{j}_{i}\partial^{k+1-i}_{t}(\frac{\widetilde{a}_{12}x}{rr^{2}_{x}})\frac{\alpha^{2}_{c}(x)}{x}\partial^{j}_{t}u\partial^{i-j}_{t}\partial_{x}u
  \nnm\\
  &-\frac{1}{c^{2}}\sum^{k}_{i=0}\sum ^{i}_{j=0}C^{i}_{k+1}C^{j}_{i}2\partial^{k+1-i}_{t}(\frac{\widetilde{a}_{12}x^{2}}{r^{2}r_{x}})\frac{\alpha^{2}_{c}(x)}{x}\partial^{j}_{t}u\frac{\partial^{i-j}_{t}u}{x}.\label{TB-5} \end{align}
  Here and in the sequel, $C^{i}_{k}=\frac{k!}{(k-i)!i!}$ and $C^{j}_{i}$ is similarly defined.
 \par For the case of $k=4,$ multiplying \eqref{TB-1} by $\partial^{5}_{t}u,$ integrating the resulting equation over $(0,t)\times(0,1)$ and using the integration by parts show
 \begin{align}\int^{1}_{0}&a_{11}\alpha_{c}(x)\frac{(\partial^{5}_{t}u)^{2}}{2}dx|^{t}_{0}
 +\int^{1}_{0}J\Theta^{2}\frac{\alpha^{2}_{c}(x)}{x}\left[\frac{(\partial^{4}_{t}\partial_{x}u)^{2}}{r^{2}_{x}} +\frac{x}{rr_{x}}\frac{\partial^{4}_{t}u}{x}\partial^{4}_{t}\partial_{x}u
 +\frac{x^{2}}{r^{2}}\frac{(\partial^{4}_{t}u)^{2}}{x^{2}}\right]dx|^{t}_{0}\nnm\\
 &~~+\mathcal {L}_{1}+\frac{1}{c^{2}}\int^{t}_{0}\int^{1}_{0}\frac{2J}{r_{x}}\frac{\alpha^{2}_{c}(x)}{x}
 (\upsilon\partial^{5}_{t}\upsilon+\omega\partial^{5}_{t}\omega)\partial^{5}_{t}\partial_{x}udxd\tau\nnm\\
 &~~~~~~~-\int^{t}_{0}\int^{1}_{0}\partial^{5}_{t}\left(a_{11}\alpha_{c}(x)\frac{\upsilon^{2}}{r}\right)
 \partial^{5}_{t}udxd\tau+\sum^{5}_{l=1}\int^{t}_{0}\int^{1}_{0}\mathcal {J}^{4}_{l}\partial^{5}_{t}udxd\tau=0,\label{TF-1}\end{align}
 where $\mathcal {L}_{1}$ is the remainder term after the integration by parts satifying
 \begin{equation}|\mathcal {L}_{1}|\leq M_{0}(\mathscr{P}(K)+1)\int^{t}_{0}\|(\sqrt{\alpha_{c}(x)}\partial^{5}_{t}u,\frac{\alpha_{c}(x)}{\sqrt{x}}\frac{\partial^{4}_{t}u}{x},\frac{\alpha_{c}(x)}{\sqrt{x}}\partial^{4}_{t}\partial_{x}u)(\tau)\|^{2}_{0}d\tau.\label{TF-3}\end{equation}
 \par We deal with the forth term of \eqref{TF-1}, which can not be controlled by the integration by parts. This difficulty is caused by the relativistic effect of Lorentz factor $\Theta$ and will vanish for the case of the compressible Euler equations \eqref{LCE}. By taking $\partial^{k}_{t}$ over \eqref{C-25}, it follows that \begin{equation}\partial^{k+1}_{t}\upsilon-\frac{1}{c^{2}}b_{11}\frac{\upsilon}{x}\alpha_{c}(x)\partial^{k}_{t}\partial_{x}u-\frac{1}{c^{2}}b_{12}u\upsilon\partial^{k+1}_{t}u=g_{k},\label{SU-1}\end{equation}
 where
 \begin{align}g_{k}&:=\frac{1}{c^{2}}\sum^{k}_{i=1}C^{i}_{k}\left[\partial^{i}_{t}(b_{11}\upsilon)\frac{\alpha_{c}(x)}{x}\partial^{k-i}_{t}\partial_{x}u+\partial^{i}_{t}(b_{12}u\upsilon)\partial^{k+1-i}u\right]\nnm\\
 &+\partial^{k}_{t}\left[\left(\frac{1}{c^{2}}\frac{b_{11}r_{x}\alpha_{c}(x)}{x}-1+\frac{1}{c^{2}}b_{12}\upsilon^{2}\right)\frac{x}{r}\frac{u}{x}\upsilon\right].\label{SU-2}
 \end{align}

\par For $k=4,$ solve $\partial^{5}_{t}\upsilon$ from \eqref{SU-1} and get
\begin{align}&-\frac{1}{c^{2}}\int^{t}_{0}\int^{1}_{0}\frac{2J}{r_{x}}\frac{\alpha^{2}_{c}(x)}{x}
\upsilon\partial^{5}_{t}\upsilon\partial^{5}_{t}\partial_{x}udxd\tau\nnm\\
&=-\frac{1}{c^{2}}\int^{1}_{0}\frac{Jb_{11}}{r_{x}}\frac{\upsilon^{2}}{c^{2}}\frac{\alpha^{3}(x)}{x^{2}}
(\partial^{4}_{t}\partial_{x}u)^{2}dx|^{t}_{0}+\frac{1}{c^{2}}\int^{t}_{0}
\int^{1}_{0}\partial_{t}(\frac{Jb_{11}}{r_{x}}\frac{\upsilon^{2}}{c^{2}}\frac{\alpha^{3}(x)}{x^{2}})
(\partial^{4}_{t}\partial_{x}u)^{2}dxd\tau\nnm\\
&~~+\frac{1}{c^{4}}\int^{t}_{0}\int^{1}_{0}\partial_{x}(\frac{Jb_{12}}{r_{x}}u\upsilon^{2}\frac{\alpha^{2}_{c}(x)}{x})
(\partial^{5}_{t}u)^{2}dxd\tau-\frac{1}{c^{2}}\int^{t}_{0}\int^{1}_{0}\frac{2J}{r_{x}}\frac{\upsilon}{x}\alpha^{2}_{c}(x)
g_{4}\partial^{5}_{t}\partial_{x}udxd\tau.\label{SU-3}
\end{align}
It is easy to see that the second and third terms on the right side of \eqref{SU-3} can be bounded by $M_{0}(\mathscr{P}(K)+1)\int^{t}_{0}\|\left(\sqrt{\alpha_{c}(x)}\partial^{5}_{t}u,\frac{\alpha_{c}(x)}{\sqrt{x}}\partial^{4}_{t}\partial_{x}u\right)(\tau)\|^{2}_{0}d\tau.$
\par  We integrate by parts with respect to $t$ and obtain
\begin{align}&-\frac{1}{c^{2}}\int^{t}_{0}\int^{1}_{0}\frac{2J}{r_{x}}\frac{\upsilon}{x}\alpha^{2}_{c}(x)
g_{4}\partial^{5}_{t}\partial_{x}udxd\tau\nnm\\
&=-\frac{1}{c^{2}}\int^{1}_{0}\frac{2J}{r_{x}}\frac{\upsilon}{x}\alpha^{2}_{c}(x)
g_{4}\partial^{4}_{t}\partial_{x}udx|^{t}_{0}+\frac{1}{c^{2}}\int^{t}_{0}\int^{1}_{0}\partial_{t}(\frac{2J}{r_{x}}\frac{\upsilon}{x}
g_{4})\alpha^{2}_{c}(x)\partial^{4}_{t}\partial_{x}udxd\tau.\label{SU-4}
\end{align}
From \eqref{SU-2},
\begin{align}&\frac{1}{c^{2}}\int^{1}_{0}\frac{2J}{r_{x}}\frac{\upsilon}{x}\alpha^{2}_{c}(x)
g_{4}\partial^{4}_{t}\partial_{x}udx\nnm\\
&=\frac{1}{c^{4}}\sum^{4}_{i=1}C^{i}_{4}\int^{1}_{0}\frac{2J}{r_{x}}\frac{\upsilon}{x}
\left[\partial^{i}_{t}(b_{11}\upsilon)\frac{\alpha_{c}(x)}{x}\partial^{4-i}_{t}\partial_{x}
u+\partial^{i}_{t}(b_{12}u\upsilon)\partial^{5-i}_{t}u\right]\alpha^{2}_{c}(x)\partial^{4}_{t}\partial_{x}udx\nnm\\
&-\int^{1}_{0}\frac{2J}{r_{x}}\frac{\upsilon}{x}\partial^{4}_{t}\left[\left(\frac{1}{c^{2}}
\frac{b_{11}r_{x}\alpha_{c}(x)}{x}-1+\frac{1}{c^{2}}b_{12}\upsilon^{2}\right)
\frac{x}{r}\frac{u}{x}\upsilon\right]\alpha^{2}_{c}(x)\partial^{4}_{t}\partial_{x}udx\nnm\\
&=I_{0}+I_{1}.\label{SU-5}
\end{align}
\par We denote $I_{0}=\sum^{4}_{i=1}I^{i}_{0}$ and estimate it as follows. Using \eqref{U-7}, the H$\ddot{o}$lder inequality and the Cauchy inequality, the fundamental theorem of calculus imply for the  arbitrary  positive constant $\varepsilon,$
\begin{align}I^{1}_{0}&\leq M_{0}\|\mathscr{K}^{1,0}_{t,x}(t)\|_{L^{4}}\left(\|\frac{\alpha_{c}(x)}{\sqrt{x}}\partial^{3}_{t}\partial_{x}u(0)\|_{L^{4}}
+\|\frac{\alpha_{c}(x)}{x}\partial^{4}_{t}u(0)\|_{L^{4}}\right)\|\frac{\alpha_{c}(x)}{\sqrt{x}}\partial^{4}_{t}\partial_{x}u(t)\|_{0}\nnm\\
&+M_{0}K\int^{t}_{0}\left(\|\frac{\alpha_{c}(x)}{\sqrt{x}}\partial^{4}_{t}\partial_{x}u(\tau)\|_{0}
+\|\sqrt{\alpha_{c}(x)}\partial^{5}_{t}u(\tau)\|_{0}\right)d\tau\|\frac{\alpha_{c}(x)}{\sqrt{x}}\partial^{4}_{t}\partial_{x}u(t)\|_{0}\nnm\\
&\leq C(\varepsilon)M_{0}\mathscr{P}(K)\int^{t}_{0}\|(\sqrt{\alpha_{c}(x)}\partial^{5}_{t}u,\frac{\alpha_{c}(x)}{\sqrt{x}}\partial^{4}_{t}\partial_{x}u)(\tau)\|^{2}_{0}d\tau\nnm\\
&+\varepsilon \|\frac{\alpha_{c}(x)}{\sqrt{x}}\partial^{4}_{t}\partial_{x}u(t)\|_{0}+C(\varepsilon)M_{0}\left[\mathscr{P}_{0}+(\mathscr{P}(K)+1)\int^{t}_{0}E(\tau)d\tau \right].\label{SU-6}
\end{align}
From now on, we repeatedly use the fact that for some positive constant $M_{0},$~~$\frac{1}{M_{0}}\alpha_{0}(x)\leq\alpha_{c}(x)\leq M_{0}\alpha_{0}(x).$
Similarly, using \eqref{U-6}, the fundamental theorem of calculous and the $L^{2}-L^{4}-L^{4}$ type H$\ddot{o}$lder's inequality, we can estimate $I^{i}_{0}(i=2,3,4)$
and obtain
\begin{align}|I_{0}|\leq&C(\varepsilon)M_{0}\left[\mathscr{P}_{0}+(\mathscr{P}(K)+1)\int^{t}_{0}E(\tau)d\tau+(\mathscr{P}(K)+1)E(\tau)\int^{t}_{0}E(\tau)d\tau
\right]\nnm\\&+ C(\varepsilon)M_{0}\mathscr{P}(K)\int^{t}_{0}\|(\sqrt{\alpha_{c}(x)}\partial^{5}_{t}u,\frac{\alpha_{c}(x)}{\sqrt{x}}\partial^{4}_{t}\partial_{x}u)(\tau)\|^{2}_{0}d\tau\nnm\\
&~~+\varepsilon \|\frac{\alpha_{c}(x)}{\sqrt{x}}\partial^{4}_{t}\partial_{x}u(t)\|_{0}.\label{SU-7}\end{align}
By the similar way, we also prove that $I_{1}$ has the bound as same as $I_{0}$ and obtain from \eqref{SU-5}
\begin{align}&\frac{1}{c^{2}}\int^{1}_{0}\frac{2J}{r_{x}}\frac{\upsilon}{x}\alpha^{2}_{c}(x)
g_{4}\partial^{4}_{t}\partial_{x}udx\nnm\\
&\leq C(\varepsilon)M_{0}\left[\mathscr{P}_{0}+(\mathscr{P}(K)+1)\int^{t}_{0}E(\tau)d\tau+(\mathscr{P}(K)+1)E(\tau)\int^{t}_{0}E(\tau)d\tau
\right]\nnm\\
&+ C(\varepsilon)M_{0}\mathscr{P}(K)\int^{t}_{0}\|(\sqrt{\alpha_{c}(x)}\partial^{5}_{t}u,\frac{\alpha_{c}(x)}{\sqrt{x}}\partial^{4}_{t}\partial_{x}u)(\tau)\|^{2}_{0}d\tau\nnm\\
&~~+\varepsilon \|\frac{\alpha_{c}(x)}{\sqrt{x}}\partial^{4}_{t}\partial_{x}u(t)\|_{0}.\label{SU-8}\end{align}
\par We also control the second term on the right side of \eqref{SU-4} and have

\begin{align}&-\frac{1}{c^{2}}\int^{t}_{0}\int^{1}_{0}\frac{2J}{r_{x}}\frac{\alpha^{2}_{c}(x)}{x}
\upsilon\partial^{5}_{t}\upsilon\partial^{5}_{t}\partial_{x}udxd\tau\nnm\\
&\geq-\frac{1}{c^{2}}\int^{1}_{0}\frac{Jb_{11}}{r_{x}}\frac{\upsilon^{2}}{c^{2}}\frac{\alpha^{3}(x)}{x^{2}}
(\partial^{4}_{t}\partial_{x}u)^{2}dx|^{t}_{0}-\mathcal {L}_{2},\label{SU-12}\end{align}
where
\begin{align}
\mathcal {L}_{2}&\leq C(\varepsilon)M_{0}\left[\mathscr{P}_{0}+(\mathscr{P}(K)+1)\int^{t}_{0}(E^{2}(t)+E(\tau))d\tau+(\mathscr{P}(K)+1)E(\tau)\int^{t}_{0}E(\tau)d\tau
\right]\nnm\\
& +C(\varepsilon)M_{0}\mathscr{P}(K)\int^{t}_{0}\|(\sqrt{\alpha_{c}(x)}\partial^{5}_{t}u,\frac{\alpha_{c}(x)}{\sqrt{x}}\partial^{4}_{t}\partial_{x}u)(\tau)\|^{2}_{0}d\tau+\varepsilon \|(\frac{\alpha_{c}(x)}{\sqrt{x}}\frac{\partial^{4}_{t}u}{x},\frac{\alpha_{c}(x)}{\sqrt{x}}\partial^{4}_{t}\partial_{x}u)(t)\|_{0}\nnm\\
&+M_{0}(\mathscr{P}(K)+1)\int^{t}_{0}\|(\sqrt{\alpha_{c}(x)}\partial^{5}_{t}u,\frac{\alpha_{c}(x)}{\sqrt{x}}\frac{\partial^{4}_{t}
u}{x},\frac{\alpha_{c}(x)}{\sqrt{x}}\partial^{4}_{t}\partial_{x}u)(\tau)\|^{2}_{0}d\tau\nnm\\
&+M_{0}(\mathscr{P}(K)+1)\int^{t}_{0}\|(\sqrt{\alpha_{c}(x)}\partial^{5}_{t}\upsilon,\sqrt{\alpha_{c}(x)}\partial^{5}_{t}\omega)(\tau)\|^{2}_{0}d\tau
\label{LL22}.
\end{align}

Using \eqref{C-26}, an argument similar to \eqref{SU-12} gives
\begin{align}&-\frac{1}{c^{2}}\int^{t}_{0}\int^{1}_{0}\frac{2J}{r_{x}}\frac{\alpha^{2}_{c}(x)}{x}
\omega\partial^{5}_{t}\omega\partial^{5}_{t}\partial_{x}udxd\tau\nnm\\
&\geq-\frac{1}{c^{2}}\int^{1}_{0}\frac{Jb_{21}}{r_{x}}\frac{\omega^{2}}{c^{2}}\frac{\alpha^{3}_{c}(x)}{x^{2}}
(\partial^{4}_{t}\partial_{x}u)^{2}dx|^{t}_{0}-\mathcal {L}_{2},\nnm
\nnm\end{align}
which in combination with \eqref{SU-12} completes the estimate for the forth term of \eqref{TF-1}.
\par Using the chain rule, an argument similar to the proof of Lemma \ref{lemma2.1} yields
\begin{align}&\|\sqrt{\alpha_{c}(x)}\partial^{5}_{t}a_{11}(t)\|_{0}\nnm\\&\leq M_{0}\|(\sqrt{\alpha_{c}(x)}\partial^{5}_{t}u,\frac{\alpha_{c}(x)}{\sqrt{x}}\frac{\partial^{4}_{t}u}{x},\frac{\alpha_{c}(x)}{\sqrt{x}}\partial^{4}_{t}\partial_{x}u,\sqrt{\alpha_{c}(x)}\partial^{5}_{t}\upsilon,\sqrt{\alpha_{c}(x)}\partial^{5}_{t}\omega)(t)\|^{2}_{0}\nnm\\
&+M_{0}\mathscr{P}(K)\sqrt{E(t)}.\nnm\end{align}
A straightforward computation gives
\begin{align}
&|\partial^{5}_{t}(\frac{x}{r})|\leq M_{0}\left[K^{5}+K^{4}+K^{3}+K^{2}+(K^{3}+K^{2})|\frac{\partial_{t}u}{x}|\right]\nnm\\
&~~~~~~~~+M_{0}\left[(K^{2}+K)|\frac{\partial^{2}_{t}u}{x}|+K|\frac{\partial^{3}_{t}u}{x}|+|\frac{\partial^{4}_{t}u}{x}|\right].\nnm
\end{align}
Then, using \eqref{1.15}-\eqref{XXX4}, Lemma \ref{lemma2.1}, the H$\ddot{o}$lder inequality and the Cauchy inequality, the fundamental theorem of calculus, we have

\begin{align}&-\int^{t}_{0}\int^{1}_{0}\partial^{5}_{t}\left(a_{11}\alpha_{c}(x)\frac{\upsilon^{2}}{r}\right)
\partial^{5}_{t}udxd\tau\nnm\\
&\leq M_{0}(\mathscr{P}(K)+1)\int^{t}_{0}\|(\sqrt{\alpha_{c}(x)}\partial^{5}_{t}u,\frac{\alpha_{c}(x)}{\sqrt{x}}\frac{\partial^{4}_{t}u}{x},\frac{\alpha_{c}(x)}{\sqrt{x}}\partial^{4}_{t}\partial_{x}u)(\tau)\|^{2}_{0}d\tau\nnm\\
&+M_{0}(\mathscr{P}(K)+1)\int^{t}_{0}\|(\sqrt{\alpha_{c}(x)}\partial^{5}_{t}\upsilon,\sqrt{\alpha_{c}(x)}\partial^{5}_{t}\omega)(\tau)\|^{2}_{0}d\tau\nnm\\
&+M_{0}\int^{t}_{0}E(\tau)d\tau\label{TF1-1}.
\end{align}
\par We estimate the last term of \eqref{TB-1}. Similar to \eqref{TF1-1},
$\int^{t}_{0}\int^{1}_{0}(\mathcal {J}^{4}_{1}+\mathcal {J}^{4}_{5})\partial^{5}_{t}udxd\tau$ is also bounded by the right side of \eqref{TF1-1}. For $\mathcal {J}^{4}_{2}$, using the integration by parts with respect to $x$ and  with respect to $t$ show
\begin{align}&\int^{t}_{0}\int^{1}_{0}\mathcal {J}^{4}_{2}\partial^{5}_{t}u(t)dxd\tau\nnm\\
&=-\sum^{4}_{i=1}C^{i}_{4}\int^{1}_{0}\left[\partial^{i}_{t}(\frac{J\Theta^{2}}{r^{2}_{x}})\frac{\alpha^{2}_{c}(x)}{x}2\partial^{4-i}\partial_{x}u
+\partial^{i}_{t}(\frac{J\Theta^{2}x}{rr^{2}_{x}})\frac{\alpha^{2}_{c}(x)}{x}\frac{\partial^{4-i}_{t}u}{x}\right]\partial^{4}_{t}\partial_{x}u(t)dx|^{t}_{0}\nnm\\
&-\sum^{4}_{i=1}C^{i}_{4}\int^{t}_{0}\int^{1}_{0}\partial_{t}\left[\partial^{i}_{t}(\frac{J\Theta^{2}}{r^{2}_{x}})\frac{\alpha^{2}_{c}(x)}{x}2\partial^{4-i}\partial_{x}u
+\partial^{i}_{t}(\frac{J\Theta^{2}x}{rr^{2}_{x}})\frac{\alpha^{2}_{c}(x)}{x}\frac{\partial^{4-i}_{t}u}{x}\right]\partial^{4}_{t}\partial_{x}u(t)dxd\tau\nnm\\
&=I_{ 4}|^{t}_{0}+I_{5}.\label{TF1-4}\end{align}
Then, the similar analysis to ${J}^{4}_{2}$ shows that
$\int^{t}_{0}\int^{1}_{0}(\mathcal {J}^{4}_{2}+\mathcal {J}^{4}_{3}+\mathcal {J}^{4}_{4})\partial^{5}_{t}u(t)dxd\tau$ can be bounded by the right hands side of \eqref{LL22}.

\par On the other hand, it follows from \eqref{C-25} and \eqref{C-26} that
\begin{align}
\|\sqrt{\alpha_{c}(x)}(\partial^{5}_{t}\upsilon,\partial^{5}_{t}\omega)(t)\|_{0}&\leq M_{0}\left[\|(\sqrt{\alpha_{c}(x)}\partial^{5}_{t}u,\frac{\alpha_{c}(x)}{\sqrt{x}}\frac{\partial^{4}_{t}
u}{x},\frac{\alpha_{c}(x)}{\sqrt{x}}\partial^{4}_{t}\partial_{x}u)(t)\|\right] \nnm\\
&+M_{0}(\mathscr{P}(K)+1)\sqrt{E(t)}.\label{TF1-5}
\end{align}

\par Using \eqref{1.15},\eqref{XXX1} and \eqref{CCC1} ,  there exist a positive constant \begin{equation}\overline{c}_{1}=\overline{c}_{1}(\|\rho_{0}\|_{L^{\infty}},\|{\mathbf{v}}_{0}\|^{2}_{L^{\infty}}\})>\overline{c}\nnm\end{equation} such that for any $c\geq \overline{c}_{1},$
\begin{equation}\frac{1}{4}-\frac{1}{c^{2}}\frac{r_{x}}{\widetilde{\theta}^{2}}\frac{\alpha_{c}(x)}{x}\left(b_{11}\frac{\upsilon^{2}}{x^{2}}+b_{12}\frac{\omega^{2}}{x^{2}}\right)\geq0.\label{mai4}\end{equation}
Therefore, \eqref{TF-1} together with  \eqref{TF-3}, \eqref{SU-12},  \eqref{TF1-1}, \eqref{TF1-5} and \eqref{mai4} yields for any $c\geq \overline{c}_{1}$

\begin{align}\int^{1}_{0}&a_{11}\alpha_{c}(x)\frac{(\partial^{5}_{t}u)^{2}}{2}dx|^{t}_{0}
+\int^{1}_{0}J\Theta^{2}\frac{\alpha^{2}_{c}(x)}{x}\left[\frac{1}{8}\frac{(\partial^{4}_{t}\partial_{x}u)^{2}}{r^{2}_{x}}
+\frac{1}{4}\frac{x^{2}}{r^{2}}\frac{(\partial^{4}_{t}u)^{2}}{x^{2}}\right]dx|^{t}_{0}\nnm\\
&+\int^{1}_{0}J\Theta^{2}\frac{\alpha^{2}_{c}(x)}{x}\left(\frac{\partial^{4}_{t}\partial_{x}u}{r_{x}}
+\frac{x}{r}\frac{\partial^{4}_{t}u}{x}\right)^{2}dx|^{t}_{0}\nnm\\
&\leq M_{0}\left[\mathscr{P}_{0}+(\mathscr{P}(K)+1)\int^{t}_{0}(E^{2}(\tau)+E(\tau))d\tau+(\mathscr{P}(K)+1)E(t)\int^{t}_{0}E(\tau)d\tau
\right]\nnm\\
&~~+M_{0}(\mathscr{P}(K)+1)\int^{t}_{0}\|(\sqrt{\alpha_{c}(x)}\partial^{5}_{t}u,\frac{\alpha_{c}(x)}{\sqrt{x}}\frac{\partial^{4}_{t}
u}{x},\frac{\alpha_{c}(x)}{\sqrt{x}}\partial^{4}_{t}\partial_{x}u)(\tau)\|^{2}_{0}d\tau.\nnm
\end{align}
Then, for any $c\geq \overline{c}_{1}$ and small enough $0<\overline{T}_{1}\leq \overline{T}$, we can obtain \eqref{mai5} using the Gronwall inequality. We easily observe that the order of polynomial function $\mathscr{P}(K)$ is $10.$ This completes the proof of Lemma \ref{PPPVV1}.\end{proof}

\section{Elliptic Estimates for the case $\gamma=2$}\label{sec-4}
In this section, we establish the higher order spatial derivatives of local smooth solutions to the free boundary value problem~\eqref{C-24}-\eqref{L1.2}  on $[0,1]\times[0,T]$
    under the assumption \eqref{1.15}.
 Since the different singularities of the original point $x=0$ and the boundary point $x=1$ , we divide our estimates of each terms into the interior estimates and the boundary estimates. More precisely, we give the estimates for $u,\partial_{t}u$ in subsection \ref{sec-41}, the estimates for $\partial^{2}_{t}u$ in subsection \ref{sec-42} and for $\partial^{3}_{t}u$ in subsection \ref{sec-43}, respectively. Finally, we give the estimates of $E(t)$ in subsection \ref{E(t)}.
\par We can rewrite \eqref{TB-1} as
\begin{align}&\alpha_{0}(x)\partial^{k}_{t}\partial^{2}_{x}u+\alpha'_{0}(x)\partial^{k}_{t}\partial_{x}u-\alpha'_{0}(x)\frac{\partial^{k}_{t}u}{x}=
-(\alpha'_{0}(x)-\frac{\alpha_{0}(x)}{x})\partial^{k}_{t}\partial_{x}u
+\mathcal {F}^{k},\label{Ellip-1}\end{align}
with
\begin{align}\mathcal {F}^{k}&:=\mathcal {J}^{k}_{9}+h_{1} \rho_{0}\partial^{k}_{t}\partial_{x}u- h_{2}\partial^{k}_{t}u+\frac{(1-\frac{1}{c^{2}}\frac{\rho_{0}}{r_{x}}\frac{x}{r}\Theta)^{3}}
{2\Theta^{2}}\Theta_{0}(1+\frac{\rho_{0}}{c^{2}})\sum^{8}_{l=1}\mathcal {J}^{k}_{l},\label{Ellip-10}
\end{align}
where $\alpha_{0}(x)=\rho_{0}x,$ $J^{k}_{l}(l=1,...,5)$ are given by \eqref{TB-5}, respectively, and $h_{i}(i=1,2),~J^{k}_{l}(l=6,7,8)$ satisfy
\begin{align}h_{1}&:=1+\frac{[\Theta_{0}(1+\frac{\rho_{0}}{c^{2}})]_{x}}{\Theta_{0}(1+\frac{\rho_{0}}{c^{2}})}x+\frac{(1-\frac{1}{c^{2}}
\frac{\rho_{0}}{r_{x}}\frac{x}{r}\Theta)^{3}}{2\Theta^{2}}\left[\frac{\Theta^{2}}{(1-\frac{1}{c^{2}}\frac{\rho_{0}}{r_{x}}\frac{x}{r}\Theta)^{3}}\right]_{x}x,\label{Ellip-2}\\ h_{2}&:=2\Theta_{0}(1+\frac{\rho_{0}}{c^{2}})[\frac{\rho_{0}}{\Theta_{0}(1-\frac{\rho_{0}}{c^{2}})}]_{x}(2+\frac{1}{c^{2}}\frac{\rho_{0}}{r_{x}}\frac{x}{r}\Theta)+\frac{[\Theta_{0}(1+\frac{\rho_{0}}{c^{2}})]_{x}}{\Theta_{0}(1+\frac{\rho_{0}}{c^{2}})}\rho_{0}\nnm\\
&+\frac{(1-\frac{1}{c^{2}}\frac{\rho_{0}}{r_{x}}\frac{x}{r}\Theta)^{3}}{2\Theta^{2}}
\left[\frac{\theta^{2}}{(1-\frac{1}{c^{2}}\frac{\rho_{0}}{r_{x}}\frac{x}{r}\Theta)^{3}}\right]_{x}
(\frac{1}{c^{2}}\frac{\rho_{0}}{r_{x}}\frac{x}{r}\Theta)\rho_{0}+\frac{1}{2}(\frac{1}{c^{2}}
\frac{\rho_{0}}{r_{x}}\frac{x}{r}\Theta)_{x}\rho_{0},\label{Ellip-3}\\
\mathcal {J}^{k}_{6}&:=\left[\frac{2J}{c ^{2}r_{x}}\frac{\alpha^{2}_{c}(x)}{x}(u\partial^{k+1}_{t}u+\upsilon\partial^{k+1}_{t}\upsilon
+\omega\partial^{k+1}_{t}\omega)\right]_{x},\label{Ellip-4}\\
\mathcal{J}^{k}_{7}&:=-\left\{\frac{\Theta^{2}}{(1-\frac{1}{c^{2}}\frac{\rho_{0}}{r_{x}}
\frac{x}{r}\Theta)^{3}}\frac{\alpha^{2}_{c}(x)}{x}\left[(1-\frac{x}{rr^{3}_{x}})2\partial^{k}_{t}\partial_{x}
u+(1-\frac{x^{2}}{r^{2}r^{2}_{x}})(1+\frac{1}{c^{2}}\frac{\rho_{0}}{r_{x}}\frac{x}{r}
\Theta)\frac{\partial^{k}_{t}u}{x}\right]\right\}_{x},\label{Ellip-5}\\
\mathcal{J}^{k}_{8}&:=-\frac{\Theta^{2}}{(1-\frac{1}{c^{2}}\frac{\rho_{0}}{r_{x}}
\frac{x}{r}\Theta)^{3}}\frac{\alpha^{2}_{c}(x)}{x^{2}}\left[(1-\frac{x^{2}}{r^{2}r^{2}_{x}})(1+\frac{1}{c^{2}}\frac{\rho_{0}}{r_{x}}\frac{x}{r}
\Theta)\partial^{k}_{t}\partial_{x}
u+(1-\frac{x^{3}}{r^{3}r_{x}})\frac{\partial^{k}_{t}u}{x}\right],\label{Ellip-6}\\
\mathcal {J}^{k}_{9}&:=\frac{(1-\frac{1}{c^{2}}\frac{\rho_{0}}{r_{x}}\frac{x}{r}\Theta)^{3}}{2\Theta^{2}}\Theta_{0}(1+\frac{\rho_{0}}{c^{2}})\left[\widetilde{a}_{11}x\partial^{k+2}_{t}u-\partial^{k+1}_{t}\left(\widetilde{a}_{11}x\frac{\upsilon^{2}}{r}\right)\right]\label{Elliptic-100}.
\end{align}
\par We first determine the constant $\delta$ in \eqref{cut-1} and \eqref{cut-2}. Because $\rho(0)>0$ and $\alpha'_{0}(0)=\rho_{0}(0)>0,$ then there exist positive constant $\delta_{0}$ such that for any $x\in (0,\delta_{0}),$
\begin{equation} \frac{\rho_{0}(0)}{2}\leq \alpha'_{0}(x)\leq\frac{3\rho_{0}(0)}{2}.\label{cut-3}\end{equation}
Then, we take $\delta$ as $0<2\delta\leq \delta_{0}.$

\subsection{Estimates for $u,\partial_{t}u$}\label{sec-41}
\begin{lemma}\label{PPVV2} Let $(r,u,\upsilon,\omega)$ be a classical solution to the free boundary problem \eqref{C-24}--\eqref{L1.2} satisfying \eqref{1.15} on $[0,1]\times[0,T]$. Then, for any $t\in(0,\overline{T}_{1}]$ and $c\geq \overline{c}_{1}$ the following estimates hold
\begin{align}&\|\xi\alpha_{0}(x)x\partial_{t}\partial^{3}_{x}(\frac{u}{x})\|^{2}_{0}+\|\xi\alpha_{0}(x)\partial_{t}\partial^{2}_{x}(\frac{u}{x})\|^{2}_{0}+\|\xi\alpha'_{0}(x)\partial_{t}\partial_{x}(\frac{u}{x})\|^{2}_{0}\nnm\\
&\leq \mathscr{P}_{0}+M_{0}\left(\mathscr{P}_{4}(K)+1\right)\int^{t}_{0}\left(E(\tau)+E^{2}(\tau)\right)d\tau\nnm\\
&+M_{0}\left(\mathscr{P}_{4}(K)+1\right)E(t)\int^{t}_{0}E(\tau)d,\label{BEEE7}\\
&\|\chi\alpha^{\frac{3}{2}}_{0}(x)\partial_{t}\partial^{3}_{x}u\|^{2}_{0}+\|\chi\alpha^{\frac{1}{2}}_{0}(x)\alpha'_{0}(x)\partial^{2}_{x}u\|^{2}_{0}\nnm\\
&\leq \mathscr{P}_{0}+M_{0}\left(\mathscr{P}_{4}(K)+1\right)\int^{t}_{0}\left(E(\tau)+E^{2}(\tau)\right)d\tau\nnm\\
&+M_{0}\left(\mathscr{P}_{4}(K)+1\right)E(t)\int^{t}_{0}E(\tau)d\tau+M_{0}\|\chi\alpha^{1/2}_{0}(x)\partial^{3}_{t}\partial_{x}u\|^{2}_{0},\label{BBEEE6}\end{align}
and
\begin{align}
&\|\chi\alpha_{0}(x)\partial^{3}_{x}u\|^{2}_{0}+\|\chi\alpha'_{0}(x)\partial^{2}_{x}u\|^{2}_{0}\nnm\\
&\leq \mathscr{P}_{0}+M_{0}\left(\mathscr{P}_{4}(K)+1\right)\int^{t}_{0}\left(E(\tau)+E^{2}(\tau)\right)d\tau\nnm\\
&+M_{0}\left(\mathscr{P}_{4}(K)+1\right)E(t)\int^{t}_{0}E(\tau)d\tau+M_{0}\|\chi\partial^{2}_{t}\partial_{x}u\|.\label{BBEEE5}
\end{align}
\end{lemma}
\begin{proof}
Using
\begin{equation}\partial^{k}_{t}\partial^{j}_{x}u:=x\partial^{j}_{x}(\frac{\partial^{k}_{t}u}{x})+j\partial^{j-1}_{x}(\frac{\partial^{k}_{t}u}{x}),~j=1,2,...\label{Formula}\end{equation}
we obtain from \eqref{Ellip-1}
\begin{align}\alpha_{0}(x)&x\partial^{k}_{t}\partial^{3}_{x}(\frac{u}{x})+5\alpha_{0}(x)\partial^{k}_{t}\partial^{2}_{x}(\frac{u}{x})
+3\alpha'_{0}(x)\partial^{k}_{t}\partial_{x}(\frac{u}{x})=-2\rho_{0x}x^{2}\partial^{k}_{t}\partial_{x}(\frac{u}{x})\nnm\\
&-2\alpha''_{0}(x)\partial^{k}_{t}\partial_{x}u-\alpha''_{0}(x)\frac{\partial^{k}_{t}u}{x}-\left[(\alpha'_{0}(x)-\frac{\alpha_{0}(x)}{x})\partial^{k}_{t}\partial_{x}u\right]_{x}+\mathcal {F}^{k}_{x}.\label{EllipticQ-2}
\end{align}
\textbf{Interior Estimate}.
Multiplying \eqref{EllipticQ-2} by $\xi$ and taking $L^{2}-$ norm, for $k=1$

\begin{align}&\|\xi\left[\alpha_{0}(x)x\partial_{t}\partial^{3}_{x}(\frac{u}{x})+5\alpha_{0}(x)\partial_{t}\partial^{2}_{x}(\frac{u}{x})
+3\alpha'_{0}(x)\partial_{t}\partial_{x}(\frac{u}{x})\right]\|^{2}_{0}\nnm\\
&\leq\|\xi\left\{-\rho_{0x}x^{2}\partial_{t}\partial_{x}(\frac{u}{x})-2\alpha''_{0}(x)\partial_{t}\partial_{x}u-\alpha''_{0}(x)\frac{\partial_{t}u}{x}-\left[(\alpha'_{0}(x)-\frac{\alpha_{0}(x)}{x})\partial_{t}\partial_{x}u\right]_{x}\right\}\|^{2}_{0}
\nnm\\&~~~+\|\xi\mathcal {F}^{1}_{x}\|^{2}_{0}.\label{EllipticQ-3}
\end{align}
Using integrating by parts and Cauchy-Schwarz inequality yields for any positive constant $\varepsilon,$
\begin{align}
&\|\xi\left[\alpha_{0}(x)x\partial_{t}\partial^{3}_{x}(\frac{u}{x})+5\alpha_{0}(x)\partial_{t}\partial^{2}_{x}(\frac{u}{x})
+3\alpha'_{0}(x)\partial_{t}\partial_{x}(\frac{u}{x})\right]\|^{2}_{0}\nnm\\
&~~\geq \|\xi\alpha_{0}(x)x\partial_{t}\partial^{3}_{x}(\frac{u}{x})\|^{2}_{0}+3\|\xi\alpha_{0}(x)\partial_{t}\partial^{2}_{x}(\frac{u}{x})\|^{2}_{0}-\varepsilon\|\xi\alpha'_{0}(x)\partial_{t}\partial_{x}(\frac{u}{x})\|^{2}_{0}
\nnm\\
&~~~-M_{0}(\varepsilon,\delta)\left[\mathscr{P}_{0}+(\mathscr{P}(K)+1)\int^{t}_{0}E(\tau)d\tau\right],\label{EllipticQ-4}
\end{align}
where we have used \eqref{cut-3}.
\par For the estimates of the right hands side in \eqref{EllipticQ-3}, the first term can be easily estimated by \begin{equation}M_{0}\left[\mathscr{P}_{0}+(\mathscr{P}(K)+1)\int^{t}_{0}E(\tau)d\tau\right].\nnm\end{equation}
 For the second term $\|\xi\mathcal {F}^{1}_{x}\|^{2}_{0},$ the highest order terms with respect to $t$ and $x$ are $\|\xi \partial_{x}\mathcal {J}^{1}_{9}\|^{2}_{0}$ and $\|\xi\left[\frac{(1-\frac{1}{c^{2}}\frac{\rho_{0}}{r_{x}}\frac{x}{r}\Theta)^{3}}
{2\Theta^{2}}\Theta_{0}(1+\frac{\rho_{0}}{c^{2}})\frac{x}{\alpha_{c}(x)}\mathcal {J}^{1}_{7}\right]_{x}\|^{2}_{0}$ respectively. Thus, we only give the estimates for these two terms while the other terms in $\|\xi\mathcal {F}^{1}_{x}\|^{2}_{0}$ can be similarly estimated and bounded by the right hand side of \eqref{BEEE7}. Using \eqref{Elliptic-100}, we have
\begin{align}\|\xi \partial_{x}\mathcal {J}^{1}_{9}\|^{2}_{0}
&\leq M_{0}\left[(1+\|\alpha_{0}(x)\mathscr{K}^{0,1}_{t,x}\|^{2}_{L^{\infty}})\|\partial^{3}_{t}u\|^{2}_{0}+\|\xi\alpha_{0}(x)\partial^{3}_{t}\partial_{x}u\|^{2}_{0}\right]\nnm\\
&+M_{0}\|\frac{\upsilon}{x}\|^{2}_{L^{\infty}}\left(\|\xi\mathscr{K}^{2,0}_{t,x}\|^{2}_{0}\|\alpha_{0}(x)\upsilon_{x}\|^{2}_{L^{\infty}}+\|\xi\alpha_{0}(x)\mathscr{K}^{1,1}_{t,x}\|^{2}_{0}\|\partial_{t}\upsilon\|^{2}_{L^{\infty}}\right)\nnm\\
&+M_{0}\left(\|\alpha_{0}(x)\partial_{t}\partial_{x}\upsilon\|^{2}_{0}\|\xi\mathscr{K}^{1,0}_{t,x}\|^{2}_{L^{\infty}}+\|\frac{\upsilon}{x}\|^{2}_{L^{\infty}}\|\alpha_{0}(x)\partial_{t}\partial^{2}_{x}\upsilon\|^{2}_{0}\right)\nnm\\
&+M_{0}\|\alpha_{0}(x)\mathscr{K}^{0,1}_{t,x}\|^{2}_{L^{\infty}}\left(\|\partial^{2}_{t}\upsilon\|^{2}_{0}\|\frac{\upsilon}{x}\|^{2}_{L^{\infty}}
+\|\partial_{t}\upsilon\|^{2}_{0}\|\frac{\partial_{t}\upsilon}{x}\|^{2}_{L^{\infty}}\right)\nnm\\
&+M_{0}\|\xi\alpha_{0}(x)\mathscr{K}^{2,1}_{t,x}\|^{2}_{0}\|\frac{\upsilon}{x}\|^{2}_{L^{\infty}}.\label{EHU1}
\end{align}
From \eqref{Ellip-5},
\begin{align}
&\|\xi\left[\frac{(1-\frac{1}{c^{2}}\frac{\rho_{0}}{r_{x}}\frac{x}{r}\Theta)^{3}}
{2\Theta^{2}}\Theta_{0}(1+\frac{\rho_{0}}{c^{2}})\frac{x}{\alpha_{c}(x)}\mathcal {J}^{1}_{7}\right]_{x}\|^{2}_{0}\nnm\\
&\leq M_{0}\|\xi\left(\partial_{t}\partial_{x}u,\frac{\partial_{t}u}{x}\right)\|^{2}_{L^{\infty}}\|\alpha_{0}(x)\mathscr{K}^{0,1}_{t,x}\|^{2}_{L^{\infty}}
\|\mathscr{K}^{0,1}_{t,x}\|^{2}_{0}\int^{t}_{0}\|\left(\partial_{x}u,\frac{u}{x}\right)(\tau)\|^{2}_{L^{\infty}}d\tau\nnm\\
&+M_{0}\|\alpha_{0}(\cdot)\mathscr{K}^{0,2}_{t,x}\|^{2}_{0}\|\xi\left(\partial_{t}\partial_{x}u,\frac{\partial_{t}u}{x}\right)\|^{2}_{L^{\infty}}
\int^{t}_{0}\|\left(\partial_{x}u,\frac{u}{x}\right)(\tau)\|^{2}_{L^{\infty}}d\tau\nnm\\
&+M_{0}\|\xi\left(\partial_{t}\partial_{x}u,\frac{\partial_{t}u}{x}\right)\|^{2}_{L^{\infty}}\left(1
+\|\alpha_{0}(x)\mathscr{K}^{0,1}_{t,x}\|^{2}_{\infty}\right)\|\left(r_{xx},(\frac{x}{r})_{x}\right)\|^{2}_{0}\nnm\\
&+M_{0}\|\xi\left(\partial_{t}\partial^{2}_{x}u,(\frac{\partial_{t}u}{x})_{x}\right)\|^{2}_{0}\left(1
+\|\alpha_{0}(x)\mathscr{K}^{0,1}_{t,x}\|^{2}_{\infty}\right)\int^{t}_{0}\|\left(\partial_{x}u,\frac{u}{x}\right)(\tau)\|^{2}_{L^{\infty}}d\tau\nnm\\
&+M_{0}\|\xi\alpha_{0}(x)\left(\partial_{t}\partial^{3}_{x}u,(\frac{\partial_{t}u}{x})_{xx}\right)\|^{2}_{0}\int^{t}_{0}\|\left(\partial_{x}u,\frac{u}{x}\right)(\tau)\|^{2}_{L^{\infty}}d\tau\nnm\\
&+M_{0}\|\xi\alpha_{0}(x)\left(\partial_{t}\partial^{2}_{x}u,(\frac{\partial_{t}u}{x})_{x}\right)\|^{2}_{L^{\infty}}\|\left(r_{xx},(\frac{x}{r})_{x}\right)\|^{2}_{0}\nnm\\
&+M_{0}\|\xi\left(\partial_{t}\partial_{x}u,\frac{\partial_{t}u}{x}\right)\|^{2}_{L^{\infty}}\|\alpha_{0}(x)\left(r_{xxx},(\frac{x}{r})_{xx}\right)\|^{2}_{0}.\label{BEEE2}
\end{align}

Finally, it holds that
\begin{align}&\|\xi\alpha_{0}(x)x\partial_{t}\partial^{3}_{x}(\frac{u}{x})\|^{2}_{0}+3\|\xi\alpha_{0}(x)\partial_{t}\partial^{2}_{x}(\frac{u}{x})\|^{2}_{0}-\varepsilon\|\xi\alpha'_{0}(x)\partial_{t}\partial_{x}(\frac{u}{x})\|^{2}_{0}\nnm\\
&~~~\leq \mathscr{P}_{0}+M_{0}\left(\mathscr{P}(K)+1\right)\int^{t}_{0}(E(\tau)+E^{2}(\tau))d\tau\nnm\\
&~~~~~~~~+M_{0}\left(\mathscr{P}(K)+1\right)E(t)\int^{t}_{0}E(\tau)d\tau.\label{BEEE6}
\end{align}
which in combination with \eqref{EllipticQ-2} $(k=1)$ yields \eqref{BEEE7} for suitably small $\varepsilon,$ where the order of polynomial function $\mathscr{P}(K)$ is $4$ .

\textbf{Boundary Estimates}. For convenience, we only give the estimates in \eqref{BBEEE6} and omit the estimate in \eqref{BBEEE5} which can be obtained by a similar proceeding.

We write \eqref{Ellip-1} as for $k=1$
\begin{align}&\alpha_{0}(x)\partial_{t}\partial^{2}_{x}u+2\alpha'_{0}(x)\partial_{t}\partial_{x}u=\alpha'_{0}(x)\frac{\partial^{k}_{t}u}{x}
+\frac{\alpha_{0}(x)}{x}\partial^{1}_{t}\partial_{x}u
+\mathcal {F}^{k}.\label{BE-1}\end{align}
Taking $\partial_{x},$
\begin{align}&\alpha_{0}(x)\partial_{t}\partial^{3}_{x}u+3\alpha'_{0}(x)\partial_{t}\partial^{2}_{x}u=-2\alpha''_{0}(x)\partial^{k}_{t}\partial_{x}u\nnm\\
&~~~~+\left(\alpha'_{0}(x)\frac{\partial_{t}u}{x}
+\frac{\alpha_{0}(x)}{x}\partial_{t}\partial_{x}u\right)_{x}
+\mathcal {F}^{1}_{x}.\label{BE-2}\end{align}
Multiplying \eqref{BE-2} by $\chi\sqrt{\alpha_{0}(x)},$
\begin{align}&\|\chi\sqrt{\alpha_{0}(x)}\left(\alpha_{0}(x)\partial_{t}\partial^{3}_{x}u+3\alpha'_{0}(x)\partial_{t}\partial^{2}_{x}u\right)\|^{2}_{0}\nnm\\
&\leq\|\chi\sqrt{\alpha_{0}(x)}\left[-2\alpha''_{0}(x)\partial_{t}\partial_{x}u+\left(\alpha'_{0}(x)\frac{\partial_{t}u}{x}
+\frac{\alpha_{0}(x)}{x}\partial_{t}\partial_{x}u\right)_{x}\right]\|^{2}_{0}\nnm\\
&+\|\sqrt{\alpha_{0}(x)}\chi\mathcal {F}^{1}_{x}|^{2}_{0}.\label{MMM1}
\end{align}
Since $\alpha'(1)<0,$ there exists a positive constant $\delta_{1}>0$ such that for any $\delta_{1}\geq\frac{\delta}{2},$ $-\infty<\alpha'_{0}(x)<0, ~\forall x\in(\delta,1],$ then using the integration by parts
\begin{align}&\|\chi\sqrt{\alpha_{0}(x)}\left(\alpha_{0}(x)\partial_{t}\partial^{3}_{x}u+3\alpha'_{0}(x)\partial_{t}\partial^{2}_{x}u\right)\|^{2}_{0}\nnm\\
&\geq\|\chi\alpha^{\frac{3}{2}}_{0}(x)\partial_{t}\partial^{3}_{x}u\|^{2}_{0}+2\|\chi\alpha^{\frac{1}{2}}_{0}(x)\alpha'_{0}(x)\partial_{t}\partial^{2}_{x}u\|^{2}_{0}
- M_{0}\left[\mathscr{P}_{0}+\int^{t}_{0}E(\tau)(\tau)d\tau\right]\label{BBBEE1}
\end{align}
It is easy to see that  the first term on the right hand side in \eqref{MMM1} can be bounded by the right hand side in \eqref{BBEEE6}. For the estimate for $\|\sqrt{\alpha_{0}(x)}\chi\mathcal {F}^{1}_{x}\|^{2}_{0},$  the main difficulties terms are $\|\sqrt{\alpha_{0}(x)}\chi(\mathcal {J}^{1}_{9})_{x}\|^{2}_{0}$ and $\|\chi\sqrt{\alpha_{0}(x)}\left(\mathcal {J}^{1}_{7}\right)_{x}\|^{2}_{0}.$
 \par We write $(\mathcal {J}^{1}_{9})_{x}:=\mathcal {J}^{10}_{9}+\mathcal {J}^{11}_{9}$ with
\begin{align}\mathcal {J}^{10}_{9}&:=\left[\frac{(1-\frac{1}{c^{2}}\frac{\rho_{0}}{r_{x}}\frac{x}{r}\Theta)^{3}}{2\Theta^{2}}\Theta_{0}(1+\frac{\rho_{0}}{c^{2}})\widetilde{a}_{11}x\partial^{3}_{t}u\right]_{x}-
\left[\frac{(1-\frac{1}{c^{2}}\frac{\rho_{0}}{r_{x}}\frac{x}{r}\Theta)^{3}}{2\Theta^{2}}\Theta_{0}(1+\frac{\rho_{0}}{c^{2}})\right]_{x}
\nnm\\&~~~~~~~~\times\partial^{2}_{t}\left(\widetilde{a}_{11}x\frac{\upsilon^{2}}{r}\right),\nnm\\
\mathcal {J}^{11}_{9}&:=-
\frac{(1-\frac{1}{c^{2}}\frac{\rho_{0}}{r_{x}}\frac{x}{r}\Theta)^{3}}{2\Theta^{2}}\Theta_{0}(1+\frac{\rho_{0}}{c^{2}})\partial^{2}_{t}\partial_{x}\left(\widetilde{a}_{11}x\frac{\upsilon^{2}}{r}\right),\nnm
\end{align}
and obtain
\begin{align}\|\sqrt{\alpha_{0}(x)}\chi\mathcal {J}^{10}_{9}\|^{2}_{0}&
\leq M_{0}\|\mathscr{K}_{B}^{0,1}\|_{L^{4}}\left(\|\partial^{2}_{t}\upsilon\|_{L^{4}}+\|\partial_{t}\upsilon\|^{2}_{L^{\infty}}+\|\mathscr{K}_{t,x}^{1,0}\|_{L^{4}}\|\partial_{t}\upsilon\|^{2}_{L^{\infty}}\right)\nnm\\
&+M_{0}\|\mathscr{K}_{B}^{0,1}\|^{2}_{L^{4}}\|\sqrt{\alpha_{0}(x)}\left(\mathscr{K}_{t,x}^{2,0},\partial^{3}_{t}u\right)\|^{2}_{L^{4}}\nnm\\
&+M_{0}\|\left(\chi\alpha_{0}(x)\partial^{3}_{t}u,\partial^{2}_{t}\upsilon,\mathscr{K}_{t,x}^{0,1}\partial_{t}\upsilon,(\partial_{t}\upsilon)^{2}\right)\|^{2}_{L^{\infty}}\|\left(r_{xx},(\frac{x}{r})_{x}\right)\|^{2}_{0}\nnm\\
&+M_{0}\left(\|\chi\alpha^{1/2}_{0}(x)\partial^{3}_{t}\partial_{x}u\|^{2}_{0}+\|\mathscr{K}_{t,x}^{2,0}\|^{2}_{0}\|(\alpha_{0}(x)r_{xx},\chi(\frac{x}{r})_{x})\|^{2}_{L^{\infty}}\right).\nnm
\nnm\end{align}
\begin{align}
\|\sqrt{\alpha_{0}(x)}\chi\mathcal {J}^{11}_{9}\|^{2}_{0}&
\leq M_{0}\|\mathscr{K}_{B}^{1,0}\|^{2}_{L^{\infty}}\|\chi\left(\mathscr{K}_{t,x}^{1,1}-|u_{xx}|,\alpha_{0}(x)u_{xx}\right)\|^{2}_{0}\nnm\\
&+M_{0}\left[\|\alpha_{0}(x)\partial_{t}\partial^{2}_{x}u\|^{2}_{0}
+\|(\partial_{t}u,\partial_{t}\upsilon,\partial_{t}\omega)\|^{2}_{L^{\infty}}\|(\partial_{t}\partial_{x}u,\partial_{t}\partial_{x}\upsilon,\partial_{t}\partial_{x}\omega)\|^{2}_{0}\right]\nnm\\
&+M_{0}\|\left(\mathscr{K}_{t,x}^{1,1}-|u_{xx}|,\alpha_{0}(x)u_{xx},\sqrt{\alpha_{0}(x)}\partial^{2}_{t}\partial_{x}\upsilon\right)\|^{2}_{0}\|\partial_{t}\upsilon\|^{2}_{\infty}\nnm\\
&+M_{0}\|\mathscr{K}_{B}^{1,0}\|^{2}_{L^{\infty}}\left(\|((\partial_{t}\upsilon)^{2},\partial^{2}_{t}\upsilon)\|^{2}_{0}+\|\partial_{x}\upsilon\|^{2}_{0}\|\partial_{t}\upsilon\|^{2}_{\infty}+\|\sqrt{\alpha_{0}(x)}\partial_{t}\partial_{x}\upsilon\|^{2}_{0}\right)\nnm\\
&+M_{0}\|\sqrt{\alpha_{0}(x)}\left(\partial^{2}_{t}\partial_{x}u,\partial^{2}_{t}\partial_{x}\upsilon,\partial^{2}_{t}\partial_{x}\omega\right)\|^{2}_{0}.\label{BBBEE2}
\end{align}
on the other hands,
\begin{align}&\|\xi\sqrt{\alpha_{0}(x)}\left(\mathcal {J}^{1}_{7}\right)_{x}\|^{2}_{0}\nnm\\
&\leq M_{0}\|\alpha_{0}(x)\mathscr{K}^{0,1}_{t,x}\|^{2}_{L^{\infty}}\|\mathscr{K}_{B}\|_{L^{4}}\|\chi\sqrt{\alpha_{0}(x)}(\partial_{t}\partial_{x}u,\partial_{t}u)\|^{2}_{L^{4}}\int^{t}_{0}\|(u_{x},\frac{u}{x})(\tau)\|^{2}_{L^{\infty}}d\tau\nnm\\
&+M_{0}\|\alpha_{0}(x)\left(r_{xxx},(\frac{x}{r})_{xx}\right)\|^{2}_{0}\|\chi\sqrt{\alpha_{0}(x)}(\partial_{t}\partial_{x}u,\partial_{t}u)\|^{2}_{L^{\infty}}\nnm\\
&+M_{0}\|\alpha_{0}(x)r_{xx}\|^{2}_{L^{\infty}}\|(\partial_{t}\partial_{x}u,\frac{\partial_{t}u}{x})\|^{2}_{0}\nnm\\
&+M_{0}\left(1+\|\alpha_{0}(x)\mathscr{K}^{0,1}_{t,x}\|^{2}_{L^{\infty}}\right)\|\chi\sqrt{\alpha_{0}(x)}(\partial_{t}\partial_{x}u,\partial_{t}u)\|^{2}_{L^{\infty}}\|(r_{xx},(\frac{x}{r})_{x})\|^{2}_{0}\nnm\\
&+M_{0}\|\alpha_{0}(x)\mathscr{K}^{0,2}_{t,x}\|_{0}\|\chi\sqrt{\alpha_{0}(x)}(\partial_{t}\partial_{x}u,\partial_{t}u)\|_{L^{4}}\int^{t}_{0}\|(u_{x},\frac{u}{x})(\tau)\|^{2}_{L^{4}}d\tau\nnm\\
&+M_{0}\|\chi\sqrt{\alpha_{0}(x)}(\partial_{t}\partial^{2}_{x}u,\partial_{t}\partial_{x}u,\partial_{t}u)\|^{2}_{0}\int^{t}_{0}\|(u_{x},\frac{u}{x})(\tau)\|^{2}_{L^{\infty}}d\tau\nnm\\
&+M_{0}\|\chi(\partial_{t}\partial_{x}u,\partial_{t}u)\|^{2}_{0}\int^{t}_{0}\|(u_{x},\frac{u}{x})(\tau)\|^{2}_{L^{\infty}}d\tau\nnm\\
&+M_{0}\|\chi\alpha^{\frac{3}{2}}_{0}(x)\left(\partial_{t}\partial^{3}_{x}u,\partial_{t}\partial^{2}_{x}u,\partial_{t}\partial_{x}u,\partial_{t}u\right)\|^{2}_{0}\int^{t}_{0}\|(u_{x},\frac{u}{x})(\tau)\|^{2}_{L^{\infty}}d\tau.\label{BBBEE3}
\end{align}
Here and in the sequel, $\mathscr{K}_{B}=\mathscr{K}_{t,x}^{0,1}-|r_{xx}|-|(\frac{x}{r})_{x}|.$ From \eqref{BBBEE1}-\eqref{BBBEE3}, we conclude
\eqref{BBEEE6}, where the order of polynomial function $\mathscr{P}(K)$ is also $4$. This is the end of proof of Lemma \ref{PPVV2}.\end{proof}

\subsection{Estimates for $\partial^{2}_{t}u$}\label{sec-42}

\begin{lemma}\label{PPVV3} Let $(r,u,\upsilon,\omega)$ be a classical solution to the free boundary problem \eqref{C-24}--\eqref{L1.2} satisfying \eqref{1.15} on $[0,1]\times[0,T]$. Then, for any $t\in(0,\overline{T}_{1}]$ and $c\geq \overline{c}_{1}$ the following estimates hold
\begin{align}&\|\xi\frac{1}{\sqrt{x}}\left(\alpha_{0}(x)\partial^{2}_{t}\partial^{2}_{x}u,\alpha'_{0}(x)\partial^{2}_{t}\partial_{x}u,\alpha'_{0}(x)\frac{\partial^{2}_{t}u}{x}\right)\|^{2}_{0}\nnm\\
&~~~~\leq \mathscr{P}_{0}+M_{0}(\mathscr{P}_{4}(K)+1)\int^{t}_{0}(E(\tau)+E^{2}(\tau))d\tau,\label{EN-4}\end{align}
\begin{align}
&\|\chi\alpha_{0}(x)\partial^{2}_{t}\partial^{2}_{x}u\|^{2}_{0}+\|\chi\alpha'_{0}(x)\partial^{2}_{t}\partial_{x}u\|^{2}_{0}\nnm\\
&\leq \mathscr{P}_{0}+M_{0}\left(\mathscr{P}_{4}(K)+1\right)\int^{t}_{0}(E(\tau)+E^{2}(\tau))d\tau\nnm\\
&+M_{0}\left(\mathscr{P}_{4}(K)+1\right)E(t)\int^{t}_{0}E(\tau)d\tau.\label{BBBEEE1}\end{align}
\end{lemma}
\begin{proof}
 The interior estimate of $\partial^{2}_{t}u$  is more complicated than the boundary estimates of it, because the $\sqrt{x}$ appears in the denominator of some terms involving the higher order spatial derivatives of  $\partial^{2}_{t}u.$ Thus, we give the interior estimate and omit the boundary estimate in this section. However, the boundary estimate of it can be given by the similar proceeding of boundary estimates for $u,\partial_{t}u$ in Subsection \ref{sec-41}.

Due to \eqref{Formula}, multiplying \eqref{Ellip-1} by $\xi\frac{1}{\sqrt{x}},$
\begin{align}&\|\xi\frac{1}{\sqrt{x}}\left(\alpha_{0}(x)\partial^{2}_{t}\partial^{2}_{x}u,\alpha'_{0}(x)\partial^{2}_{t}\partial_{x}u,\alpha'_{0}(x)\frac{\partial^{2}_{t}u}{x}\right)\|^{2}_{0}\nnm\\
&\leq M_{0}(\mathscr{P}_{4}(K)+1)\int^{t}_{0}(E(\tau)+E^{2}(\tau))d\tau+\|\xi\frac{1}{\sqrt{x}}\mathcal {F}^{2}\|^{2}_{0}.
\end{align}
Using the integration by parts,
\begin{align} &\|\xi\frac{\alpha_{0}(x)}{\sqrt{x}}\left[\partial^{2}_{t}\partial^{2}_{x}(\frac{u}{x})+\partial^{2}_{t}\partial_{x}(\frac{u}{x})\right]\|^{2}_{0}\nnm\\
&~~~~\geq \|\xi\frac{\alpha_{0}(x)}{\sqrt{x}}x\partial^{2}_{t}\partial^{2}_{x}(\frac{u}{x})\|^{2}_{0}+\|\xi\frac{\alpha_{0}(x)}{\sqrt{x}}\partial^{2}_{t}\partial_{x}(\frac{u}{x})\|^{2}_{0}-M_{0}\left(\mathscr{P}_{0}+\int^{t}_{0}E(\tau)d\tau\right).
\label{EN-0}\end{align}
For the estimate of $\|\xi\frac{1}{\sqrt{x}}\mathcal {F}^{2}\|^{2}_{0},$ the main difficulty term is  $\|\xi\frac{1}{\sqrt{x}}\frac{(1-\frac{1}{c^{2}}\frac{\rho_{0}}{r_{x}}\frac{x}{r}\Theta)^{3}}
{2\Theta^{2}}\Theta_{0}(1+\frac{\rho_{0}}{c^{2}})\frac{x}{\alpha_{c}(x)}\mathcal {J}^{2}_{2}\|^{2}_{0}$ which can be estimated as
\begin{align}&\|\xi\frac{1}{\sqrt{x}}\frac{(1-\frac{1}{c^{2}}\frac{\rho_{0}}{r_{x}}\frac{x}{r}\Theta)^{3}}
{2\Theta^{2}}\Theta_{0}(1+\frac{\rho_{0}}{c^{2}})\frac{x}{\alpha_{c}(x)}\mathcal {J}^{2}_{2}\|^{2}_{0}\nnm\\
&\leq M_{0}\left[\|\alpha_{0}(x)\mathscr{K}^{0,1}_{t,x}\|^{2}_{L^{\infty}}\|\left(\frac{\partial_{t}\partial_{x}u}{\sqrt{x}},\frac{\partial_{t}u}{x\sqrt{x}}\right)\|^{2}_{0}
+\|\alpha_{0}(x)\left(\frac{\partial_{t}\partial_{x}u}{\sqrt{x}},\frac{\partial_{t}u}{x\sqrt{x}}\right)\|^{2}_{L^{\infty}}\|\xi\mathscr{K}^{1,1}_{t,x}\|^{2}_{0}\right]
\nnm\\&+M_{0}\left[\|\xi\frac{\alpha_{0}(x)}{\sqrt{x}}\mathscr{K}^{2,1}_{t,x}\|^{2}_{0}\|(u_{x},\frac{u}{x})(0)\|^{2}_{L^{\infty}}+
\|\frac{\alpha_{0}(x)}{\sqrt{x}}\mathscr{K}^{2,1}_{t,x}\|^{2}_{0}\int^{t}_{0}
\|\left(\partial_{t}\partial_{x}u,\frac{\partial_{t}u}{x}\right)(\tau)\|^{2}_{L^{\infty}}d\tau\right]\nnm\\
&+M_{0}\left[\|\xi\frac{\mathscr{K}^{2,0}_{t,x}}{\sqrt{x}}\|^{2}_{0}\|(u_{x},\frac{u}{x})(0)\|^{2}_{L^{\infty}}+\|\xi\mathscr{K}^{2,0}_{t,x}\|^{2}_{L^{\infty}}\int^{t}_{0}\|\left(\frac{\partial_{t}\partial_{x}u}{\sqrt{x}},\frac{\partial_{t}u}{x\sqrt{x}}\right)(\tau)\|^{2}_{0}d\tau\right]\nnm\\
&+M_{0}\|\left(\mathscr{K}^{1,0}_{t,x}-|u_{x}|,u_{x}(0)\right)\|^{2}_{L^{\infty}}\|\xi\frac{\alpha_{0}(x)}{\sqrt{x}}\left(\partial_{t}\partial^{2}_{x}u,\partial_{t}\partial_{x}(\frac{u}{x}\right)\|^{2}_{0}+G.\nnm\\
\end{align}
with
\begin{align}
&G\leq M_{0}\|\frac{\alpha_{0}(x)}{\sqrt{x}}\left(\partial_{t}\partial^{2}_{x}u,\partial_{t}\partial_{x}(\frac{u}{x}\right)\|^{2}_{0}\int^{t}_{x}\|\xi\partial_{t}\partial_{x}u(\tau)\|^{2}_{L^{\infty}}d\tau\nnm\\
&+M_{0}\|\xi\frac{\alpha_{0}(x)}{\sqrt{x}}\mathscr{K}^{2,0}_{t,x}\|^{2}_{L^{\infty}}\|(u_{xx},(\frac{u}{x})_{x})(0)\|^{2}_{L^{\infty}}\nnm\\
&+M_{0}
\|\xi\frac{1}{\sqrt{x}}\mathscr{K}^{2,0}_{t,x}\|^{2}_{0}\int^{t}_{0}
\|\xi\alpha_{0}(x)\left(\frac{\partial_{t}\partial^{2}_{x}u}{\sqrt{x}},\partial_{t}\partial_{x}(\frac{u}{x})\right)(\tau)\|^{2}_{0}d\tau.\label{BBBEEE3}
\end{align}
Finally, we obtain
\begin{align}&\|\xi\frac{\alpha_{0}(x)}{\sqrt{x}}x\partial^{2}_{t}\partial^{2}_{x}(\frac{u}{x})\|^{2}_{0}+\|\xi\frac{\alpha_{0}(x)}{\sqrt{x}}\partial^{2}_{t}\partial_{x}(\frac{u}{x})\|^{2}_{0}\nnm\\
&\leq M_{0}\left[\mathscr{P}_{0}+(\mathscr{P}(K)+1)\int^{t}_{0}\left(E(\tau)+E^{2}(\tau)\right)d\tau\right].\label{EN-1}\end{align}
Due to $\alpha'_{0}(x)=x\rho_{0x}+\frac{\alpha_{0}(x)}{x},$ we have
\begin{align}&\|\xi\frac{\alpha_{0}(x)}{\sqrt{x}}\partial^{2}_{t}\partial_{x}(\frac{u}{x})\|+M_{0}\|(\xi\alpha_{0}(x)\partial^{2}_{t}\partial_{x}u,\xi\partial^{2}_{x}u)\|^{2}_{0}\nnm\\
&\geq\|\alpha'_{0}(x)(\frac{\partial^{2}_{t}\partial_{x}u}{\sqrt{x}}+\frac{\partial^{2}_{t}u}{x})\|^{2}_{0}-M_{0}\|(\xi\alpha_{0}(x)\partial^{2}_{t}\partial_{x}u,\xi\partial^{2}_{x}u)\|^{2}_{0},
\nnm\end{align}
which in combination with \eqref{EN-1} shows
\begin{align}&-4\int^{1}_{0}\xi^{2}(\alpha'(x))^{2}\frac{\partial^{2}_{t}\partial_{x}u}{\sqrt{x}}\frac{\partial^{2}_{t}u}{x\sqrt{x}}dx\nnm\\
&\leq\|\xi\frac{\alpha_{0}(x)}{\sqrt{x}}\partial^{2}_{t}\partial_{x}(\frac{u}{x})\|+M_{0}\|(\xi\alpha_{0}(x)\partial^{2}_{t}\partial_{x}u,\xi\partial^{2}u)\|^{2}_{0}\nnm\\
&\leq M_{0}\left[\mathscr{P}_{0}+(\mathscr{P}(K)+1)\int^{t}_{0}(E(\tau)+E^{2}(\tau))d\tau\right].\nnm
\end{align}
Thus, we obtain \eqref{EN-4}, where the order of polynomial function $\mathscr{P}(K)$ is $4$ . This is the end of proof for Lemma \ref{PPVV3}.
\end{proof}

\subsection{Estimates for $\partial^{3}_{t}u$}\label{sec-43}
\begin{lemma}\label{PPVV4} Let $(r,u,\upsilon,\omega)$ be a classical solution to the free boundary problem \eqref{C-24}--\eqref{L1.2} satisfying \eqref{1.15} on $[0,1]\times[0,T]$. Then, for any $t\in(0,\overline{T}_{1}]$ and $c\geq \overline{c}_{1}$ the following estimates hold
\begin{align}&\|\xi\alpha_{0}(x)\partial^{3}_{t}\partial^{2}_{x}u\|^{2}_{0}+\|\xi\alpha'_{0}(x)\partial^{3}_{t}\partial_{x}u\|^{2}_{0}+\|\xi\alpha'_{0}(x)\frac{\partial^{3}_{t}u}{x}\|^{2}_{0}\nnm\\
&\leq \mathscr{P}_{0}+M_{0}\left(\mathscr{P}_{4}(K)+1\right)\int^{t}_{0}\left(E(\tau)+E^{2}(\tau)\right)d\tau\nnm\\
&+M_{0}\left(\mathscr{P}_{4}(K)+1\right)E(t)\int^{t}_{0}E(\tau)d\tau.\label{LL3}\\
&\|\chi\alpha^{\frac{3}{2}}_{0}(x)\partial^{3}_{t}\partial^{2}_{x}u\|^{2}_{0}+3\|\chi\alpha^{\frac{1}{2}}_{0}(x)\alpha'_{0}(x)\partial^{3}_{t}\partial_{x}u\|^{2}_{0}\nnm\\
&\leq \mathscr{P}_{0}+M_{0}\left(\mathscr{P}_{4}(K)+1\right)\int^{t}_{0}\left(E(\tau)+E^{2}(\tau)\right)d\tau\nnm\\
&+M_{0}\left(\mathscr{P}_{4}(K)+1\right)E(t)\int^{t}_{0}E(\tau)d\tau.\label{BBCEEE6}\end{align}

\end{lemma}
\begin{proof}
\textbf{Interior Estimates}. By \eqref{Ellip-1}, we have for  $k=3$
\begin{align}&\|\xi\left(\alpha_{0}(x)\partial^{3}_{t}\partial^{2}_{x}u+\alpha'_{0}(x)\partial^{3}_{t}\partial_{x}u-\alpha'_{0}(x)\frac{\partial^{3}_{t}u}{x}\right)\|^{2}_{0}\nnm\\
&\leq M_{0}\left(\|\xi(\alpha'_{0}(x)-\frac{\alpha_{0}(x)}{x})\partial^{3}_{t}\partial_{x}u\|^{2}_{0}+\|\xi\mathcal {F}^{3}\|^{2}_{0}\right).
\nnm\end{align}
Similar to \eqref{EllipticQ-4},
\begin{align} &\|\xi\left[\alpha_{0}(x)\partial^{3}_{t}\partial^{2}_{x}u+\alpha'_{0}(x)\partial^{3}_{t}\partial_{x}u-\alpha'_{0}\frac{\partial^{3}_{t}u}{x}\right]\|^{2}_{0}\nnm\\
&~~\geq \|\xi\alpha_{0}(x)\partial^{3}_{t}\partial^{2}_{x}u\|^{2}_{0}+\frac{1}{4}\|\xi\alpha'_{0}(x)\partial^{3}_{t}\partial_{x}u\|^{2}_{0}+\frac{1}{6}\|\xi\alpha'_{0}(x)\frac{\partial^{3}_{t}u}{x}\|^{2}_{0}\nnm\\
&~~~~- M_{0}\left[\mathscr{P}_{0}+\int^{t}_{0}E(\tau)(\tau)d\tau\right].\label{EN-5}
\end{align}
For the estimate in $\|\xi\mathcal {F}^{3}\|^{2}_{0},$  we only give the estimate of $\|\xi\frac{(1-\frac{1}{c^{2}}\frac{\rho_{0}}{r_{x}}\frac{x}{r}
\Theta)^{3}}{2\Theta^{2}}\Theta_{0}(1+\frac{\rho_{0}}{c^{2}})\frac{x}{\alpha_{c}(x)}\mathcal {J}^{3}_{2}\|^{2}_{0}$ while the other terms can be bounded by the right hands side of \eqref{LL3}. In fact, we have the following estiamte
\begin{align}\|\xi\frac{(1-\frac{1}{c^{2}}\frac{\rho_{0}}{r_{x}}\frac{x}{r}
\Theta)^{3}}{2\Theta^{2}}\Theta_{0}(1+\frac{\rho_{0}}{c^{2}})\frac{x}{\alpha_{c}(x)}\mathcal {J}^{3}_{2}
\|^{2}_{0}\leq\sum^{3}_{l=1}\|\xi\mathcal {J}^{3,l}_{2}\|^{2}_{0}\label{FNU-01}
\end{align}
with
\begin{align}&\|\xi\mathcal {J}^{3,1}_{2}\|^{2}_{0}\leq M_{0}\|\left(\mathscr{K}_{t,x}^{1,0}-|u_{x}|,\alpha_{0}(x)u_{x}\right)\|^{2}_{L^{\infty}}\|\xi\left(\partial^{2}_{t}\partial_{x}u,\frac{\partial^{2}_{t}u}{x}\right)\|^{2}_{0}\nnm\\
&+M_{0}\|(\partial^{2}_{t}
\partial_{x}u,\frac{\partial^{2}_{t}u}{x})\|^{2}_{0}\int^{t}_{0}\|\xi\partial^{2}_{t}\partial_{x}u(\tau)\|^{2}_{0}d\tau\nnm\\
&+M_{0}\|\xi\alpha_{0}(x)\mathscr{K}^{1,1}_{t,x}\|^{2}_{L^{\infty}}\|\left(\partial_{t}\partial_{x}u,\frac{u}{x}\right)(0)\|^{2}_{0}
\nnm\\
&+M_{0}\|\alpha_{0}(x)\mathscr{K}^{1,1}_{t,x}\|^{2}_{L^{\infty}}\int^{t}_{0}\|\xi\left(\partial^{3}_{t}\partial_{x}u,\frac{\partial^{3}_{t}u}{x}\right)(\tau)\|^{2}_{0}d\tau\nnm\\
&+M_{0}\|\left(\mathscr{K}^{1,0}_{t,x}-|u_{x}|,u_{x}(0)\right)\|^{2}_{L^{\infty}}\|\xi\alpha_{0}(x)\left(\partial^{2}_{t}\partial^{2}_{x}u,\partial^{2}_{t}\partial_{x}(\frac{u}{x})\right)\|^{2}_{0}\nnm\\
&+M_{0}\|\alpha_{0}(x)\left(\partial^{2}_{t}\partial^{2}_{x}u,\partial^{2}_{t}\partial_{x}(\frac{u}{x})\right)\|^{2}_{0}\int^{t}_{0}\|\xi\partial_{t}\partial_{x}u(\tau)\|^{2}_{L^{\infty}}d\tau,\nnm
\end{align}
and

 \begin{align}&\|\xi\mathcal {J}^{3,2}_{2}\|^{2}_{0}\leq M_{0}\|\mathscr{K}^{2,0}_{t,x}\|^{2}_{0}\|\xi(\partial_{t}
\partial_{x}u,\frac{\partial_{t}u}{x})(0)\|^{2}_{L^{\infty}}\nnm\\
&+M_{0}\|\xi\mathscr{K}^{2,0}_{t,x}\|^{2}_{L^{\infty}}\int^{t}_{0}\|\left(\partial^{2}_{t}\partial_{x}u,\frac{\partial^{2}_{t}u}{x}\right)(\tau)\|^{2}_{0}d\tau\nnm\\
&+ M_{0}\|\alpha_{0}(x)\mathscr{K}^{2,1}_{t,x}\|^{2}_{0}\|\xi(\partial_{t}
\partial_{x}u,\frac{\partial_{t}u}{x})(0)\|^{2}_{L^{\infty}}\nnm\\
&+M_{0}\|\xi\mathscr{K}^{2,1}_{t,x}\|^{2}_{L^{\infty}}\int^{t}_{0}\|\xi\left(\partial^{2}_{t}\partial_{x}u,\frac{\partial^{2}_{t}u}{x}\right)(\tau)\|^{2}_{0}d\tau\nnm\\
&+M_{0}\|\mathscr{K}^{2,0}_{t,x}\|^{2}_{0}\|\xi\alpha_{0}(x)\left(\partial_{t}
\partial^{2}_{x}u,\partial_{t}\partial_{x}(\frac{u}{x})\right)(0)\|^{2}_{L^{\infty}}\nnm\\
&+M_{0}\|\xi\mathscr{K}^{2,0}_{t,x}\|^{2}_{L^{\infty}}\int^{t}_{0}\|\alpha_{0}(x)\left(\partial^{2}_{t}\partial^{2}_{x}u,\partial^{2}_{t}\partial_{x}(\frac{u}{x})\right)(\tau)\|^{2}_{0}d\tau.\label{LL1}\end{align}
Similarly,
\begin{align}&\|\xi\mathcal {J}^{3,3}_{2}\|^{2}_{0}\leq M_{0}\|\xi\mathscr{K}^{3,0}_{t,x}\|^{2}_{0}\|(
u_{x},\frac{u}{x})(0)\|^{2}_{L^{\infty}}\nnm\\
&+M_{0}\|\xi\mathscr{K}^{3,0}_{t,x}\|^{2}_{0}\int^{t}_{0}\|\xi\left(\partial_{t}\partial_{x}u,\frac{\partial_{t}u}{x}\right)(\tau)\|^{2}_{0}d\tau\nnm\\
&+ M_{0}\|\xi\alpha_{0}(x)\mathscr{K}^{3,1}_{t,x}\|^{2}_{0}\|(
u_{x},\frac{\widetilde {u}}{x})(0)\|^{2}_{L^{\infty}}\nnm\\
&+M_{0}\|\alpha_{0}(x)\mathscr{K}^{3,1}_{t,x}\|^{2}_{0}\int^{t}_{0}\|\xi\left(\partial_{t}\partial_{x}u,\frac{\partial_{t}u}{x}\right)(\tau)\|^{2}_{0}d\tau\nnm\\
&+M_{0}\|\xi\alpha_{0}(x)\mathscr{K}^{3,0}_{t,x}\|^{2}_{L^{\infty}}\|\left(
u_{xx},(\frac{u}{x})_{x}\right)(0)\|^{2}_{L^{\infty}}\nnm\\
&+M_{0}\|\alpha_{0}(x)\mathscr{K}^{3,0}_{t,x}\|^{2}_{0}\int^{t}_{0}\|\xi\left(\partial_{t}\partial^{2}_{x}u,\partial_{t}\partial_{x}(\frac{u}{x})\right)(\tau)\|^{2}_{0}d\tau,
.\label{LL2}
\end{align}
Combing \eqref{EN-5}-\eqref{LL2} gives \eqref{LL3}, where the order of polynomial function $\mathscr{P}(K)$ is $4$.

\textbf{Boundary Estimates}.
Multiplying \eqref{BE-1} by $\chi\sqrt{\alpha_{0}(x)}$ for $k=3$ and taking $L^{2}-$norm, the similar proceeding to the boundary estimates in \eqref{BBEEE6} to show \eqref{BBCEEE6}. This is the end of the proof for Lemma \ref{PPVV4}.\end{proof}

\subsection{Estimates for $E(t)$}\label{E(t)}
\begin{lemma}\label{PPVV56} Let $(r,u,\upsilon,\omega)$ be a classical solution to the free boundary problem \eqref{C-24}--\eqref{L1.2} satisfying \eqref{1.15} on $[0,1]\times[0,T]$. Then, for any $t\in(0,\overline{T}_{1}]$ and $c\geq \overline{c}_{1}$ the estimate in \eqref{DRE} holds.\end{lemma}
\begin{proof}
 We conclude from \eqref{mai5}, \eqref{BEEE7}, \eqref{BBEEE6}, \eqref{EN-4}, \eqref{LL3}, \eqref{BBCEEE6} that
\begin{align}
&E(u)\leq \mathscr{P}_{0}+M_{0}\left(\mathscr{P}_{10}(K)+1\right)\int^{t}_{0}(E(\tau)+E^{2}(\tau))d\tau\nnm\\
&~~~~+M_{0}\left(\mathscr{P}_{4}(K)+1\right)E(t)\int^{t}_{0}E(\tau)d\tau.\label{BBCCEEE6}\end{align}
Similarly, we can obtain from \eqref{C-25} and \eqref{C-26}
\begin{align}
&E(\upsilon)+E(\omega)\leq \mathscr{P}_{0}+M_{0}\left(\mathscr{P}_{10}(K)+1\right)\int^{t}_{0}(E(\tau)+E^{2}(\tau))d\tau\nnm\\
&~~~~+M_{0}\left(\mathscr{P}_{4}(K)+1\right)E(t)\int^{t}_{0}E(\tau)d\tau.\label{BBBCCEEE6}\end{align}
which in cobmination with \eqref{BBCCEEE6} yields
\begin{align}
&E(t)\leq \mathscr{P}_{0}+M_{0}\left(\mathscr{P}_{10}(K)+1\right)\int^{t}_{0}(E(\tau)+E^{2}(\tau))d\tau\nnm\\
&~~~~+M_{0}\left(\mathscr{P}_{4}(K)+1\right)E(t)\int^{t}_{0}E(\tau)d\tau.\label{BBBCCEEE60}\end{align}
However, we can chose $K\leq\sup_{[0,t]}E(t), $ then we use the Gronwall inequality to obtain \eqref{DRE}.\end{proof}
\section{Existence results for the case $\gamma=2$}\label{sec-5}
In this section, we prove the existence of classical solution to the free boundary value problem \eqref{C-24}- \eqref{L1.2}
by using a degenerate hyperbolic regularization based on the higher order Hardy type inequality.

\par In order to obtain the existence, we use the following degenerate parabolic approximation \cite{Coutand1,GuXM1}:
\begin{align}
&a_{11}\alpha_{c}(x)(u_{t}-\frac{\upsilon^{2}}{r})+\left(\frac{\alpha^{2}_{c}(x)}{x}\frac{x\Theta^{2}}{rr^{2}_{x}(1-\frac{1}{c^{2}}\frac{\rho_{0}}{r_{x}}\frac{x}{r}\Theta)^{2}}\right)_{x}-\frac{\alpha^{2}_{c}(x)}{x^{2}}\frac{x^{2}\Theta^{2}}{r^{2}r_{x}(1-\frac{1}{c^{2}}\frac{\rho_{0}}{r_{x}}\frac{x}{r}\Theta)^{2}}\nnm\\
&+\frac{xa_{12}}{c^{2}rr^{2}_{x}}\frac{\alpha^{2}_{c}(x)}{x}
(u_{x}+\frac{u}{r}r_{x})u
=2a_{11}\mu\left(\alpha_{c}(x)u_{xx}+(2\alpha'(x)-\frac{\alpha_{c}(x)}{x})u_{x}-\alpha'(x)\frac{u}{x}\right).\label{C-45669}
\end{align}
\par
We give the following two important lemmas in \cite{Coutand1,GuXM1}.  Lemma \ref{lemma5.1} implies the higher order Hardy type inequality and  will be used to construct the approximation solution and Lemma \ref{lemma5.2} is to obtain estimates independent of $\mu.$

\begin{lemma}\label{lemma5.1}Let $s\geq1$ be a given integer, and suppose that
\begin{equation*} u\in H^{s}(I)\bigcap H^{1}_{0}(I)\end{equation*}
with $I=(0,1),$  and $d$  is the distance function to $\partial I,$ then $\frac{u}{d}\in H^{s-1}(I)$ with
\begin{equation*}\|\frac{u}{d}\|_{s-1}\leq\|u\|_{s}.\end{equation*}
\end{lemma}
\begin{lemma}\label{lemma5.2}Let $\mu>0$ ~and ~$g\in L^{\infty}(0,T; H^{s}(I))$~with ~$I=(0,1)$~ be given, and let ~$f\in H^{1}(0,T; H^{s}(I))$ be such that
\begin{equation*}f+\mu f_{t}=g,~\text{in}~ (0,T)\times I.\end{equation*}
Then,
\begin{equation*}\|f\|_{L^{\infty}(0,T;H^{s}(I))}\leq C\max\{\|f(0)\|_{s},\|g\|_{L^{\infty}(0,T; H^{s}(I))}\}.\end{equation*}
\end{lemma}
\par The next step is to prove the existence of solutions for the regularized problem \eqref{C-45669}, \eqref{C-25} and \eqref{C-26} with \eqref{L1.2} by using fixed point scheme \cite{Coutand1,GuXM1}.  We assume $\rho_{0},u_{0},\upsilon_{0},\omega_{0}\in C^{\infty}([0,1]),$ which can be done by the elliptic regularization as same as in \cite{Coutand1,GuXM1}, and satisfy \eqref{C-3}--\eqref{positive1} for any $c\geq c_{0}$ with $c_{0}~$ only depending on $\|\rho_{0}\|_{L^{\infty}},\|{\mathbf{v}}_{0}\|^{2}_{L^{\infty}}.$
\par Set $\mathcal {X}_{T}$ be the following Hilbert space: For any $(\overline{u},\overline{\upsilon},\overline{\omega})\in\mathcal {X}_{T}$ satisfying:
\begin{align}&\frac{\alpha_{0}(x)}{\sqrt{x}}\partial^{6}_{t}\overline{u}\in L^{2}(0,T;H^{1}(0,1)),\frac{\alpha_{0}(x)}{\sqrt{x}}\frac{\partial^{6}_{t}\overline{u}}{x}\in L^{2}(0,T;L^{2}(0,1)),\nnm\\
&\frac{\alpha_{0}(x)}{\sqrt{x}}\overline{u}\in W^{5,2}(0,T;H^{2}(0,1))\cap W^{4,2}(0,T;H^{3}(0,1)),\nnm\\
&\overline{u}\in W^{5,2}(0,T;H^{1}(0,1))\cap W^{4,2}(0,T;H^{2}(0,1)),\nnm\\
&\frac{\overline{u}}{x}\in W^{5,2}(0,T;L^{2}(0,1))\cap W^{4,2}(0,T;H^{1}(0,1)),\nnm\\
&\xi\alpha_{0}(x)\frac{\overline{u}}{x}\in W^{5,2}(0,T;H^{1}(0,1))\cap W^{4,2}(0,T;H^{2}(0,1)),\nnm\\
&\xi\alpha_{0}(x)\overline{u}\in W^{5,2}(0,T;H^{3}(0,1)),\nnm\\
&\alpha_{0}(x)\overline{\upsilon}\in L^{2}(0,T;H^{1}(0,1)),~\frac{\alpha_{0}(x)}{\sqrt{x}}\overline{\upsilon}\in W^{5,2}(0,T;H^{2}(0,1)), \nnm\\
&\overline{\upsilon}\in W^{5,2}(0,T;H^{1}(0,1)),~\frac{\overline{\upsilon}}{x} \in W^{4,2}(0,T;H^{1}(0,1)),\nnm\\
&\alpha_{0}(x)\overline{\omega}\in L^{2}(0,T;H^{1}(0,1)),~\frac{\alpha_{0}(x)}{\sqrt{x}}\overline{\omega}\in W^{5,2}(0,T;H^{2}(0,1)),\nnm\\ &\overline{\omega}\in W^{5,2}(0,T;H^{1}(0,1)),~\frac{\overline{\omega}}{x} \in W^{4,2}(0,T;H^{1}(0,1)).\nnm
\end{align}
The function space $\mathcal {X}_{T}$ is equipped with the following natural Hilbert norm: \begin{equation}\|(\overline{u},\overline{\upsilon},\overline{\omega})\|_{\mathcal {X}_{T}}=\|\overline{u}\|_{\mathcal {X}_{T}}+\|\overline{\upsilon}\|_{\mathcal {X}_{T}}+\|\overline{\omega}\|_{\mathcal {X}_{T}}\nnm\end{equation} with
\begin{align}\|\overline{u}\|_{\mathcal {X}_{T}}&:=\|\frac{\alpha_{0}(x)}{\sqrt{x}}\partial^{6}_{t}\overline{u}\|_{L^{2}(0,T;H^{1}(0,1))}+\|\frac{\alpha_{0}(x)}{\sqrt{x}}\frac{\partial^{6}_{t}\overline{u}}{x}\|_{L^{2}(0,T;L^{2}(0,1))}\nnm\\
&+\|\overline{u}\|_{W^{5,2}(0,T;H^{1}(0,1))}+\|\overline{u}\|_{W^{4,2}(0,T;H^{2}(0,1))}\nnm\\
&+\|\frac{\alpha_{0}(x)}{\sqrt{x}}\overline{u}\|_{W^{5,2}(0,T;H^{2}(0,1))}+\|\frac{\alpha_{0}(x)}{\sqrt{x}}\overline{u}\|_{W^{4,2}(0,T;H^{3}(0,1))}\nnm\\
&+\|\frac{\overline{u}}{x}\|_{W^{5,2}(0,T;L^{2}(0,1))}+\|\frac{\overline{u}}{x}\|_{W^{4,2}(0,T;H^{1}(0,1))}\nnm\\
&+\|\xi\alpha_{0}(x)\frac{\overline{u}}{x}\|_{W^{5,2}(0,T;H^{1}(0,1))}+\|\xi\alpha_{0}(x)\frac{\overline{u}}{x}\|_{W^{4,2}(0,T;H^{2}(0,1))}\nnm\\
&+\|\xi\alpha_{0}(x)u\|_{W^{5,2}(0,T;H^{3}(0,1))},\nnm\\
\|\overline{\upsilon}\|_{\mathcal {X}_{T}}&:=\|\alpha_{0}(x)\partial^{6}_{t}\overline{\upsilon}\|_{L^{2}(0,T;H^{1}(0,1))}+\|\overline{\upsilon}\|_{W^{5,2}(0,T;H^{1}(0,1)}\nnm\\
&+\|\frac{\alpha_{0}(x)}{\sqrt{x}}\overline{\upsilon}\|_{W^{5,2}(0,T;H^{2}(0,1))}+\|\frac{\overline{\upsilon}}{x}\|_{W^{4,2}(0,T;H^{1}(0,1))},\nnm\end{align}
and
\begin{align}\|\overline{\omega}\|_{\mathcal {X}_{T}}&:=\|\alpha_{0}(x)\partial^{6}_{t}\overline{\omega}\|_{L^{2}(0,T;H^{1}(0,1))}+\|\overline{\omega}\|_{W^{5,2}(0,T;H^{1}(0,1)}\nnm\\
&+\|\frac{\alpha_{0}(x)}{\sqrt{x}}\overline{\omega}\|_{W^{5,2}(0,T;H^{2}(0,1))}+\|\frac{\overline{\omega}}{x}\|_{W^{4,2}(0,T;H^{1}(0,1))},\nnm
\end{align}

\par Set $\partial^{a_{1}}_{t}u(x,0)=u_{a_{1}},~\partial^{a_{2}}_{t}\upsilon(x,0)=\upsilon_{a_{2}},~\partial^{a_{2}}_{t}\omega(x,0)=\omega_{a_{2}}$ (see\eqref{compat-1}-\eqref{compat-3}).
For the suitably large positive constant $M_{0},$  which will be determined later, the following closed, bounded and convex subset of $\mathcal {X}_{T}$ is given by:
\begin{equation*}
\begin{split}
\mathcal {C}_{T}(M_{0}):=\big\{(&\overline{u},\overline{\upsilon},\overline{\omega})\in\mathcal {X}_{T} :\partial^{a_{1}}_{t}\overline{u}|_{t=0}=u_{a_{1}},\partial^{a_{2}}_{t}\overline{\upsilon}|_{t=0}=u_{a_{2}},\partial^{a_{2}}_{t}\overline{u}|_{t=0}=\omega_{a_{2}},\\
&a_{1}=0,1,...,6,a_{2}=0,1,...,5,\\&\|(\overline{u},\overline{\upsilon},\overline{\omega})\|_{\mathcal {X}_{T}}\leq M_{0}\big\}.
\end{split}
\end{equation*} which is non-empty set if $M_{0}$ is large enough.
\par For any $c\geq c_{0}$ and $(\overline{u},\overline{\upsilon},\overline{\omega})\in \mathcal {C}_{T}(M_{0}),$ we define $\overline{r}=x+\int^{t}_{0}\overline{u}(x,\tau)d\tau$
and the operator $\Phi:(\overline{u},\overline{\upsilon},\overline{\omega})\rightarrow(u,\upsilon,\omega)$ by solving:
\begin{align}&xu_{t}-2\mu\left(\alpha_{c}(x)u_{xx}+(2\alpha'(x)-\frac{\alpha_{c}(x)}{x})u_{x}-\alpha'(x)\frac{u}{x}\right)=\frac{x}{\alpha_{c}(x)}\overline{\mathcal {F}_{1}},\label{exist-1}\\
&\upsilon=\upsilon_{0}+\int^{t}_{0}\overline{\mathcal {F}_{2}}(x,\tau)d\tau,~~\omega=\omega_{0}+\int^{t}_{0}\overline{\mathcal {F}_{3}}(x,\tau)d\tau,\label{exist-2}
\end{align}
where\begin{align}
&\overline{\mathcal {F}_{1}}:=x\frac{\overline{\upsilon}^{2}}{r}-\frac{x}{\alpha_{c}(x)}\frac{1}{\overline{a}_{11}}\left(\frac{\alpha^{2}_{c}(x)}{x}\frac{x\overline{\Theta}^{2}}{rr^{2}_{x}(1-\frac{1}{c^{2}}\frac{\overline{\rho}_{0}}{r_{x}}\frac{x}{r}\overline{\Theta})^{2}}\right)_{x}\nnm\\
&+\frac{x}{\alpha_{c}(x)}\frac{1}{\overline{a}_{11}}\frac{\alpha^{2}_{c}(x)}{x^{2}}\frac{x^{2}\overline{\Theta}^{2}}{r^{2}r_{x}(1-\frac{1}{c^{2}}\frac{\overline{\rho}_{0}}{r_{x}}\frac{x}{r}\overline{\Theta})^{2}}-\frac{x}{\alpha_{c}(x)}\frac{1}{\overline{a}_{11}}\frac{x\overline{a}_{12}}{c^{2}rr^{2}_{x}}
\frac{\alpha^{2}_{c}(x)}{x}(\overline{u}_{x}+\frac{\overline{u}}{r}r_{x})\overline{u},\nnm\end{align}
\begin{align}
&\overline{\mathcal {F}_{2}}:=-\frac{\overline{u}}{r}\overline{\upsilon}+\frac{\overline{b}_{11}}{c^{2}}\frac{\alpha_{c}(x)}{x}(\overline{u}_{x}+\frac{\overline{u}}{r}r_{x})\overline{\upsilon}-\frac{\overline{b}_{12}}{c^{2}}(\overline{u}_{t}-\frac{\overline{\upsilon}^{2}}{r})\overline{u}~\overline{\upsilon},\nnm\\
&\overline{\mathcal {F}_{3}}:\frac{\overline{b}_{11}}{c^{2}}\frac{\alpha_{c}(x)}{x}(\overline{u}_{x}+\frac{\overline{u}}{r}r_{x})\overline{\omega}-\frac{\overline{b}_{12}}{c^{2}}(\overline{u}_{t}-\frac{\overline{\upsilon}^{2}}{r})
\overline{u}~\overline{\omega},\nnm
\end{align}
where $\overline{a}_{ij}$~and~$\overline{b}_{ij}$ are given by \eqref{CCC1} with $\gamma=2$ , respectively, and take the value at  $(\overline{u}^{2},\overline{\upsilon}^{2},\overline{\omega}^{2},\overline{r}_{x},\frac{x}{\overline{r}},\rho_{0}).$

\par We first consider the sixth time-differentiated problem. In order to use the Hardy type inequality in Lemma \ref{lemma5.1}, introducing the new variable $X=\alpha_{0}(x)\partial^{6}_{t}u$ and taking $\partial^{6}_{t}$ over \eqref{exist-1} yields
\begin{align}\frac{x}{\alpha_{c}(x)}X_{t}-2\mu&\left[X_{xx}-\frac{X}{x^{2}}-\left((\frac{\alpha_{c}(x)}{x})_{x}+\alpha''(x)\right)\frac{X}{\alpha_{c}(x)}\right]=\frac{x}{\alpha_{c}(x)}\partial^{6}_{t}\overline{\mathcal {F}_{1}}.\label{exist-3}\\
& X=0,~\text{on}~~ (0,1)\times [0,T],\label{exist-40}\\
&X|_{t=0}=\rho_{0}u_{6}~\text{in} ~~(0,1),\label{exist-50}
\end{align}
Because the existence of this problem can be  obtained by Galerkin scheme as \cite{Coutand1,GuXM1}, thus we omit it here and only give the priori estimates $(u,\upsilon,\omega)\in \mathcal {C}_{T}(M_{0})$.

\par Multiplying by $\frac{X}{x}$ and integrating over $(0,t)\times(0,1),$  the Gronwall inequality implies for sufficiently small $T,$
\begin{align}\int^{1}_{0}\frac{X^{2}}{\alpha_{c}(x)}dx+\mu\int^{T}_{0}\int^{1}_{0}\left(\frac{X^{2}_{x}}{x}+\frac{X^{2}}{x^{3}}\right)dx\leq \mathscr{P}_{0}+C\mathscr{P}(M_{0}).\nnm\end{align}
which implies $\frac{\alpha_{0}(x)}{\sqrt{x}}\partial^{6}_{t}\overline{u}\in L^{2}(0,T;H^{1}(0,1)),\frac{\alpha_{0}(x)}{\sqrt{x}}\frac{\partial^{6}_{t}\overline{u}}{x}\in L^{2}(0,T;L^{2}(0,1)),$
\par Set
\begin{align}&Z:=\int^{t}_{0}X(x,\tau)d\tau+\alpha_{c}(x)u_{5},~~W=\int^{t}_{0}Z(x,\tau)d\tau+\alpha_{c}(x)u_{4},\nnm\\
&Y:=\int^{t}_{0}\int^{t_{1}}_{0}\int^{t_{2}}_{0}\int^{t_{3}}_{0}\int^{t_{4}}_{0}Z(x,\tau)d\tau_{4}d\tau_{3}d\tau_{2}d\tau_{1}d\tau+\sum^{5}_{i=0}\alpha_{c}(x)\frac{u_{i}}{i!}.\nnm
\end{align}
\par The remaining steps of the proof on the existence of solutions for the problem \eqref{exist-3}-\eqref{exist-50} can be similarly obtained by that in \cite{GuXM1} , besides the estimates of the original point $x=0$ can be constructed by the similar
proceeding as in Section \ref{sec-4}. Then, we can obtain the existence result for the original problem \eqref{C-24}- \eqref{L1.2} based on the priori estimates constructed in Section \ref{sec-4} and Lemma \ref{lemma5.2}.
\section{Uniqueness results for the case $\gamma=2$}\label{sec-6}
In this section, we verify the uniqueness of classical solutions to the free boundary problem \eqref{C-24}- \eqref{L1.2} obtained in section \ref{sec-5} by using the energy method.
\begin{lemma}\label{uniqueness}\textbf{(Uniqueness)}Let $(u_{1},\upsilon_{1},\omega_{1})$~and ~$(u_{2},\upsilon_{2},\omega_{2})$ ~be two classical solutions of the problem \eqref{C-24}- \eqref{L1.2} in $[0,1]\times[0,T]$ with
\begin{equation}r_{1}=x+\int^{1}_{0}u_{1}(x,\tau)d\tau,~r_{2}=x+\int^{1}_{0}u_{2}(x,\tau)d\tau,~\end{equation}
satisfying
\eqref{1.15} and estimates \eqref{XXX1}
-\eqref{BP-6} in Lemma \ref{PPP1} for any $c\geq\overline{c}$ and $K$ being a positive constant depending on $\overline{c}$. If $(u_{1},\upsilon_{1},\omega_{1})(x,0)=(u_{2},\upsilon_{2},\omega_{2})(0,x)~$for$~x\in[0,1],$
then there exist the positive constants  $T^{*}, ~c^{*}~$ such that $0<T^{*}<T~,~c^{*}\geq \overline{c}~$and
\begin{equation}(u_{1},\upsilon_{1},\omega_{1})(x,t)=(u_{2},\upsilon_{2},\omega_{2})(x,t)\nnm\end{equation}
for any $c\geq c^{*}~$and$~(x,t)\in[0,1]\times[0,T^{*}].$

\end{lemma}
 \begin{proof}
 Set
\begin{align}&R:=r_{1}-r_{2},~R_{t}=U:=u_{1}-u_{2},~V:=\upsilon_{1}-\upsilon_{2},~W:=\omega_{1}-\omega_{2},\nnm\\
&\Theta_{k}:=\sqrt{1-\frac{(u^{2}_{k}+\upsilon^{2}_{k}+\omega^{2}_{k})}{c^{2}}},~a^{k}_{ij}:=a_{ij}(u^{2}_{k},\upsilon^{2}_{k},\omega^{2}_{k},\frac{x}{r_{k}},r_{k,x})\left(\text{see} ~\eqref{CCC0}-\eqref{CCC1}\right),\nnm\\
&b^{k}_{ij}:=b_{ij}(u^{2}_{k},\upsilon^{2}_{k},\omega^{2}_{k},\frac{x}{r_{k}},r_{k,x}),~k,i,j=1,2.\left(\text{see}~ \eqref{CCC0}-\eqref{CCC1}\right)\nnm
\end{align}
\par Due to \eqref{C-24}, the fundamental theorem of calculus and  straightforward computation imply
\begin{align}\alpha_{c}(x)&a^{1}_{11}U_{t}-\left\{\frac{\alpha^{2}_{c}(x)}{x}\left[2A_{1}R_{x}+A_{2}\frac{R}{x}\right]\right\}_{x}\nnm\\
&+\frac{\alpha^{2}_{c}(x)}{x^{2}}\left(A_{3}R_{x}+2A_{4}\frac{R}{x}\right)+G(U,V,W,\frac{R}{x},R_{x})=0,
\label{VE}\end{align}
where
\begin{align}A_{1}&:=\int^{1}_{0}\frac{x\Theta^{2}_{1}}{r_{1}r^{3}_{x}(1-\frac{1}{c^{2}}\frac{\rho_{0}}{r_{x}}\frac{x}{r_{1}}\Theta_{1})^{3}}|_{r_{x}=r_{2,x}+\mu(r_{1,x}-r_{2,x})}d\mu,\label{AA1}\\
A_{2}&:=\int^{1}_{0}\frac{x\Theta^{2}_{1}(1+\frac{1}{c^{2}}\frac{\rho_{0}}{\partial_{x}r_{2}}\frac{x}{r}\Theta_{1})}{r^{2}\partial_{x}r_{2}(1-\frac{1}{c^{2}}\frac{\rho_{0}}{r_{2,x}}\frac{x}{r}\Theta_{1})^{3}}|_{r=r_{2}+\mu(r_{1}-r_{2})}d\mu,\nnm\\
A_{3}&:=\frac{1}{\Theta_{1}+\Theta_{2}}\int^{1}_{0}\frac{x\Theta^{2}}{r_{2}r_{2x}(1-\frac{1}{c^{2}}\frac{\rho_{0}}{r_{2x}}\frac{x}{r_{2}}\Theta)^{3}}|_{\Theta=\Theta_{2}+\mu(\Theta_{1}-\Theta_{2})}d\mu,\nnm\\
A_{4}&:=\int^{1}_{0}\frac{x\Theta^{2}_{1}(1+\frac{1}{c^{2}}\frac{\rho_{0}}{r_{x}}\frac{x}{r_{1}}\Theta_{1})}{r^{2}_{1}r_{x}(1-\frac{1}{c^{2}}\frac{\rho_{0}}{r_{x}}\frac{x}{r_{1}}\Theta_{1})^{3}}|_{r_{x}=r_{2,x}+\mu(r_{1,x}-r_{2,x})}d\mu,\nnm\\
A_{5}&:=\int^{1}_{0}\frac{x^{3}\Theta^{2}_{1}}{r^{3}\partial_{x}r_{2}(1-\frac{1}{c^{2}}\frac{\rho_{0}}{r_{2,x}}\frac{x}{r}\Theta_{1})^{3}}|_{r=r_{2}+\mu(r_{1}-r_{2})}d\mu,\nnm
\end{align}
and $G(U,V,W,\frac{R}{x},R_{x})$ satisfies
\begin{equation}|G(U,V,W,\frac{R}{x},R_{x})|\leq C\alpha_{c}(x)\left(|U|+|V|+|W|+\frac{R}{x}+|R_{x}|\right).\nnm\end{equation}

\par Multiplying \eqref{VE} by $U$ and integration over $(0,t)\times(0,1),$ then the integrating by parts and Cauchy-Schwarz inequality show
\begin{align}\int^{1}_{0}\alpha^{2}_{c}(x)&a^{1}_{11}\frac{U^{2}}{2}dx+\int^{1}_{0}\frac{\alpha^{2}_{c}(x)}{x}\left(A_{1}R^{2}_{x}+A_{2}\frac{R}{x}R_{x}+A_{5}\frac{R^{2}}{x^{2}}\right)dx\nnm\\
&\leq C(K)\int^{t}_{0}\int^{1}_{0}\alpha_{c}(x)(U^{2}+V^{2}+W^{2})dxd\tau\nnm\\
&+C(K)\int^{t}_{0}\int^{1}_{0}\frac{\alpha^{2}_{c}(x)}{x}\left(R^{2}_{x}+\frac{R^{2}}{x^{2}}+U^{2}_{x}\right)dxd\tau.\label{uniq-1}
\end{align}
Substituting  $V$ and  $W$ into \eqref{C-25} and \eqref{C-26} respectively, we also have
\begin{align}\int^{1}_{0}&\alpha_{c}(x)\left(V^{2}+W^{2}\right)dx\leq C(K)\int^{t}_{0}\int^{1}_{0}\alpha_{c}(x)(U^{2}+R^{2})dxd\tau\nnm\\
&\leq C(K)\int^{t}_{0}\int^{1}_{0}\frac{\alpha^{2}_{c}(x)}{x}\left(R^{2}_{x}+\frac{R^{2}}{x^{2}}+U^{2}_{x}+U^{2}_{t}+\frac{U^{2}}{x^{2}}\right)dxd\tau\label{uniq-2}
\end{align}
\par Differentiating \eqref{VE} with respect to $t$ and multiplying the resulting equations by $U_{t}$, similar to \eqref{uniq-1},
\begin{align}&\int^{1}_{0}\alpha_{c}(x)a^{1}_{11}U^{2}_{t}d_{x}+\int^{1}_{0}\frac{\alpha^{2}_{c}(x)}{x}\left[\Psi U^{2}_{x}+A_{2}\frac{U}{x}U_{x}+\frac{5}{6}A_{5}\frac{U^{2}}{x^{2}}\right]dx\nnm\\
&\leq-\int^{t}_{0}\int^{1}_{0}\frac{\alpha^{2}_{c}(x)}{x}(A_{4}-A_{2})\frac{U_{t}}{x}U_{x}dxd\tau+C(K)\int^{t}_{0}\int^{1}_{0}\alpha_{c}(x)(U^{2}+V^{2}+W^{2}+R^{2})dx.\nnm\\
&+C(K)\int^{1}_{0}\frac{\alpha^{2}_{c}(x)}{x}(R^{2}_{x}+\frac{R^{2}}{x^{2}})dx+C(K)\int^{t}_{0}\int^{1}_{0}\alpha_{c}(x)(U^{2}+V^{2}+W^{2}+R^{2})dxd\tau\nnm\\
&+C(K)\int^{t}_{0}\int^{1}_{0}\frac{\alpha^{2}_{c}(x)}{x}(R^{2}_{x}+\frac{R^{2}}{x^{2}}+U^{2}+U^{2}_{x}+\frac{U^{2}}{x^{2}})dx\label{uniq-30}
\end{align}
where
\begin{equation}\Psi=\left[\frac{11}{12}A_{1}
-\frac{1}{c^{4}}\frac{\alpha_{c}(x)}{x}b^{1}_{11}A_{3}(\upsilon_{1}+\upsilon_{2})\upsilon_{2}-\frac{1}{c^{4}}\frac{\alpha_{c}(x)}{x}b^{1}_{11}A_{3}(w_{1}+w_{2})w_{2}\right].\nnm\end{equation}
A straightforward computation implies
\begin{equation}|A_{4}-A_{2}|\leq C(K)(|R|+|U|)\nnm,\end{equation}
 which implies
 \begin{equation}-\int^{t}_{0}\int^{1}_{0}\frac{\alpha^{2}_{c}(x)}{x}(A_{4}-A_{2})\frac{U_{t}}{x}U_{x}dxd\tau\leq C(K)\int^{t}_{0}\int^{1}_{0}\frac{\alpha^{2}_{c}(x)}{x}(R^{2}_{x}+U^{2}+U^{2}_{x}+\frac{U^{2}}{x})dx.\label{uniq-32}\end{equation}

\par  We deal with the second term  on the left hand side of \eqref{uniq-30}. It is easy to obtain that there exist the positive constants $T^{*}_{1}$ and  $\varepsilon_{T^{*}_{1}}$ such that $~0<T^{*}_{1}<T~$ and
for any $0<t\leq T^{*}_{1},$
\begin{align}\frac{1}{1+\varepsilon_{T^{*}_{1}}}\leq\frac{x}{r_{i}}\leq\frac{1}{1-\varepsilon_{T^{*}_{1}}},~1-\varepsilon_{T^{*}_{1}}\leq r_{i,x}\leq 1+\varepsilon_{T^{*}_{1}},i=1,2 \label{JQ1}\end{align}
where $\lim_{T^{*}_{1}\rightarrow0}\varepsilon_{T^{*}_{1}}=0.$ Thus, for any $r=r_{2,x}+\mu(r_{1,x}-r_{2,x}),~r=r_{2}+\mu(r_{1}-r_{2})$
\begin{equation}\frac{-2\varepsilon_{T^{*}_{1}}}{(1-\varepsilon_{T^{*}_{1}})(1+\varepsilon_{T^{*}_{1}})}+\frac{x}{r}\leq\frac{x}{r_{1}}\leq\frac{x}{r}+\frac{2\varepsilon_{T^{*}_{1}}}{(1-\varepsilon_{T^{*}_{1}})(1+2\varepsilon_{T^{*}_{1}})},~r_{2x}-\varepsilon_{T^{*}_{1}}\leq r_{x}\leq r_{2x}+2\varepsilon_{T^{*}_{1}}.\nnm\end{equation}
From \eqref{AA1},
$A_{1}\geq F(\mu,\varepsilon_{T^{*}_{1}})$
where
\begin{align}
&F(\mu,\varepsilon_{T^{*}_{1}}):=\int^{1}_{0}\left(\frac{x}{r}-\frac{2\varepsilon_{T^{*}_{1}}}{(1-\varepsilon_{T^{*}_{1}})(1+\varepsilon_{T^{*}_{1}})}\right)
\nnm\\ &\times\frac{\Theta^{2}_{1}}{(r_{2,x}+2\varepsilon_{T^{*}_{1}})^{3}}\frac{1}{\left[1-\frac{1}{c^{2}}\frac{\rho_{0}}{(r_{2,x}+2\varepsilon_{T^{*}_{1}})}(\frac{x}{r}-\frac{2\varepsilon_{T^{*}_{1}}}{(1-\varepsilon_{T^{*}_{1}})(1+\varepsilon_{T^{*}_{1}})})\right]^{3}}|_{r_{x}=r_{2}+\mu(r_{1}-r_{2})}d\mu.\nnm
\end{align}
It is easily to obtain
\begin{equation}\lim_{T^{*}_{1}\rightarrow0}F(\mu,\varepsilon_{T^{*}_{1}})=A^{*}_{1},\nnm\end{equation}
with
\begin{equation}A^{*}_{1}:=\int^{1}_{0}\frac{x\theta^{2}_{1}}{rr^{3}_{2,x}(1-\frac{1}{c^{2}}\frac{\rho_{0}}{r_{2x}}\frac{x}{r}\theta_{1})^{3}}|_{r=r_{2}+\mu(r_{1}-r_{2})}d\mu\nnm\end{equation}
Then, there exist a positive constant $T^{*}_{2}$  such that $0<T^{*}_{2}<T^{*}_{1}$~for any $0<t\leq T^{*}_{2},$
\begin{align}&\Psi U^{2}_{x}+A_{2}\frac{U}{x}U_{x}+\frac{5}{6}A_{5}\frac{U^{2}}{x^{2}}\geq(\frac{1}{2}+\frac{1}{12})A^{*}_{1}U^{2}_{x}+A_{2}\frac{U}{x}U_{x}+\frac{5}{6}A_{5}\frac{U^{2}}{x^{2}}\nnm\\
&+\left(\frac{1}{4}A_{1}
-\frac{1}{c^{4}}\frac{\alpha_{c}(x)}{x}b^{1}_{11}A_{3}(\upsilon_{1}+\upsilon_{2})\upsilon_{2}-\frac{1}{c^{4}}\frac{\alpha_{c}(x)}{x}b^{1}_{11}A_{3}(w_{1}+w_{2})w_{2}\right)U^{2}_{x}.\label{uniq-33}
\end{align}
A simple computation implies
\begin{align}&(\frac{1}{2}+\frac{1}{12})A^{*}_{1}U^{2}_{x}+A_{2}\frac{U}{x}U_{x}+\frac{5}{6}A_{5}\frac{U^{2}}{x^{2}}\nnm\\
&\geq\int^{1}_{0}\frac{x}{rr_{2,x}}\frac{\Theta^{2}_{1}}{(1-\frac{1}{c^{2}}\frac{\rho_{0}}{r_{2,x}}\frac{x}{r}\Theta_{1})^{3}}\left[(\frac{1}{12}-\frac{1}{c^{2}}\frac{\rho_{0}}{r_{2,x}}\frac{x}{r}\Theta_{1})U^{2}_{x}+(\frac{1}{3}-\frac{1}{c^{2}}\frac{\rho_{0}}{r_{2,x}}\frac{x}{r}\Theta_{1})\frac{x^{2}}{r^{2}}\frac{U^{2}}{x^{2}}\right]d\mu.\nnm
\end{align}
where  $r=r_{2}+\mu(r_{1}-r_{2}).$ Then, there exists the positive constants $c^{*}_{1}\geq \overline{c}~$ and $~M_{1}$ such that

\begin{align}&(\frac{1}{2}+\frac{1}{12})A^{*}_{1}U^{2}_{x}+A_{2}\frac{U}{x}U_{x}+\frac{5}{6}A_{0}\frac{U^{2}}{x^{2}}\geq M_{1}(U^{2}_{x}+\frac{U^{2}}{x^{2}}).\label{uniq-34}\end{align}
Similarly, there exist the positive constants $c^{*}_{2}\geq \overline{c}~$ and $~M_{2}$ such that
\begin{align}&\left(\frac{1}{4}\widetilde{A}_{1}
-\frac{1}{c^{4}}\frac{\alpha_{c}(x)}{x}b^{1}_{11}A_{3}(\upsilon_{1}+\upsilon_{2})\upsilon_{2}-\frac{1}{c^{4}}\frac{\alpha_{c}(x)}{x}b^{1}_{11}A_{3}(w_{1}+w_{2})w_{2}\right)U^{2}_{x}
\nnm\\
&~~~~~~~~~~~~~~~~~~~\geq M_{2}U^{2}_{x}.\label{uniq-35}\end{align}
Finally, we have for any $0<t<\min\{T_{1}^{*},T_{2}^{*}\}~$ and $c\geq\max\{c_{1}^{*},c_{2}^{*}\}$
 \begin{align}\Psi U^{2}_{x}+A_{2}\frac{U}{x}U_{x}+\frac{5}{6}A_{5}\frac{U^{2}}{x^{2}}\geq M(U^{2}_{x}+\frac{U^{2}}{x^{2}}),\label{uniq-333}\end{align}
where $M=\min\{M_{1},M_{2}\}.~$ By \eqref{uniq-1}-\eqref{uniq-32} and \eqref{uniq-333},  we apply the Gronwall inequality to obtain $U=R=0.$
This completes the proof.\end{proof}

\section{Non-relativistic limits for the case $\gamma=2$}\label{sec-7}
In this section, we consider the non-relativistic limits of smooth solutions to the free boundary problem \eqref{C-24}- \eqref{L1.2} obtained in section \ref{sec-5}.
\begin{lemma}Let
$(r^{c},u^{c},\upsilon^{c},\omega^{c})$ be a smooth solution to the problem \eqref{C-24}-\eqref{L1.2} satisfying Theorem \ref{THM-1} for any $~(x,t)\in[0,1]\times[0,T_{c_{0}}]$ and
$(r ,u,\upsilon,\omega)$ be a smooth solution of the problem \eqref{LCE} with \eqref{L1.2} satisfying \eqref{DRE} for any$~(x,t)\in[0,t]\times[0,T_{0}]$. Then, there exists a positive constant $T^{*}_{0}$ such that $0<T^{*}_{0}\leq\min\{T_{c_{0}},T_{0}\},$ and for any $0<t<T^{*}_{0}~$ and $c\geq c_{0}$
the estimate in \eqref{rate-non} holds.

\end{lemma}
\begin{proof} Based on the uniform estimates in \eqref{DRE}, the convergence of $(r^{c},u^{c},\upsilon^{c},\omega^{c})$ to $(r ,u,\upsilon,\omega)$ can be obtained by the standard procedure as mentioned in Remark \ref{LFF}. Thus, it is enough to prove \eqref{rate-non} in order to the proof of this Lemma.
\par
Set $U^{c}=u^{c}-u, V^{c}=\upsilon^{c}-\upsilon,~W^{c}=\omega^{c}-\omega,~R^{c}=r^{c}-r.$
Subtracting $\eqref{LCE}_{1}$ from \eqref{C-24} and using Taylor's expansion, we have
\begin{align}&\alpha_{0}(x)\left(U^{c}_{t}-\frac{\upsilon^{c}-\upsilon}{r^{c}}+\frac{\upsilon^{2}}{rr_{c}}R^{c}\right)-\left\{\frac{\alpha^{2}_{0}(x)}{x}
\left(\frac{1}{(r^{c}_{x})^{2}}\frac{x}{r}\frac{x}{r_{c}}\frac{R^{c}}{x}+\frac{x}{r}\frac{(r^{c}_{x}+r_{x})}{r^{2}_{x}(r^{c}_{x})^{2}}R^{c}_{x}+O(c^{-2})\right)\right\}_{x}\nnm\\
&+\frac{\alpha^{2}_{0}(x)}{x^{2}}\frac{1}{r^{c}_{x}}\frac{x^{2}}{r^{2}}\frac{(r+r^{c})}{r^{c}}\frac{x}{r^{c}}\frac{R^{c}}{x}+\frac{\alpha^{2}_{0}(x)}{x^{2}}\frac{x^{2}}{r^{2}}\frac{1}{r_{x}r^{c}_{x}}R^{c}_{x}+(\frac{\alpha^{2}_{0}(x)}{x^{2}}+\frac{\alpha^{2}_{0}(x)}{x})O(c^{-2})=0.\label{RN1.7}
\end{align}
Multiplying \eqref{RN1.7} by $U^{c},$
\begin{align}&\frac{d}{dt}\int^{1}_{0}\alpha_{0}\frac{(U^{c})^{2}}{2}dx\nnm\\
&+\frac{d}{dt}\int^{1}_{0}\frac{\alpha^{2}_{0}(x)}{x}\left(\frac{x}{r}\frac{(r^{c}_{x}+r_{x})}{r^{2}_{x}(r^{c}_{x})^{2}}(R^{c}_{x})^{2}+\frac{1}{(r^{c}_{x})^{2}}\frac{x}{r}\frac{x}{r_{c}}\frac{R^{c}}{x}R^{c}_{x}+\frac{1}{r^{c}_{x}}\frac{x^{2}}{r^{2}}\frac{(r+r^{c})}{r^{c}}\frac{x}{r^{c}}(\frac{R^{c})^{2}}{x^{2}}\right)\nnm\\
&\leq C(K)\int^{1}_{0}\alpha_{0}(x)\left((V^{c})^{2}+(U^{c})^{2}+(R^{c})^{2}\right)dx\nnm\\
&+C(K)\int^{1}_{0}\frac{\alpha^{2}_{0}}{x}\left((R^{c}_{x})^{2}+(\frac{R^{c}_{x}}{x^{2}})^{2}\right)dx+O(c^{-4}).\label{NRL1}
\end{align}
Similar to \eqref{JQ1}, there exists a positive constant $T^{*}_{0}$  such that $0<T^{*}_{0}\leq\min\{T_{c_{0}},T^{*}_{0}\},$ and for any $0<t<T^{*}_{0}~$
\begin{align}\frac{1}{1+\varepsilon_{T^{*}_{0}}}\leq\frac{x}{r},\frac{x}{r^{c}}\leq\frac{1}{1-\varepsilon_{T^{*}_{0}}},~1-\varepsilon_{T^{*}_{0}}\leq r_{x},~r^{c}_{x}\leq 1+\varepsilon_{T^{*}_{0}}.\nnm \end{align}
For small $T^{*}_{0}$ and large enough $c,$
\begin{align}&\frac{x}{r}\frac{(r^{c}_{x}+r_{x})}{r^{2}_{x}(r^{c}_{x})^{2}}(R^{c}_{x})^{2}+\frac{1}{(r^{c}_{x})^{2}}\frac{x}{r}\frac{x}{r_{c}}\frac{R^{c}}{x}R^{c}_{x}+\frac{1}{r^{c}_{x}}\frac{x^{2}}{r^{2}}\frac{(r+r^{c})}{r^{c}}\frac{x}{r^{c}}(\frac{R^{c})^{2}}{x^{2}}\nnm\\
&\geq\left(\frac{r_{x}}{r^{c}_{x}r^{2}_{x}}-2\frac{\varepsilon_{T^{*}_{0}}}{(1-\varepsilon_{T^{*}_{0}})(1-\varepsilon_{T^{*}_{0}})}\right)(R^{c}_{x})^{2}+\frac{1}{2}\left(\frac{x}{r^{c}}\frac{x^{2}}{(r^{c})^{2}}-\frac{x}{r^{c}}\frac{\varepsilon_{T^{*}_{0}}}{(1-\varepsilon_{T^{*}_{0}})(1-\varepsilon_{T^{*}_{0}})}\right)\frac{R^{2}}{x^{2}}\nnm\\
&\geq C(K)\left((R^{c}_{x})^{2}+\frac{(R^{c})^{2}}{x^{2}}\right).\label{NRL2}
\end{align}
Similarly, subtracting  $\eqref{LCE}_{2}$ from \eqref{C-25} ,

\begin{equation}\int^{1}_{0}\alpha_{0}(x)(V^{c})^{2}dx\leq C(K)\int^{1}_{0}\alpha_{0}(x)\left((U^{c})^{2}+(R^{c})^{2}\right)dx+O(c^{-4}).\nnm\end{equation}
which in combination with \eqref{NRL1} and \eqref{NRL2} yields
\begin{align}\int^{1}_{0}\alpha_{0}(x)\left[(V^{c})^{2}+(U^{c})^{2}+(W^{c})^{2}\right]dx +\int^{1}_{0}\frac{\alpha^{2}_{0}(x)}{x}\left((\frac{R^{c}}{x})^{2}+(R^{c})^{2}\right)dx\leq O(c^{-4})\nnm\end{align}
Then, using the weighted embedding estimates in \eqref{WS2}-\eqref{WS3} and $H^{1}(0,1)\hookrightarrow C^{0}[0,1]$, we can similarly prove \eqref{rate-non}.\end{proof}

\section{Results for the cases $\gamma\neq2$}\label{sec-8}
This section is to generalize our result to the more general case of $\gamma\neq2.$ We will consider the well-posedness and non-relativistic limit of local smooth solution for the free boundary value problem \eqref{cylind-R21},\eqref{cylind-R24} and \eqref{1.2}. Due to the physical vacuum condition $\eqref{1.2}_{5}$, the value of $\gamma$ confirms the rate of degeneracy near the vacuum boundary $x=1$, but  it will not affect the rate of degeneracy  near the original point $x=0,$ since $\rho_{0}\sim (1-x)^{\frac{1}{\gamma-1}}$ ~as ~$x\rightarrow1.$ In fact, the rate of degeneracy is more strong for the smaller value of $\gamma.$ Thus, we divide $\gamma$ into the two cases $1<\gamma<2$ and $\gamma>2.$ In the spirit of idea in \cite{LuoT}, we can prove the well-posedness and non-relativistic limits of local smooth solutions  by the similar argument to the case for $\gamma=2$ and omit it here, based on the following main equations:
\begin{align}&a^{\gamma}_{11}\alpha_{c}(x)\partial^{2}_{t}u+\partial_{t}a^{\gamma}_{11}\alpha_{c}(x)u_{t}
-\partial_{t}\left(\alpha_{c}(x)a^{\gamma}_{11}\frac{\upsilon^{2}}{r}\right)
-\left[J_{\gamma}\Theta^{\gamma}\frac{\alpha^{2}_{c}(x)}{x}(\gamma\frac{u_{x}}{r^{2}_{x}}+(\gamma-1)\frac{x}{rr_{x}}\frac{u}{x})\right]_{x}\nnm\\
&-\left[\frac{\gamma J_{\gamma}}{c ^{2}r_{x}}\Theta^{\gamma-2}\frac{\alpha^{2}_{c}(x)}{x}(uu_{t}
+\upsilon\upsilon_{t}
+\omega\omega_{t})\right]_{x}
+J_{\gamma}\Theta^{\gamma}\frac{x}{r}\frac{\alpha^{2}_{c}(x)}{x^{2}}\left((\gamma-1)\frac{u_{x}}{r_{x}}
+\gamma\frac{x}{r}\frac{u}{x}\right)\nnm\\
&+\frac{\gamma J_{\gamma}x}{c ^{2}r}\Theta^{\gamma-2}\frac{\alpha^{2}_{c}(x)}{x^{2}}(uu_{t}+\upsilon\upsilon_{t}+\omega\omega_{t})-\alpha_{c}(x)\frac{2-\gamma}{\gamma-1}(\frac{\alpha_{c}(x)}{x})_{x}\Theta^{\gamma}J_{\gamma}\left[\gamma\frac{u_{x}}{r^{2}_{x}}+(\gamma-1)\frac{x}{rr_{x}}\frac{u}{x}\right]\nnm\\
&-\alpha_{c}(x)\frac{2-\gamma}{\gamma-1}(\frac{\alpha_{c}(x)}{x})_{x}J_{\gamma}\Theta^{\gamma-2}\frac{1}{r_{x}}\left[uu_{t}+\upsilon\upsilon_{t}+\omega\omega_{t}\right]+\frac{1}{c^{2}}\left[a_{12}(\frac{x}{r})^{\gamma-1}\frac{1}{r^{\gamma}_{x}}\frac{\alpha^{2}_{c}(x)}{x}(u_{x}+\frac{xr_{x}}{r}\frac{u}{x})u\right]_{t}=0,\nnm\\
&~~\upsilon_{t}+\frac{u}{r}\upsilon-\frac{b^{\gamma}_{11}}{c^{2}}\frac{\alpha_{c}(x)}{x}(u_{x}+\frac{u}{r}r_{x})\upsilon+\frac{b^{\gamma}_{12}}{c^{2}}(u_{t}-\frac{\upsilon^{2}}{r})
u~\upsilon=0,\nnm\\
&~~\omega_{t}-\frac{b^{\gamma}_{21}}{c^{2}}\frac{\alpha_{c}(x)}{x}(u_{x}+\frac{u}{r}r_{x})\omega+\frac{b^{\gamma}_{22}}{c^{2}}(u_{t}-\frac{\upsilon^{2}}{r})
u~\omega=0,\nnm
\end{align}
where

\begin{equation}J_{\gamma}:=(\frac{x}{rr_{x}})^{\gamma-1}\left(1-\frac{1}{c^{2}}(\frac{\rho_{0}}{r_{x}}\frac{x}{r}\Theta)^{\gamma-1}\right)^{\frac{2\gamma-1}{1-\gamma}},~ \alpha_{c}(x)=\left(\frac{\rho_{0}}{(1+\frac{\rho^{\gamma-1}_{0}}{c^{2}})^{1-\gamma}\Theta_{0}}\right)^{\gamma-1}x.\nnm\end{equation}
These equations follow from the similar way as that of the equations in \eqref{C-24}-\eqref{C-25}.
\section*{Acknowledgements}
\qquad Mai's work is partially supported by the National Natural Science Foundation of China
(Nos.11771071, 11601246), and the Program of Higher Level talents of Inner Mongolia University (No.21100-5165105).  Li's work is supported by the National Natural Science Foundation of China
(Nos.11171228, 11231006 and 11225102), and the Importation and Development of HighCaliber Talents Project of Beijing Municipal Institutions (No. CIT\&TCD20140323). P.Marcati is partially supported by INdAM-GNAMPA, PRIN 2015YCJY3A-003 and the EU MCS Network {\it ModCompShock }.

\end{document}